%% file: stochasticHarmonicMapHeatFlow.tex
\documentclass[a4paper, 12pt, reqno,final]{amsart}
\usepackage[foot]{amsaddr} 
\usepackage[utf8x]{inputenc}
\usepackage[T1]{fontenc}
\usepackage[english]{babel}
\usepackage{amsmath, amsfonts, amsthm, amssymb, amscd}
\usepackage[final]{graphicx}
\usepackage[below]{placeins}
\usepackage{multirow}
\usepackage{mathtools}
\usepackage{caption}
\usepackage{subcaption}
\usepackage{float}

\usepackage[hidelinks]{hyperref}
\hypersetup{
  pdfauthor = {Stefan Metzger},
  pdftitle = {SEL}
}

\usepackage{showkeys} 
\usepackage{esint}
\usepackage{dsfont}

\usepackage{tikz}
\usetikzlibrary{patterns}
\usetikzlibrary{shapes.geometric}

\allowdisplaybreaks
\usepackage[]{todonotes}
\usepackage{diagbox}

\newtheorem{theorem}{Theorem}[section]
\newtheorem{lemma}[theorem]{Lemma}
\newtheorem*{lemma*}{Lemma}
\newtheorem{corollary}[theorem]{Corollary}

\newtheorem{remark}[theorem]{Remark}
\newtheorem{definition}[theorem]{Definition}
\numberwithin{equation}{section}



\input{header.tex}

\begin{document}
\title[Numerical approximation of stochastic harmonic map heat flow]{Numerical approximation of probabilistically strong solutions to the stochastic harmonic map heat flow}
\date{\today}
\author[R.~Lasarzik, S.~Metzger, and M.~E.~V.~Reiter]{Robert Lasarzik, Stefan Metzger, and Maximilian E.V. Reiter}
\address[S.~Metzger]{Friedrich--Alexander Universität Erlangen--Nürnberg,~Cauerstraße 11,~91058~Erlangen,~Germany}
\email{stefan.metzger@fau.de}
\address[R.~Lasarzik]{Friedrich--Alexander Universität Erlangen--Nürnberg,~Cauerstraße 11,~91058~Erlangen,~Germany, and Weierstrass Institute for Applied Analysis and Stochastics,~Anton-Wilhem-Amo-Str.~39,~10117~Berlin,~Germany}
\email{robert.lasarzik@wias-berlin.de}
\address[M.E.V.~Reiter]{Technical University of Berlin, Strasse des 17.~Juni~135, 10623 Berlin, Germany}
\email{research@MaximilianReiter.de}
\subjclass[2010]{}



%
\selectlanguage{english}

\allowdisplaybreaks

\begin{abstract}
We study the stochastic harmonic map heat flow into the sphere driven by a Stratonovich noise term that formally preserves the spherical constraint by construction. Starting from a fully discrete, implementable numerical scheme, we prove its convergence to the novel solution concept of probabilistically strong, analytically energy-variational solutions. This yields the first existence result for probabilistically strong solutions in the three-dimensional setting on the prescribed stochastic basis. 
In addition, we establish weak-strong uniqueness and illustrate our theoretical findings with numerical experiments.

\end{abstract}\maketitle

\input{./content.tex}

\input{stochasticEL.bbl}

\end{document}

%% file: header.tex
\makeatletter
\newcommand{\labitem}[2]{%
\def\@itemlabel{\textbf{#1}}
\item
\def\@currentlabel{#1}\label{#2}}
\makeatother

\newcommand{\norm}[1]{\left\|{#1}\right\|}
\newcommand{\abs}[1]{\left|{#1}\right|}

\newcommand{\rkla}[1]{{\left(#1\right)}}
\newcommand{\trkla}[1]{{(#1)}}
\newcommand{\gkla}[1]{{\left\{#1\right\}}}
\newcommand{\tgkla}[1]{{\{#1\}}}
\newcommand{\skla}[1]{{\left\langle#1\right\rangle}}
\newcommand{\ekla}[1]{{\left[#1\right]}}
\newcommand{\tekla}[1]{{[#1]}}
\newcommand{\tabs}[1]{|{#1}|}

\newcommand{\dtau}{\,\mathrm{d}\tau}
\newcommand{\de}{\,\mathrm{d}}
\newcommand{\bs}[1]{\boldsymbol{#1}}

\newcommand{\expect}[1]{\mathds{E}\left[ #1\right]}

\newcommand{\D}{\mathcal{D}}
\newcommand{\iD}{\int_{\D}}

\newcommand{\dx}{\, \mathrm{d}\bs{x}}
\newcommand{\dt}{\, \mathrm{d}t}
\newcommand{\ds}{\, \mathrm{d}s}

\renewcommand{\div}{\operatorname{div}}

\newcommand{\id}{I}

\newcommand{\nn}{^{n+1}}
\newcommand{\nh}{^{n+1/2}}
\newcommand{\no}{^{n}}

\newcommand{\tl}{^{\tau}}
\newcommand{\tp}{^{\tau,+}}
\newcommand{\tm}{^{\tau,-}}
\newcommand{\ti}{^{\tau,\bullet}}
\newcommand{\tpm}{^{\tau,(\pm)}}

\newcommand{\h}{_{h}}

\newcommand{\weakstartop}{{w^*}}

\newcommand{\Uh}{U_{h}}
\newcommand{\Th}{\mathcal{T}_h}

\newcommand{\Ihop}{\mathcal{I}_h}
\newcommand{\Ih}[1]{\Ihop\gkla{#1}}

\newcommand{\Projop}{\mathcal{P}_h}
\newcommand{\Proj}[1]{\Projop\gkla{#1}}

\newcommand{\diam}{\operatorname{diam}}

\newcommand{\g}[1]{\mathfrak{g}_{#1}}
\newcommand{\U}{\mathfrak{U}}

\newcommand{\p}{{\mathfrak{p}}}

\newcommand{\expected}[1]{\mathds{E}\ekla{#1}}

\newcommand{\Ito}{It\^{o}}

\newcommand{\sinc}[1]{\blacktriangle^{#1}\boldsymbol{\xi}^\tau}
\newcommand{\incbeta}[3]{\delta^{#1}_\tau\beta_{#2}^{#3}}

%% file: content.tex
\section{Introduction}
In this article, we are concerned with the harmonic map heat flow into the sphere with Stratonovich noise, 
\begin{subequations}\label{eq:model:StratIntro}
    \begin{align}\label{eq:model:StratIntro:a}
\mathrm{d}\bs{d}-\trkla{\bs{d}\times\Delta\bs{d}}\times\bs{d}\dt&=\Phi\trkla{\bs{d}}\times\circ\mathrm{d}\bs{W} \quad &&\text{ in } \Omega \times \D \times (0,T) 
\intertext{for a vector field $\bs{d}\in \mathds{R}^3$ equipped with Neumann boundary conditions}
    (\bs{n} \cdot \nabla) \bs{d} &= 0 \quad&&\text{ in } \Omega \times \partial \D \times (0,T) 
\end{align}
\end{subequations}
for a domain $\D \subset \mathds{R}^d $ with $d\in\tgkla{1,2,3}$, a time interval $\trkla{0,T}$, and a complete filtered probability space $\trkla{\Omega,\mathcal{A},\tgkla{\mathcal{F}_t}_{t\geq0},\mathds{P}}$.
Here $\bs{W}$ is a cylindrical Wiener process and $\Phi\trkla{\bs{d}}\times$ is a suitably weighted Hilbert--Schmidt operator of cross product type, i.e., the right hand side of \eqref{eq:model:StratIntro:a} can be written as $\Phi\trkla{\bs{d}}\times\circ\mathrm{d}\bs{W}=:\bs{d}\times\circ\mathrm{d}\trkla{\mathcal{Q}^{1/2}\bs{W}}$ with a suitably scaled Hilbert--Schmidt operator $\mathcal{Q}^{1/2}$. The exact assumptions on $\Phi$ ,$\bs{W}$, and $\mathcal{Q}$ can be found in assumptions \ref{item:W1}, \ref{item:W2}, and \ref{item:W3} in Section \ref{sec:notation}.

In the deterministic setting, this equation can be seen as the $L^2$-gradient flow of the Dirichlet energy with the restriction on the sphere, \textit{i.e.,}
$$
\min_{ |\bs u | = 1 \text{ in }\D} \int_{\D}|\nabla \bs u |^2 \mathrm{d} x \,. 
$$
The associated gradient flow, the harmonic map heat flow into the sphere, can be seen as a simplification of the Landau--Lifshitz--Gilbert equation that keeps the main difficulties but also as a simplification of the Ericksen--Leslie equations with vanishing velocity field and simple Dirichlet elastic energy. 
The Landau--Lifshitz--Gilbert equation models the orientation of spin fields in ferromagnetism~\cite{ProhlBook}, whereas the Ericksen--Leslie equations model anisotropic fluids, especially liquid crystals~\cite{leslie}. \\
In the stochastic setting, the deterministic equations are augmented by the additional source term on the right-hand side of \eqref{eq:model:StratIntro:a}, which is driven by a suitable cylindrical Wiener process which we shall introduce in more detail in Section \ref{sec:notation}.
Here, it is important to mention that the noise has to be understood in the sense of Stratonovich, as this interpretation preserves the geometric constraints (cf.~\cite{Hocquet1} and Remark \ref{rem:StratNoise}).

In the deterministic setting, Struwe~\cite{Struwe1} established the existence of global weak solutions for harmonic map heat flow from two-dimensional surfaces, which are smooth apart from finitely many space-time singularities, nowadays these solutions are called Struwe-type solutions. The existence and partial regularity theory was extended to higher-dimensional domains by Chen and Struwe~\cite{Struwe2}.
The deterministic harmonic map heat flow exhibits a delicate uniqueness theory. In two spatial dimensions, Freire \cite{Freire1995} established uniqueness for energy-decreasing weak solutions, whereas Coron \cite{Coron1990} demonstrated nonuniqueness of weak solutions in three dimensions. 

There are several numerical investigations for the deterministic case or of the very similar Landau--Lifshitz--Gilbert equations, see for instance~\cite{BartelsProhl07,BanasProhl10,Alouges06}. 

For the stochastic case, there is an existence result  of martingale solutions for the three dimensional Landau--Lifshitz--Gilbert equation~\cite{BrzezniakGoldysJegaraj2013}  and the convergence of an implementable numerical scheme for this model~\cite{Alouges2014}. Both works consider analytically and probabilistically weak solutions. 
In the two-dimensional case the existence of Struwe-type solutions was shown in~\cite{Hocquet1} that are strong in the probabilistic sense and regular up to finitely many points in time. Also in the two-dimensional case, it was shown in~\cite{Hocquet2} that {prescribing Dirichlet-type boundary conditions
generates a blow up with positive probability. 
The existence of similar Struwe-type solutions could also be extended to a simplified version of the Ericksen--Leslie system in~\cite{Hocquet3}. 
In~\cite{BrzezniakGoldysJegaraj2017} the authors show the existence of probabilistically strong solutions in the case of the domain being a bounded interval and a finite dimensional noise by showing pathwise uniqueness and applying the Yamada--Watanabe theory. 
Recently, Struwe like solutions were also investigated for the two-dimensional Landau--Lifshitz--Gilbert equations including path-wise solvability and bubbling~\cite{GoldysJiaoMelcher2025}.

A semi-discrete scheme for the Landau--Lifshitz equation was proposed and analyzed in~\cite{Alouges2014}, whereas a fully discrete Finite Element scheme was proposed and shown to converge in~\cite{Banas2013} to a probabilistically weak and analytically weak solution. In~\cite{Dunst2014} T.~Dunst manages to prove convergences rates for the simpler case of finite dimensional noise. 

In this article, we investigate the convergence of a fully discrete implementable structure inheriting scheme, similar to the one in~\cite{Banas2013}. 
We will prove convergence of solutions to that scheme to a stochastically strong analytically energy-variational solution.
This direction contrasts the usual way results for stochastic evolution equations are shown. 
Typically, after establishing finite moment bounds through appropriate \textit{a priori} estimates,  the Skorokhod--Jakubowski representation theorem~\cite{Jakubowski_2} is invoked in order to 
deduce almost sure convergence of the approximations to a limit. 
The downside of this procedure is the required change of the underlying probability space.  Therefore, such a procedure leads to probabilistically weak solutions, so-called martingale solutions. 
Often, the only way to infer the existence of probabilistically strong solutions on the original probability space is via pathwise uniqueness and an application of the Yamada--Watanabe theory (see e.g.~\cite{Yamada1971,Rockner2008,theewis2025arxiv} and the references therein).
This can be done in a constructive way by employing the Gyöngy--Krylov characterization of convergence in probability (cf.~\cite{GyongyKrylov1996, BreitFeireislHofmanova}). 

In the current article, we propose a different approach. Our main results are the following new contributions: 
\begin{itemize}
    \item We introduce the novel concept of probabilistically strong analytically energy-variational solutions to the stochastic harmonic map heat flow.
    \item We establish existence of such solutions via the convergence of a fully discrete implementable numerical scheme without changing the stochastic basis relying only on weak convergence results. 
    \item We show that the energy-variational solutions fulfills a pathwise weak--strong uniqueness result, meaning that if there is a solutions emanating from the same initial datum on the same underlying stochastic basis with additional regularity, then all energy-variational solutions coincide with this more regular solution. 
\end{itemize}
\subsection{Probabilistically strong analytically energy-variational solutions}
In contrast to previous works, we do not show convergence towards martingale solutions but probabilistically strong solutions, i.e.~the stochastic basis is not part of the solution concept but the given data. 
To achieve this, we have to abstain from using the Skorokhod--Jakubowski theorem to establish strong convergence for the approximating sequence. 
However, we are still able to identify the limit of the approximating sequence by generalizing the analytical formulation of the solution. 
Relaxing our solution concept to energy-variational solutions, we can pass to the limit using only weak convergence of the approximating sequence and weak sequential lower semicontinuity. In the deterministic case, the resulting concept would certainly be more general and thus weaker than classical weak solutions, but in the stochastic case, this is not the case anymore. 
The probabilistically weak, analytically weak solutions are not  comparable anymore with the probabilistically strong analytically energy-variational solutions.  

The idea of energy-variational solutions was developed in the realm of fluid dynamics~\cite{Lasa24} and conservation laws~\cite{EiterLasa24}, but was afterwards identified as a general scheme in physically motivated partial differential evolution equations~\cite{AgostiRoccaLasa24,Las25}. 
We consider an operator-valued ODE in a deterministic setting of the form 
$$ \partial_t \bs U + \bs A (\bs U) = 0 \quad \in \mathds{Y}^* \text{ with } \bs U = \bs U_0 \text{ in }\mathds V^*$$
for certain Banach spaces $ \mathds Y  \hookrightarrow \mathds V$. For physical systems, there usually exists an energy $ \mathcal E : \mathds V^* \to [0,\infty ] $ and a dissipation potential $ \Psi : \mathds V^* \to (-\infty, +\infty]$ such that a certain energy inequality is fulfilled, \textit{i.e.,} $ \displaystyle\mathcal E (\bs U) \Big|_s^t + \int_s^t \Psi(\bs U) \de \tau \leq 0 $
 for all $s<t\in [0,T]$. In order to derive the energy-variational formulation for such a system, we introduce an auxiliary variable $ E \in \mathrm{BV}([0,T])$, which is an upper bound of the energy, \textit{i.e.,} $E \geq \mathcal E(\bs U)$ for all $ t\in [0,T]$ and the energy-variational inequality is fulfilled, \textit{i.e.,}
\begin{align}\label{eq:envarintro}
    \left [ E - \langle \bs U , \Phi \rangle_{\mathds V}\right ] \Big |_s^t + \int_s^t \Psi (\bs U) + \langle \partial_t \Phi , \bs U\rangle _{\mathds V} - \langle A (\bs U ) , \Phi \rangle_{\mathds Y} + \mathcal{K}(\Phi) [ \mathcal E(\bs U) - E ] \de \tau \leq 0 
\end{align}
for all suitable test functions $ \Phi \in C^1([0,T];\mathds Y)$ and all $0\leq s \leq t \leq T$. 
For a large class of partial differential equations, it was observed that there exists
a regularity weight $\mathcal K$ such that the formulation~\eqref{eq:envarintro} will be 
stable under weak convergence. This is often connected to the convexity of the mapping
$$\bs U \mapsto \Psi(\bs U) -\langle A(\bs U) , \Phi\rangle_{\mathds Y} + \mathcal{K}(\Phi) \mathcal E(\bs U)
$$
for any $ \Phi \in \mathds Y$. 
This structure was not only exploited for (in)compressible fluid dynamics 
but also for general conservation laws~\cite{EiterLasa24}, viscoelastic fluids~\cite{AgostiRoccaLasa24}, 
polyconvex elasticity~\cite{Las25}, and nematic liquid crystals~\cite{Max}. 
In the context of PDEs with stochastic forcing, the concept of energy-variational solutions was recently 
applied to the stochastic incompressible Euler equations. 
There, the first author was able to improve the existence results from probabilistically weak to probabilistically strong solutions via a vanishing viscosity limit from the stochastic incompressible Navier--Stokes equations~\cite{GessLasarzik26}. 

The main difficulty of the system considered in this work is that establishing enough compactness to verify that the limit process satisfies the norm constraint $\abs{\bs{d}}=1$ seems to be out of reach --- despite the fact that fully discrete, approximate solution satisfies this constraint in every node of the mesh.
As a consequence, we are not able to establish strong convergence of the energy.

\subsection{Convergence of the fully discrete structure inheriting scheme}
The convergence of solutions to the proposed scheme to martingale solutions was already shown in~\cite{BanasBrzezniakNeklyudovProhl2014}. 
In order to prove the convergence on the initial probability space, we repeatedly make use of the energy-variational  
structure, which in turn heavily relies on the underlying energy-dissipation mechanism. 
Therefore, in order to prove convergence towards energy-variational solutions, establishing the limit of the associated energy-inequality is crucial.
This was not done in the previous works as only the weak formulation was investigated. 
We note that usually the existence of weak solutions without the energy-inequality does not suffice to prove weak strong uniqueness of solutions. 
Establishing the convergence of the energy-inequality, in particular convergence of the higher-order terms, requires us to verify additional estimates that allow for a careful handling of the discrete error terms. 
In comparison to martingale solutions, the main drawback of our solution concept is that we are only able to show that $|\bs d |\leq 1$ in the limit. 
Yet, assuming that the limit process satisfies $ |\bs d|=1$, we actually can show that the energy-variational solutions simplifies to the usual analytically weak solution (see Remark \ref{rem:variationalVSweak} and Theorem \ref{thm:weaksol}).

\subsection{Pathwise weak-strong uniqueness of energy-variational solutions}
To the best of the authors' knowledge, there is no weak-strong uniqueness result available for the stochastic harmonic map heat flow. 
As we are able to show the existence of probabilistically strong solutions, all solutions are defined on the same stochastic basis and it suffices to show pathwise weak-strong uniqueness.
Therefore, this result is stronger than usual results in probability theory as it states that the solutions coincide a.s. not only in law. 
 The strategy for our proof is very similar to the usual relative energy idea. However, the usual relative energy, which coincides with the $H^1$-norm for the considered case, 
 has to be augmented by additional nonnegative error-terms. Namely, we are not able to show that the energy nor the norm of the approximate sequence converge. 
 Therefore, the difference between the limit of the approximate energy $E$ and the energy of the limit $ \mathcal E(\bs d)$ as well as the deviation of the limit of the norm of the discrete solutions (which is $1$) from the norm of the limit $ |\bs d|^2 $ need to be included. 
 The resulting relative energy is thus given by 
\begin{align}
    \mathcal{R}(\bs d , E | \tilde{\bs d} ) = 
E - \mathcal E(\bs d ) + \frac{1}{2}\int_{\D} 1 - |\bs d|^2 \dx + \frac{1}{2} \| \bs d - \tilde{\bs d} \| _{H^1(\D)}\,,
\end{align}
where $\tilde{\bs d} $ denotes a more regular function.
With this technique, we are able to prove pathwise weak-strong uniqueness for the concept of energy-variational solutions (cf.~Definition \ref{def:envar2}). \\

Finally, we present some numerical simulations that show the influence of the stochastic fluctuations on the dynamics. Starting from an instable stationary point, we observe how the stochastic driving force moves the solution to a local or global minimum, depending on the actual stochastic realization. 

The article is structured as follows: 
We introduce the relevant notations and several preliminary results in Section~\ref{sec:notation}. The the discrete scheme and the main results of this work are stated in Section~\ref{sec:main}. The proofs of these results can be found in the subsequent sections: After deriving the relevant \textit{a priori} estimates in Section~\ref{sec:apriori}, we show the convergence towards the probabilistically strong analytically energy-variational solution in Section~\ref{sec:limenergy}, whereas in Section~\ref{sec:weakstrong}, the weak-strong uniqueness is shown to hold for the established solutions. Finally, the results are illustrated via some numerical experiments in Section~\ref{sec:numerics}.

\section{Notation and preliminaries}\label{sec:notation}
The spatial domain $\D\subset\mathds{R}^d$ ($d\in\tgkla{1,2,3}$) is assumed to be bounded, convex, and polygonal (or polyhedral, respectively).
We denote the space of $k$-times weakly differentiable  functions with weak derivatives in $L^p\trkla{\D}$ by $W^{k,p}\trkla{\D}$.
For $p=2$, we denote the Hilbert spaces $W^{k,2}\trkla{\D}$ by $H^k\trkla{\D}$ and the $L^2\trkla{\D}$-inner product by $\trkla{.,.}$. 
We will denote the duality pairing between $H^1\rkla{\D}$ and its dual space $\trkla{H^1\trkla{\D}}^*$ by $\skla{.,.}$.
This allows us to define a generalized mean value of a function $v\in\trkla{H^1\trkla{\D}}^*$ via $\trkla{v}_{\D}:=\tabs{\D}^{-1}\skla{v,1}$.
We equip $\trkla{H^1\trkla{\D}}^*$ with the inner product
\begin{align}
    \rkla{a,b}_*:=\rkla{\nabla\trkla{-\Delta_N^{-1}}\rkla{a-\trkla{a}_{\D}},\nabla\trkla{-\Delta_N^{-1}}\trkla{b-\trkla{b}_{\D}}} + \trkla{a}_{\D}\trkla{b}_{\D}
\end{align}
and the induced norm $\norm{a}_{*}^2:=\trkla{a,a}_*$, which makes it a Hilbert space.
In an abuse of notation, we denote by 
\begin{align}
-\Delta_N : H^1(\D)^3 \to (H^1(\D)^3)^* 
\quad \langle - \Delta_N \bs b , \bs a \rangle = \int_{\D} \nabla \bs b : \nabla \bs a \dx \quad \forall \bs a \in H^1(\D)^3\,
\end{align}
the Neumann--Laplace operator. 
Given a Banach space $B$ and a set $I$, the symbol $L^p\trkla{I;B}$ ($p\in\tekla{1,\infty}$) stands for the Bochner space of strongly measurable $L^p$-integrable functions on $I$ with values in $X$.
If $B$ is only separable (but not reflexive), we 
denote the dual space of $L^{p/\trkla{p-1}}\trkla{I;B}$ by $L^p_{\weakstartop}\trkla{I;B^*}$.
Further, the symbol $C^{k,\alpha}\trkla{I;B}$ with $k\in\mathds{N}_0$ and $\alpha\in(0,1]$ denotes the space of $k$-times continuously differentiable functions from $I$ to $B$ whose $k$-th derivatives are Hölder continuous with Hölder exponent $\alpha$. For $\alpha\in\trkla{0,1}$ and $\beta\in\tekla{1,\infty}$, we define the Nikolskii space $N^{\alpha,\beta}$ for a time interval $\trkla{0,T}$ and a Banach space $X$ via
\begin{align}
N^{\alpha,\beta}\trkla{0,T;X}:=\gkla{f\in L^\beta\trkla{0,T;X}\,:\,\sup_{k>0}k^{-\alpha}\norm{f\trkla{\cdot+k}-f\trkla{\cdot}}_{L^\beta\trkla{0,T-k;X}}<\infty}\,.
\end{align}

With respect to time, we use an equidistant discretization, i.e.~we assume that
\begin{itemize}
\labitem{\textbf{(T)}}{item:T} the time interval $I:=\tekla{0,T}$ is subdivided into $N$ equidistant intervals $I_n$ given by nodes $\trkla{t^n}_{n=0,\ldots,N}$ with $t^0=0$, $t^N=T$, and $t\nn-t\no=\tau=\tfrac{T}{N}$.
\end{itemize}
Our spatial discretization is based on partitions $\Th$ of $\D$ which depend on a parameter $h>0$ and satisfy the following assumption:
\begin{itemize}
\labitem{\textbf{(S)}}{item:S} The family $\tgkla{\Th}_{h>0}$ is a quasi-uniform family of partitions of $\D$ into disjoint, open simplices $K$ which satisfy
\begin{align*}
\overline{\D}\equiv\bigcup_{K\in\Th}\overline{K}&&\text{with~~}\max_{K\in\Th}\diam\trkla{K}\leq h\,.
\end{align*}
\end{itemize}
Thus, the space $\Uh$ of continuous piecewise linear finite element functions on $\Th$ is spanned by functions $\tgkla{\chi_{h,k}}_{k=1,\ldots,\dim\Uh}$ forming a dual basis to the vertices $\tgkla{\bs{x}_{h,k}}_{k=1,\ldots,\dim\Uh}$ of $\Th$.
For each $h>0$, the nodal interpolation operator $\Ihop$ mapping $C^0\trkla{\overline{\D}}$ to $\Uh$ is defined by
\begin{align}
\Ih{a}:=\sum_{k=1}^{\dim{\Uh}}a\trkla{\bs{x}_{h,k}}\chi_{h,k}\,.
\end{align}
This interpolation operator allows us to define equivalent versions of the standard $L^p\trkla{\D}$-norms. 
In particular, the inequalities
\begin{align}\label{eq:normequivalence}
\norm{a\h}_{L^p\trkla{\D}}\leq\rkla{\iD\Ih{\abs{a\h}^p}\dx}^{1/p}\leq C\norm{a\h}_{L^p\trkla{\D}}
\end{align}
hold true for $p\in[1,\infty)$ and $a\h\in\Uh$ with  a constant $C>0$ independent of $h$. 
This estimate follows from Jensen's inequality and standard inverse estimates. A detailed proof can be found, e.g., in \cite[Lemma 3.2.11]{SieberDiss2021} or \cite[Lemma~3.3.1]{MaxDiss}.
For $p=2$, this norm is induced by the lumped masses inner product $\trkla{a,b}\h:=\iD\Ih{ab}\dx$ ($\forall a,b\in\Uh$).
For future reference, we recall the following error estimates for $\Ihop$ which were proven e.g.~in \cite[Lemma 2.1]{Metzger2020}:
\begin{lemma}\label{lem:interpolationerror}
Let $\mathcal{T}\h$ satisfy assumption \ref{item:S} and let $p\in[1,\infty)$ and $1\leq q\leq \infty$. Then, for $q^*=\tfrac{q}{q-1}$ (with the obvious modifications for $q=1$ and $q=\infty$) the estimates
\begin{subequations}
\begin{align}
\norm{\trkla{I-\Ihop}\tgkla{f\h g\h}}_{L^p\trkla{\mathcal{D}}}\leq&\, Ch^2\norm{\nabla f\h}_{L^{pq}\trkla{\mathcal{D}}}\norm{\nabla g\h}_{L^{pq^*}\trkla{\mathcal{D}}}\,,\\
\norm{\trkla{I-\Ihop}\tgkla{f\h g\h}}_{W^{1,p}\trkla{\mathcal{D}}}\leq&\, Ch\norm{\nabla f\h}_{L^{pq}\trkla{\mathcal{D}}}\norm{\nabla g\h}_{L^{pq^*}\trkla{\mathcal{D}}}\,
\end{align}
\end{subequations}
hold true for all $f\h,g\h\in\Uh$.
\end{lemma}
Further error estimates for the application of $\Ihop$ to smooth functions can be found e.g.~in \cite[Theorem 4.4.20]{BrennerScott}.
In a slight misuse of notation, we shall also write $\Ih{\bs{d}}$ (and $\trkla{.,.}\h$) when we apply $\Ihop$ component wise to a vector field $\bs{d}$.
Further, we shall define a discrete counterpart of the Laplacian via
\begin{align}\label{eq:defDiscLap}
    \rkla{\nabla a\h,\nabla b\h}=\rkla{-\Delta\h a\h,b\h}\h&&\forall a\h,b\h\in\Uh\,.
\end{align}
The (vector-valued) cylindrical Wiener process $\bs{W}$ governing the right-hand side of \eqref{eq:model:StratIntro} is defined on a filtered probability space $\trkla{\Omega,\mathcal{A},\tgkla{\mathcal{F}_t}_{t\geq0},\mathds{P}}$ with a complete right-continuous filtration $\tgkla{\mathcal{F}_t}_{t\geq0}$.
Denoting the cartesian basis of $\mathds{R}^3$ by $\trkla{\bs{e}_i}_{i=1}^3$, we shall assume that
\begin{itemize}
\labitem{\textbf{(W1)}}{item:W1} each component of the vector-valued cylindrical Wiener process $\bs{W}$ has a representation
\begin{align}\label{eq:representationW}
\bs{W}_i\trkla{t}=\sum_{l\in\mathds{Z}}\g{li}\cdot\bs{e}_i\beta_l^i\trkla{t}\,,
\end{align}
where $\trkla{\beta_l^i}_{l\in\mathds{Z}}^{i=1,2,3}$ is a family of mutually independent Brownian motions and $\trkla{\g{li}}_{l\in\mathds{Z}}^{i=1,2,3}$ are a complete orthonormal system in a separable Hilbert space $\U^3$ over $\D$ consisting of smooth function satisfying $\g{li}\cdot\bs{e}_j=0$ for $i\neq j$.
\end{itemize}
Next, we specify the coloring of the stochastic source term by assuming that
\begin{itemize}
\labitem{\textbf{(W2)}}{item:W2} $\mathcal{Q}\,:\,\U^3\rightarrow\trkla{L^2\trkla{\D}}^3$ is a trace class operator with the representation
\begin{align}
\mathcal{Q}\bs{g}:=\sum_{i=1}^3\sum_{l\in\mathds{Z}}\tabs{\lambda_{li}}^2\trkla{\bs{g},\g{li}}\g{li}&&\text{for~}\bs{g}\in \U^3\,,
\end{align}
where $\trkla{\cdot,\cdot}$ denotes the $L^2$-inner product, $\trkla{\lambda_{li}}_{l\in\mathds{Z}}^{i=1,2,3}$ are given positive real numbers and $\trkla{\g{li}}_{l\in\mathds{Z}}^{i=1,2,3}$ are the basis functions introduced in \ref{item:W1}.
\end{itemize}
Consequently, $\displaystyle\mathcal{Q}^{1/2}$, which is given by $\displaystyle\mathcal{Q}^{1/2}\bs{g}:=\sum_{i=1}^3\sum_{l\in\mathds{Z}}\lambda_{li}\trkla{\bs{g},\g{li}}\g{li}$, is a Hilbert--Schmidt operator from $\U^3$ to $\trkla{L^2\trkla{\D}}^3$.
Furthermore, we shall assume that $\mathcal{Q}$ satisfies the following bound:
\begin{itemize}
\labitem{\textbf{(W3)}}{item:W3} There exists a positive constant $\widehat{C}$ such that
\begin{align}
\sum_{i=1}^3\sum_{l\in\mathds{Z}}\tabs{\lambda_{li}}^2\rkla{\norm{\g{li}}_{W^{1,\infty}\trkla{\D}}^2+\norm{\g{li}}_{H^2\trkla{\D}}^2}\leq\widehat{C}\,.
\end{align}
\end{itemize}
This allows us to define the operator $\Phi\trkla{\cdot}\times$ which maps $\trkla{L^\infty\trkla{\D}}^3$-valued stochastic processes $\bs{d}$ into the space of Hilbert--Schmidt operators from $\U^3$ to $\trkla{L^2\trkla{\D}}^3$, via
\begin{align}\label{eq:def:PHI}
\Phi\trkla{\bs{d}}\times\bs{g}:=\sum_{i=1}^3\sum_{l\in\mathds{Z}}\trkla{\bs{g},\g{li}}\lambda_{li}\bs{d}\times\g{li}\,.
\end{align}

In our discrete scheme, we shall use for given $h>0$ only a finite number of modes.
Thus, we introduce an $h$-dependent family $\trkla{\mathds{Z}\h}_{h>0}$ of finite subsets of $\mathds{Z}$ satisfying
\begin{itemize}
\labitem{\textbf{(Z)}}{item:Z} $\displaystyle\lim_{h\rightarrow0}\mathds{Z}\h=\mathds{Z}$
\end{itemize}
and approximate $\Phi$ by
\begin{align}\label{eq:def:PHIH}
\Phi\h\trkla{\bs{d}}\times\bs{g}:=\sum_{i=1}^3\sum_{l\in\mathds{Z}\h}\trkla{\bs{g},\g{li}}\lambda_{li}\bs{d}\times\g{li}\,.
\end{align} 
In Section \ref{sec:limenergy} we will deduce a pathwise energy estimate satisfied by the solution process $\bs{d}$ to \eqref{eq:model:StratIntro}.
As the regularity obtained for $\bs{d}$ is not sufficient to apply \Ito's formula directly, we shall verify the energy estimate on the discrete level and identify its limit.
To achieve this, we will require a certain compatibility between our spatial and temporal discretization parameters which state that $\tau$ vanishes slightly faster than $h^2$.
In particular we assume that
\begin{itemize}
    \labitem{\textbf(C)}{item:C} there exists $0<\widehat{\epsilon}<1/2$ and $C_*>0$ such that 
    \begin{align}
        \frac{\tau^{1-2\widehat{\epsilon}}}{h^2}\leq C_*&&\text{as~}\trkla{h,\tau}\rightarrow0\,.
    \end{align}
\end{itemize}
Concerning the initial data $\bs{d}_0$, we consider a unit vector field with $H^2\trkla{\D}$-regularity, i.e.~we assume that
\begin{itemize}
\labitem{\textbf{(I)}}{item:I} $\bs{d}_0\in L^\infty\trkla{\Omega;\trkla{H^2\trkla{\D}}^3}$ is $\mathcal{F}_0$-measurable and satisfies $\tabs{\bs{d}_0\trkla{\bs{x}}}=1$ $\mathds{P}$-almost surely for all $\bs{x}\in \D$.
Further, our discrete initial data is given as $\bs{d}\h^0:=\Ih{\bs{d}_0}\in\trkla{\Uh}^3$.
\end{itemize}

\begin{remark}\label{rem:StratNoise}
    We introduced our model in \eqref{eq:model:StratIntro} using Stratonovich noise. 
    Using representation \eqref{eq:representationW} of the cylindrical Wiener process and \eqref{eq:def:PHI}, we can reformulate \eqref{eq:model:StratIntro} as the equivalent \Ito-equation
    \begin{subequations}\label{eq:Ito}
    \begin{align}\label{eq:Ito:d}
\mathrm{d}\bs{d}-\trkla{\bs{d}\times\Delta\bs{d}}\times\bs{d}\dt= \tfrac12\sum_{i=1}^3\sum_{l\in\mathds{Z}}\tabs{\lambda_{li}}^2\rkla{\bs{d}\times\g{li}}\times\g{li}\mathrm{d}t+ \sum_{i=1}^3\sum_{l\in\mathds{Z}}\bs{d}\times\lambda_{li}\g{li}\mathrm{d}\beta_l^i
\end{align}
in $\Omega\times\trkla{0,T}\times\D$, subjected to the boundary conditions
\begin{align}
    \trkla{\bs{n}\cdot\nabla}\bs{d}=0&&\text{on~}\Omega\times\trkla{0,T}\times\partial\D\,.
\end{align}
\end{subequations}
If we assume that all components of the cylindrical vector-valued Wiener process can be expressed with the same basis functions $\rkla{\widehat{\mathfrak{g}}_l}_{l\in\mathds{Z}}$, i.e.~$\g{li}=\widehat{\mathfrak{g}}_l\bs{e}_i$ for $i,j=1,2,3$, and $\lambda_{li}=\lambda_{lj}=:\widehat{\lambda}_l$ ($i,j=1,2,3$), we recover the formulation used in \cite{Hocquet3} for the \Ito-correction term
\begin{align*}
    \tfrac12\sum_{i=1}^3\sum_{l\in\mathds{Z}}\tabs{\lambda_{li}}^2\rkla{\bs{d}\times\g{li}}\times\g{li} =\tfrac12\sum_{l\in\mathds{Z}}\tabs{\widehat{\lambda_{l}}}^2\tabs{\widehat{\mathfrak{g}}_l}^2\sum_{i=1}^3 ( (\bs e_i\cdot \bs d ) \bs e_i - \bs d ) 
     = -\sum_{l\in\mathds{Z}}\tabs{\widehat{\lambda_{l}}}^2\tabs{\widehat{\mathfrak{g}}_l}^2\bs{d}\,.
\end{align*}
The stochastic term on the right-hand side of \eqref{eq:model:StratIntro} preserves the length constraint:
On a formal level, a pointwise application of \Ito's formula to $F\trkla{t,\bs{d}}:=\tfrac12\bs{d}\cdot\bs{d}$ yields
\begin{align}\label{eq:formallength}
\begin{split}
\tfrac12\abs{\bs{d}\trkla{\bs{x},t}}^2=&\,\tfrac12\abs{\bs{d}\trkla{\bs{x},0}}^2+\sum_{i=1}^3\sum_{l\in\mathds{Z}}\int_0^t\rkla{\trkla{\lambda_{li}\g{li}}\times\bs{d}}\cdot\bs{d}\,\mathrm{d}\beta_l^i+\int_0^t\rkla{\trkla{\bs{d}\times\Delta\bs{d}}\times\bs{d}}\cdot\bs{d}\ds\\
&\,+\tfrac12\sum_{i=1}^3\sum_{l\in\mathds{Z}}\tabs{\lambda_{li}}^2\int_0^t\rkla{\trkla{\bs{d}\times\g{li}}\times\g{li}}\cdot\bs{d}\ds +\tfrac12\sum_{i=1}^3\sum_{l\in\mathds{Z}}\abs{\bs{d}\times\lambda_{li}\g{li}}^2\ds\\
=&\,\tfrac12\abs{\bs{d}\trkla{\bs{x},0}}^2\,.
\end{split}
\end{align}
Another formal application of \Ito's formula shows that solutions to \eqref{eq:model:StratIntro} (and \eqref{eq:Ito}, respectively) should satisfy the energy estimate of the form
\begin{align}
\tfrac12&\int_{\mathcal{D}}\abs{\nabla\bs{d}\trkla{t}}^2\dx-\tfrac12\int_{\mathcal{D}}\abs{\nabla\bs{d}\trkla{0}}^2\dx\nonumber \\=&\, \int_0^t\rkla{\nabla\bs{d},\nabla \ekla{\trkla{\bs{d}\times\Delta\bs{d}}\times\bs{d}}}\ds+\sum_{i=1}^3\sum_{l\in\mathds{Z}}\int_0^t\rkla{\nabla\bs{d},\nabla \trkla{\bs{d}\times\lambda_{li}\g{li}}}\mathrm{d}\beta_l^i\nonumber\\
&\,+\tfrac12\sum_{i=1}^3\sum_{l\in\mathds{Z}}\tabs{\lambda_{li}}^2\int_0^t \rkla{\nabla\bs{d},\nabla\ekla{\trkla{\bs{d}\times\g{li}}\times\g{li}}}+\rkla{\nabla\ekla{\bs{d}\times\g{li}},\nabla\ekla{\bs{d}\times\g{li}}}\ds 
 \nonumber\\
=&\,
-\int_0^t\rkla{\bs{d}\times\Delta\bs{d},\bs{d}\times\Delta\bs{d}}\ds+\sum_{i=1}^3\sum_{l\in\mathds{Z}}\int_0^t\rkla{\bs{d}\times\Delta\bs{d},\lambda_{li}\g{li}}\mathrm{d}\beta_l^i\label{eq:formalIto}\\
&\,+\tfrac12\sum_{i=1}^3\sum_{l\in\mathds{Z}}\abs{\lambda_{li}}^2\int_0^t\sum_{j=1}^d\int_{\D} |\bs d|^2| \partial_j \g{li}|^2 - |\bs d \cdot \partial_j \g{li} |^2 -  (\bs d \times \partial_j \bs d) \cdot( \g{li}\times \partial _j \g{li}  )
\dx\ds\nonumber
 %
\end{align}
When deriving such an estimate rigorously in Section \ref{sec:limenergy}, we can only rely on weak sequential lower semi-continuity of quadratic terms and weak convergence results for the director $ \bs d $ and the tangential derivative $ \bs d \times \nabla \bs d = \bs q$.
Thus, the later formulation of the right-hand side is better suited for our approach. The equivalence of these formulations is verified in Lemma \ref{lem}.
\end{remark}

The following lemma will allow to rewrite the terms appearing in the energy inequality \eqref{eq:formalIto} due to the application of \Ito's formula.

\begin{lemma}\label{lem}
       Let $ \bs d \in L^\infty(0,T;H^1(\D)^3)$ and $ \g{li} \in W^{1,\infty}(\D)^3$ as well as $\lambda_{li} \in \mathds{R}$.

    Then it holds 
 \[
\begin{aligned}
&
\rkla{\nabla\bs{d},\nabla\ekla{\trkla{\bs{d}\times\g{li}}\times\g{li}}}
+
\rkla{\nabla\ekla{\bs{d}\times\g{li}},\nabla\ekla{\bs{d}\times\g{li}}}
\\
&{}=\, 
\sum_{j=1}^d\int_{\D} |\bs d|^2| \partial_j \g{li}|^2 - |\bs d \cdot \partial_j \g{li} |^2 -  (\bs d \times \partial_j \bs d) \cdot( \g{li}\times \partial _j \g{li}  )
\dx
\,.
\end{aligned}
\]
\end{lemma}
\begin{proof}
    We expand the first term via the relation $ (a\times b)\times b = (b\cdot a) b - |b|^2 a $ into 
\begin{align}
\nabla\bs{d}:\nabla\ekla{\trkla{\bs{d}\times\g{li}}\times\g{li}} = {}&\nabla\bs{d}:\nabla\ekla{\trkla{\g{li}\cdot \bs{d}}\g{li} - |\g{li}|^2 \bs{d}}\notag
    \\
    ={}& \sum_{j=1}^3 \big[(\partial_j\bs d \cdot \g{li}) (\partial_j \g{li} \cdot \bs d)+ | \partial_j \bs d \cdot \g{li}|^2 + \rkla{\partial_j\bs{d}\cdot \partial_j\g{li}}\rkla{ \bs d \cdot \g{li}} \notag\\&\qquad- \rkla{\partial_j \bs d \cdot \bs d }\rkla{ 2\partial_j \g{li} \cdot \g{li}}  - | \g{li}|^2 | \partial_j \bs d |^2\Big] \,.\label{eq:crossproduct1}
    \end{align}
For the second term, we observe via the cross product property \begin{equation}\label{crossproduct}    ( a\times b) \cdot ( c \times d) = (a\cdot c) (b\cdot d) - ( a \cdot d )( b \cdot c)  
\end{equation} that 
\begin{align}\label{eq:crossproduct2}
\begin{split}
        | \nabla (\bs{d} \times \g{li})|^2 ={}& \sum_{j=1}^3 | \partial_j \bs d \times \g{li} + \bs d \times \partial_j \g{li}|^2 \\
    ={}& \sum_{j=1}^d \Big[| \g{li} \times \partial_j \bs d |^2 + | \bs d \times \partial_j \g{li} |^2 + 2 ( \partial_j \bs d \times  \g{li} )\rkla{ \bs d \times \partial_j \g{li} }\Big]
    \\
    ={}& \sum_{j=1}^d \Big[ | \g{li}|^2 | \partial_j \bs d|^2 - ( \g{li}\cdot \partial_j \bs d)^2 + | \bs d|^2 |\partial_j \g{li}|^2 - ( \bs d \cdot \partial _j \g{li})^2 \\ &\qquad+ 2 ( \partial_j \bs d \cdot \bs d ) ( \g{li} \cdot \partial _j \g{li}) - 2( \partial_j \bs d \cdot \partial _j\g{li} )( \g{li}\cdot \bs d) \Big] \,.
\end{split}
\end{align}
Summing the remaining terms of~\eqref{eq:crossproduct1} and \eqref{eq:crossproduct2}, we find again using~\eqref{crossproduct} that
    \begin{align*}
\nabla\bs{d}:&\nabla\ekla{\trkla{\bs{d}\times\g{li}}\times\g{li}}+ | \nabla (\bs{d} \times \g{li})|^2
\\
={}&\sum_{j=1}^3 \big[(\partial_j\bs d \cdot \g{li}) (\partial_j \g{li} \cdot \bs d)+ | \partial_j \bs d \cdot \g{li}|^2 + \rkla{\partial_j\bs{d}\cdot \partial_j\g{li}}\rkla{ \bs d \cdot \g{li}} 
\\&\quad
- 2\rkla{\partial_j \bs d \cdot \bs d }\rkla{ \partial_j \g{li} \cdot \g{li}} 
- | \g{li}|^2 | \partial_j \bs d |^2
+| \g{li}|^2 | \partial_j \bs d|^2 - ( \g{li}\cdot \partial_j \bs d)^2 + | \bs d|^2 |\partial_j \g{li}|^2 \\ 
&\quad- ( \bs d \cdot \partial _j \g{li})^2 + 2 ( \partial_j \bs d \cdot \bs d ) ( \g{li} \cdot \partial _j \g{li}) - 2( \partial_j \bs d \cdot \partial _j\g{li} )( \g{li}\cdot \bs d)
\Big] 
\\
={}& \sum_{j=1}^3 \big[(\partial_j\bs d \cdot \g{li}) (\partial_j \g{li} \cdot \bs d)- \rkla{\partial_j\bs{d}\cdot \partial_j\g{li}}\rkla{ \bs d \cdot \g{li}} 
+|\bs d |^2  |\partial_j \g{li}|^2 - ( \bs d \cdot \partial _j \g{li})^2 \Big] 
\\
={}& \sum_{j=1}^3 \big[ |\bs d|^2|\partial_j \g{li}|^2 - ( \bs d \cdot \partial _j \g{li})^2 - ( \partial _j \bs d \times \bs d )\cdot(  \partial _j \g{li} \times \g{li} ) \Big] \,,
    \end{align*}
where we used again~\eqref{crossproduct} in the last equality.   This proves the assertion.
\end{proof}

We define the Skorokhod space $\mathfrak{D}([0, T])$ as the space of the real-valued càdlàg functions defined on $[0, T]$. More precisely, $\Psi$ belongs to the space $\mathfrak{D}([0, T])$ if it is right-continuous and has left-hand limits:
\begin{enumerate}
    \item for $0<t \leq T, \Psi(t-)=\lim _{s \uparrow t} \Psi(s)$ exists,
    \item for $0 \leq t<T, \Psi(t+)=\lim _{s \downarrow t} \Psi(s)$ exists and $\Psi(t+)=\Psi(t)$.
\end{enumerate}
It is well known that càdlàg functions are bounded, measurable, and have at most countably many discontinuities. 
The following Lemma can be found in~\cite[Lemma~2.1]{GessLasarzik26}.
\begin{lemma}\label{lem:invar}
Let $f\in L^p(\Omega;L^1(0,T))$, $g\in L^p_{w^*}(\Omega;L^\infty(0,T )) $,  $g_0\in L^p(\Omega)$ for $p\geq 1$, and let $h$ be $(\mathcal F_t)$-progressively measurable and satisfy
\begin{equation}\label{eq:h-H1}
\mathds E\left[
\left(
\int_0^T
\|h(t)\|_{L_2(\mathfrak{U}^3;\mathds R)}^2\,\de t
\right)^{1/2}
\right]<\infty.
\end{equation}
Then the following three statements are equivalent:
\begin{enumerate}
\item[i)]
The inequality 
\begin{multline}
-\expect{\psi(\omega)\int_0^T\partial_t\phi(\tau) g(\tau,\omega) \de \tau}  \\+\expect{\psi(\omega)\left( \int_0^T \phi(\tau) f(\tau,\omega) \de \tau+\int_0^T \phi (\tau) h(\tau,\omega) \de W(\tau) - \phi(0)g_0 (\omega)\right)}\leq 0 
\label{ineq1}
\end{multline}
holds for all $\phi \in{C}^1_c ([0,T))$ with $\phi \geq 0$ and $\psi \in L^\infty(\Omega)$ with $\psi \geq 0$ $\mathds P$-almost surely.
\item[ii)]
The inequality
\begin{equation}
    g(t,\omega) -g(s,\omega) + \int_s^t f(\tau,\omega) \de \tau + \int_s^t h(\tau,\omega   ) \de W(\tau)\leq 0 
    \label{ineq2}
\end{equation}
holds for a.e.~$s< t\in (0,T)$ and $\mathds{P}$-almost surely,
including $s=0$ if we replace $g(0)$ with $g_0$.

\item[iii)]
There exists a representative of $g$ that can be identified with a function in $\mathfrak{D}([0,T])$  such that
\begin{equation}
    g(t) -g(s-) + \int_s^t f(\tau) \de \tau + \int_s^t h(\tau ) \de W(\tau)\leq 0 \,
    \label{ineq.pw}
\end{equation}
for all $s\leq t\in[0,T]$,
where we set $g(0-)\coloneqq g_0$ and $\mathds{P}$-almost surely.
In particular, it holds $g(0)\leq g_0$ and $g(t)\leq \lim_{s\nearrow t}g(s)$ for all $t\in (0,T]$ $\mathds{P}$-almost surely.
\end{enumerate}
\end{lemma}

In Section \ref{sec:limenergy}, we will pass to the limit $\rkla{h,\tau}\searrow\rkla{0,0}$ with families of fully discrete random variables and derive the existence of energy-variational solutions (cf.~Definition \ref{def:envar2}) from solutions to our fully discrete scheme \eqref{eq:discscheme}.
For this reason, we define time-interpolants of a time-discrete function $a^n$, $n=0,\ldots,N$, and introduce some time-index free notation as follows
\begin{subequations}\label{eq:deftimeinterpolation}
\begin{align}
a\tl\trkla{.,t}:=&\,\tfrac{t-t\no}{\tau}a\nn\trkla{.}+\tfrac{t\nn-t}{\tau} a\no\trkla{.}&&t\in\tekla{t\no,t\nn},\,n\geq0\,,\\
a\tm\trkla{.,t}:=&\,a\no\trkla{.}&&t\in[t\no,t\nn),\,n\geq0\,,\\
a\tp\trkla{.,t}:=&\,a\nn\trkla{.}&&t\in[t\no,t\nn),\,n\geq0\,,\\
a\ti\trkla{.,t}:=&\,\tfrac12\tekla{a\tp\trkla{.,t}+a\tm\trkla{.,t}}\,.
\end{align}
\end{subequations}
If a statement is valid for all time-interpolations defined in \eqref{eq:deftimeinterpolation}, we use the abbreviation $a\tpm$.
\section{The Discrete Scheme and main results}\label{sec:main}

\begin{definition}[Test process]\label{def:testprocess} 
     We call the triple $ ( \varphi_0 , A ,B)$  a test process, if $\varphi_0 \in L^2(\Omega;L^2(\D)^3)$ 
     and $(A,B)\in L^2(\Omega; L^1(0,T;(H^{1}(\D)^3)^*)) \times L^4(\Omega\times\trkla{0,T} ;L_2(\mathfrak{U}^3;(H^1(\D)^3)^*))$ are $(\mathcal{F}_t)$-progressively measurable  functions 
     such that $\bs{\varphi} \in L^\p(\Omega ; \mathcal{C}([0,T]; L^4(\D)^3)) $ ($\p>2$) given by 
     $$\mathrm d \bs\varphi = A\mathrm {d}t + B \mathrm{d}\bs{W} \,. $$ 
\end{definition}
The control of the $\p$-th moments of $\bs{\varphi}$ will be relevant in the error estimates stated in Corollary \ref{cor:diffS}.
Writing $B  \mathrm{d}\bs{W} = \sum_{i=1}^3\sum_{l\in\mathds Z} B_{li}(t)\g{li}\mathrm d \beta_l^i$, the assumptions on $B$ translate to 
\begin{align}
\expected{\int_0^T \norm{B}_{L_2\trkla{\mathfrak{U};\trkla{H^1\trkla{\D}^3}^*}}^4\ds}=\expected{\int_0^T\rkla{\sum_{i=1}^3\sum_{l\in\mathds{Z}}\norm{B\g{li}}_*^2}^2\ds}<\infty\,.
\end{align}
\begin{remark}[Formulation of solutions to SPDEs]
    Before we define the energy-variational solution concept, we argue how we will use the above defined solution test processes. 
    If we consider a solution $ \bs u$ to an SPDE given by 
    $$ \mathrm d \bs u = U \dt + V \mathrm d \bs W  \quad \text{with } \bs u (0) = \bs u_0 \,$$
    where $ U \in L^1(\Omega; \mathds Y^*) $ and $ V \in L_2(H,\mathds Y^*)$ for some Hilbert space $H$ and a Banach space $\mathds Y$.
    Usually generalized solutions are defined by multiplying with test functions that are only space dependent. 
    This leads to a concept like 
     $$
        \langle \bs u , \varphi \rangle_{\mathds Y} \Big|_0^t = \int_0^t \langle U , \varphi \rangle_{\mathds Y} \dt + \int_0^t \langle V , \varphi \rangle_{\mathds Y} \mathrm{d} \bs W \,,
        $$
for all $ \varphi \in \mathds Y$.
In contrast to this, we do not want to test by test functions but rather test processes defined in Definition~\ref{def:testprocess}. 
Applying \Ito's formula to the product of $ \bs u$ and a test process $ \bs \varphi$, we find 
\begin{align*}
   \langle \bs u , \bs \varphi \rangle \Big |_0^t ={}& \int_0^t \langle U , \bs \varphi \rangle \dt + \int_0^t \langle V , \bs \varphi \rangle \mathrm{d} \bs W
   \\
& + \int_0^t \langle A , \bs u \rangle \dt + \int_0^t \langle B , \bs u \rangle \mathrm d\bs W + \int_0^t \langle B , V \rangle \de t \,
\end{align*}
for all test processes $ \bs \varphi $ in the sense of Definition~\ref{def:testprocess}.
This idea has already been applied in other articles cf.~\cite{GessLasarzik26,HZZEuler}. 
We combine this idea with the energy-variational formulation in order to arrive at the next definition. 
\end{remark}

\begin{definition}[Probabilistically strong energy-variational solutions]\label{def:envar2}
    For a cylindrical Wiener process defined on a filtered probability space $\trkla{\Omega,\mathcal{A},\tgkla{\mathcal{F}_t}_{t\geq0},\mathds{P}}$ and a given $\mathcal{F}_0$-measurable random variable $\bs{d}_0$ and the operator $\Phi\trkla{.}\times$ defined in \eqref{eq:def:PHI}, we call 
    $\bs  d \in L^\infty( \Omega\times (0,T)\times \D)$ a probabilistically strong energy-variational solution, iff $\bs d$ is $\mathcal{F}_t$-progressively measurable such that for every $p\in [1,\infty)$ it holds 
    \begin{align}
        \bs d \in L^{2p}_{\weakstartop}(\Omega;L^\infty(0,T;H^1(\D)^3))
    \end{align}
     with 
      $ |\bs d| \leq 1 $ a.e.~in $\D \times (0,T)$ and $\mathds{P} $-a.s., such that $ \bs d \in C_{w}([0,T];H^1(\D)^3)$ $\mathds{P} $-a.s. with $ \bs d(0)=\bs d_0$, and
    \begin{align}\label{eq:d-dW}
        \bs{d} - \int_0^{\cdot}\Phi\trkla{\bs{d}} \times \mathrm{d} \bs{W} &\in L^2(\Omega; H^1(0,T;L^2(\D)^3)) 
        \,.
    \end{align}
    Moreover, there exists $ \bs q  \in L^{2p}_{w^*}(\Omega;L^\infty(0,T; L^2( \D; \mathds R^{3\times d} ))) $ such that its divergence $ \div \bs q \in L^{2p}(\Omega; L^2((0,T)\times \D; \mathds R^{3} )) $ and $ E \in L^p(\Omega;\mathfrak{D}([0,T]))$ such that $ E  \geq \mathcal E(\bs d) $ a.e.~in $\D \times (0,T)$ and $ \mathds P$-a.s. with $ E(0-)=\mathcal E(\bs d_0)$, and the energy-variational inequality is fulfilled
    \begin{equation}
        \label{eq:envarineq2}
    \begin{aligned}
    E\Big|_{s-}^t& + \int_s^t\int_{\D} |\div \bs q |^2 \dx\dtau 
   -\sum_{i=1}^3\sum_{l\in\mathds{Z}}\int_s^t\trkla{\bs q;\lambda_{li}\nabla \g{li}}\mathrm{d}\beta_l^i\\
   & -  (\bs d , \bs \varphi)  \Big|_s^t + \int_s^t   (\div \bs q , \bs d \times \bs \varphi) + 
   \frac{1}{4} \int_{\D}|\bs\varphi|^2 \left[  |\bs d|^2 - 1\right] \dx \dtau\\
 \leq {} & \tfrac12\sum_{i=1}^3\sum_{l\in\mathds{Z}}\tabs{\lambda_{li}}^2\int_s^t \| \nabla \g{li}\|_{L^2\trkla{\D}}^2 - \|\bs d \cdot \nabla \g{li} \|_{L^2\trkla{\D}}^2 -  (\bs q ,  \g{li}\times \nabla  \g{li}  )
\dtau
\\
&{- \int_s^t \langle A , \bs d \rangle  \dt  - \sum_{i=1}^3
\sum_{l\in\mathds{Z}}\int_s^t \langle B_{li}\g{li},  ( \bs d \times \lambda_{li}\g{l})\rangle \dtau
-\sum_{i=1}^3\sum_{l\in\mathds{Z}}\int_s^t \langle B_{li}\g{li}, \bs d \rangle \mathrm{d}\beta_l^i}
 \\
& -\frac{1}{2}\sum_{i=1}^3\sum_{l\in \mathds{Z}} \int_s^t\rkla{ |\lambda_{li}|^2 (\bs d \times\g{li} )\times \g{li} , \bs \varphi } \dtau -\sum_{i=1}^3\sum_{l\in \mathds{Z}} \int_s^t \rkla{ \lambda_{li}\bs d \times \g{li},\bs \varphi } \mathrm{d}\beta_l^i
\end{aligned}
    \end{equation}
for all test processes in the sense of Definition~\ref{def:testprocess}
\\
Moreover, the term $\bs q$ fulfills the inequality 
\begin{equation}\label{eq:formulationq}
      \int_{\D}   ( \bs d \times \nabla  \bs d - \bs q ) : \bs \Psi \dx \leq   \left[ \frac{1}{2} \int_{\D} |\bs \Psi|^2[1-|\bs d|^2 ]\dx +E- \mathcal{E}(\bs d)\right]  
\end{equation}
 for all $\bs \Psi \in L^4(\D; \mathds R^{3\times d})$ 
$\mathds{P}$-a.s.. 
\end{definition}
\begin{remark}[Connection to weak solutions]\label{rem:variationalVSweak}
    The main difficulty in the analysis of the convergence of the numerical scheme lies in the lacking strong convergence of the approximation of the director $\bs d$. Especially, we are not able to prove that in the limit $ |\bs d|=1$ holds true pointwise a.e.~in $\D\times (0,T)$ and $\mathds{P}$-a.s.~in $\Omega$.  Due to the lower semicontinuity of the norm, we only know that $ |\bs d| \leq 1 $ pointwise a.e.~in $\D\times (0,T)$ and $\mathds{P}$-a.s.~in $\Omega$, which is not sufficient to infer $ |\bs d| =1$. 

    On the other hand, in case that $ |\bs d| =1$ holds true, the solution in the sense of Definition~\ref{def:envar2} simplifies to a standard probabilistically strong, analytically weak solution with energy inequality. 
\end{remark}

\begin{theorem}[Weak solution]\label{thm:weaksol}
    Let $\bs d$ be an energy-variational solution in the sense of Definition~\eqref{def:envar2} with $ |\bs d|=1$ a.e.~in $\D\times (0,T)$ $\mathds P$-a.s.. 

    Then $\bs d$ is a probabilistically strong analytically weak solution, \textit{i.e.,} 
   $$ \begin{aligned}
           &\int_{\D} (\bs d  \cdot  \zeta) \dx \Big|_s^t + \int_s^t \iD  \trkla{\bs d \times(\bs d \times \Delta \bs d  ) }\cdot    \zeta\dx \dtau\\
= & \frac{1}{2}\sum_{i=1}^3\sum_{l\in \mathds{Z}} \int_s^t\iD |\lambda_{li}|^2 \trkla{(\bs d \times\g{li} )\times \g{li} }\cdot  \zeta \dx \dtau +\sum_{i=1}^3\sum_{l\in \mathds{Z}} \int_s^t \int_{\D} \lambda_{li}(\bs d \times \g{li})\cdot \zeta \dx \mathrm{d}\beta_l^i
    \end{aligned}
    $$ holds
    for all $  \zeta\in L^2(\D)^3$ $\mathds P$-a.s.. 
\end{theorem}
\begin{proof}    
    Indeed, from~\eqref{eq:formulationq}, we find  
    $$
    \int_{\D}    ( \bs d \times \nabla  \bs d- \bs q ) : \bs \Psi   \dx \leq   \left[ E- \frac{1}{2}\norm{\nabla \bs d}_{L^2(\D)}^2   \right]  \,.
    $$
    Choosing $ \bs \Psi = \alpha \bs \Theta$ for $\bs \Theta \in L^4(\D;\mathds R^{d\times 3})$, multiplying by $\frac{1}{\alpha}$ and letting $ \alpha \to \infty$, we find that $\int_{\D}  ( \bs d \times \nabla  \bs d- \bs q) : \bs \Theta \dx \leq 0 $, which allows to identify $ \div \bs q $ with $ \bs d \times \Delta \bs d $ due to the linearity of the space of test functions. 

    In~\eqref{eq:envarineq2}, the last term on the left-hand side vanishes. In a similar scaling argument, we choose the test process to be 
    $(\bs{\varphi},A,B) = \alpha (\tilde{\zeta}_0,0,0)  $, multiply the resulting inequality by $\frac{1}{\alpha}$ and send $\alpha \to \infty$, we find the standard weak formulation 
    $$
    \begin{aligned}
           &\int_{\D} (\bs d  \cdot  \zeta) \dx \Big|_s^t + \int_s^t \iD  \bs d \times(\bs d \times \Delta \bs d  ) \cdot    \zeta\dx \\
\geq & \frac{1}{2}\sum_{i=1}^3\sum_{l\in \mathds{Z}} \int_s^t\iD |\lambda_{li}|^2 \trkla{(\bs d \times\g{li} )\times \g{li} }\cdot  \zeta \dx \dtau +\sum_{i=1}^3\sum_{l\in \mathds{Z}} \int_s^t \int_{\D} \lambda_{li}(\bs d \times \g{li})\cdot \zeta \dx \mathrm{d}\beta_l^i
    \end{aligned}
    $$
    for all $  \zeta \in L^2(\D)^3$. Again this formulation is equivalent to the inequality replaced by an equality due to the linearity of the test functions. 
    \end{proof}

Using the notation introduced in Section \ref{sec:notation}, we formulate our numerical scheme as follows:\\
For a given $\trkla{\Uh}^3$-valued random variable $\bs{d}\h\no$, find a $\trkla{\Uh}^3$-valued random variable $\bs{d}\h\nn$ such that pathwise

\begin{align}\label{eq:discscheme}
\rkla{\bs{d}\h\nn-\bs{d}\h\no,\bs{\psi}\h}\h -\tau\rkla{\bs{d}\h\nh\times\Delta\h\bs{d}\h\nn,\bs{d}\h\nh\times\bs{\psi}\h}\h =\rkla{\bs{\psi}\h,\bs{d}\h\nh\times\sinc{n+1}}\h
\end{align}
for $\bs{\psi}\h\in\trkla{\Uh}^3$ with $\bs{d}\h\nh:=\tfrac12\trkla{\bs{d}\h\nn+\bs{d}\h\no}$ and
\begin{align}
\sinc{n+1}:=\sum_{i=1}^3\sum_{l\in\mathds{Z}\h}\lambda_{li}\g{li}\incbeta{n+1}{l}{i}:=\sum_{i=1}^3\sum_{l\in\mathds{Z}\h}\lambda_{li}\g{li}\int_{t\no}^{t\nn}\mathrm{d}\beta_l^i\,.
\end{align}
In this scheme, we approximated the Stratonovich integral $\displaystyle\int_{t^n}^{t^{n+1}} \Phi_h\trkla{\bs{d}}\times\circ\mathrm{d}\bs{W}$ by
\begin{align}
 \Phi_h\trkla{\bs{d}\h\nh}\times\int_{t^n}^{t^{n+1}}\mathrm{d}\bs{W} = \bs{d}\h\nh\times\sinc{n+1}\,,
\end{align}
i.e.~by an approximation of the solution at $t\no+\tfrac12\tau$.
This approach guarantees that if a solution exists, it satisfies the length constraint at the nodes of $\Th$.
Choosing $\bs{\psi}\h=\bs{d}\h\nh\trkla{\bs{x}_i}\chi_{h,i}$, we obtain a discrete version of the length constraint \eqref{eq:formallength}:
\begin{align}\label{eq:lengthconstraint}
\tfrac12\abs{\bs{d}\h\nn\trkla{\bs{x}_i}}^2=\tfrac12\abs{\bs{d}\h\no\trkla{\bs{x}_i}}^2=\tfrac12\abs{\bs{d}\h^0\trkla{\bs{x}_i}}^2
\end{align}
$\mathds{P}$-almost surely for all $i=1,\ldots,\dim\Uh$ and all $n\in\tgkla{0,\ldots,N-1}$.
As shown in \cite[Lemma 4.1]{Banas2013} or \cite[Lemma 2.22]{BanasBrzezniakNeklyudovProhl2014} there exists an $\mathcal{F}_{t\nn}$-measurable process satisfying \eqref{eq:discscheme}:
\begin{lemma}\label{lem:existence}
    Let assumptions \ref{item:T}, \ref{item:S}, \ref{item:I}, \ref{item:W1}, \ref{item:W2}, and \ref{item:W3} hold true. Then there exists a sequence $\tgkla{\bs{d}\h\no}_{n=0}^N$ which solves \eqref{eq:discscheme}. Furthermore, the map $\bs{d}\h\no\,:\Omega\rightarrow \trkla{\Uh}^3\subset\trkla{H^1\trkla{\D}}^3$ is $\mathcal{F}_{t\no}$-measurable for each $0\leq n\leq N$.
\end{lemma}
We are now in the position to state our main results:
\begin{theorem}[Existence of probabilistically strong analytically energy-variational solutions]\label{thm:mainResult}
    Let $\mathcal{D}$ be bounded, convex, and polygonal (or polyhedral, respectively), let $\bs{W}$ be a cylindrical Wiener process which satisfies \ref{item:W1} and is defined on a complete stochastic basis $\left(\Omega, \mathcal{A},\left\{\mathcal{F}_t\right\}_{t \geq 0}, \mathds{P}\right)$ with a probability measure $\mathds{P}$ and a right-continuous filtration $\left\{\mathcal{F}_t\right\}$, and let $\mathcal{Q}$ be a trace class operator satisfying \ref{item:W2} and \ref{item:W3}. 
    Let $\bs d_0$ be an initial condition fulfilling~\ref{item:I}.
    Then, subsequences of fully discrete solutions to \eqref{eq:discscheme} converge weakly$^*$ towards a probabilistically strong analytically energy-variational solution in the sense of Definition~\ref{def:envar2}.

    In particular, we have
    \begin{subequations}
        \begin{align}
            \bs{d}\h\tpm\stackrel{*}{\rightharpoonup}&\,\bs{d}&&\text{in~}L_{\weakstartop}^p\trkla{\Omega;L^\infty\trkla{0,T;H^1\trkla{\D}^3}}\,,\\
            \bs{d}\h\ti\times\nabla\bs{d}\h\ti\stackrel{*}{\rightharpoonup}&\,\bs{q}&&\text{in~}L_{\weakstartop}^p\trkla{\Omega;L^\infty(0,T;L^2\trkla{\D;\mathds{R}^{3\times d}})}\,,\\
            \tfrac{1}{2}\norm{\nabla\bs{d}\h\tp}_{L^2\trkla{\D}}^2\stackrel{*}{\rightharpoonup}&\,E&&\text{in~}L^p_{w^*}\trkla{\Omega;L^\infty\trkla{{0,T}}}
        \end{align}
        for any $p\in[1,\infty)$. Moreover, for any $p\in [1,\infty)$ and $q\in\trkla{1,\infty}$, it holds  
        \begin{align}
\begin{split}
\bs{d}\in& L^\infty\trkla{\Omega\times\trkla{0,T}\times\D}^3 \cap L^{p}_{\weakstartop}\trkla{\Omega;L^\infty\trkla{0,T;H^1\trkla{\D}^3}}\\
&\,\cap L^{2q}\trkla{\Omega;C^{0,\trkla{q-1}/\trkla{2q}}\rkla{\tekla{0,T};L^2\trkla{\D}^3}}\,.
\end{split}
\end{align}
    \end{subequations}
    \end{theorem}
    \begin{proof}[Sketch of the proof]
        The proof of Theorem \ref{thm:mainResult} can be found in Sections \ref{sec:apriori}--\ref{sec:limenergy}.
        For the reader's convenience, we provide a short overview:\\
        In Section \ref{sec:apriori}, we establish $\trkla{h,\tau}$-independent regularity results, which allow us to identify weakly (or weakly$^*$) converging subsequences.
        In Section \ref{sec:limenergy}, we construct a discrete version of the energy estimate and pass to the limit to obtain the existence of and the convergence towards energy-variational solutions (cf.~Theorem \ref{thm:convergence}).
        Thereby, we have to rewrite our time-discrete approximation of the Stratonovich integral as a time-discrete approximation of an \Ito~integral.
        The error terms arising in this discrete transformation formula are discussed in Section \ref{sec:trafoerror}.
        The limit of the discrete \Ito-Stratonovich-correction terms is discussed in Section \ref{sec:correctionterms}.
        In Section \ref{sec:timederivative}, we establish a discrete version of the regularity result \eqref{eq:d-dW}. This is important, as our concept of energy-variational solutions does not include an evolution equation for $\bs{d}$ that can be used to derived additional improved regularity results after passing to the limit.
        These results are then used in Section \ref{sec:conv} to compute the limit $\rkla{h,\tau}\rightarrow\trkla{0,0}$ and establish the main result.
    \end{proof}
Further, we will show that if a sufficiently regular (analytically) strong solution exists, it coincides with the energy-variational solutions established in Theorem \ref{thm:mainResult}:   
\begin{definition}[Strong solution]\label{def:strongSolution}
Let $\trkla{\Omega,\mathcal{A},\tgkla{\mathcal{F}_t}_{t\geq0},\mathds{P}}$ be a stochastic basis with a complete right-continuous filtration, $\widetilde{\bs{d}}_0$ be a given $\mathcal{F}_0$-measurable random variable, and let $\bs{W}$ be a vector-valued cylindrical Wiener process adapted to $\mathcal{F}_t$ and $\mathcal{Q}^{1/2}$ be a Hilbert--Schmidt operator such that \ref{item:W1}, \ref{item:W2}, and \ref{item:W3} hold true.
A random variable $\widetilde{\bs{d}}$ and a stopping time $\mathfrak{t}$ are called a local analytically and stochastically strong solution to \eqref{eq:Ito} provided
\begin{itemize}
    \item the quantity $\widetilde{\bs{d}}
    \in C\trkla{\tekla{0,T};W^{2,\infty}\trkla{\D}^3}$ and  the quantity $\widetilde{\bs{d}}\times\Delta\widetilde{\bs{d}}\in C\trkla{\tekla{0,T};W^{1,\infty}\trkla{\D}^3}$ $\mathds{P}$-a.s.~are $\mathcal{F}_t$-adapted and
    \begin{align}\label{eq:regStrongSol}
        \expected{\sup_{t\in\tekla{0,T}}\norm{\widetilde{\bs{d}}}_{W^{2,\infty}\trkla{\D}}^p +\sup_{t\in\tekla{0,T}}\norm{\widetilde{\bs{d}}\times\Delta\widetilde{\bs{d}}}_{W^{1,\infty}\trkla{\D}}^p}<\infty&&\text{for all~}1\leq p<\infty\,,
    \end{align} 
    \item the identity
    \begin{align}
    \begin{split}
        \widetilde{\bs{d}}\trkla{t}=&\,\widetilde{\bs{d}}_0+\int_0^{t\wedge\mathfrak{t}}\trkla{\widetilde{\bs{d}}\times\Delta\widetilde{\bs{d}}}\times\widetilde{\bs{d}}\ds \\
        &\,+\int_0^{t\wedge\mathfrak{t}}\tfrac12\sum_{i=1}^3\sum_{l\in\mathds{Z}}\tabs{\lambda_{li}}^2\trkla{\widetilde{\bs{d}}\times\g{li}}\times\g{li}\ds + \int_0^{t\wedge\mathfrak{t}}\widetilde{\bs{d}}\times\mathrm{d}\mathcal{Q}^{1/2}\bs{W}
        \end{split}
    \end{align}
    holds $\mathds{P}$-almost surely for all $t\in\trkla{0,T}$ a.e.~in $\D$ and $\trkla{\bs{n}\cdot\nabla}\widetilde{\bs{d}}=0$ on $\partial\D$ in the sense of the trace $\mathds{P}$-almost surely for all times $t\in\trkla{0,T}$.
\end{itemize}

\end{definition}
The regularity assumptions \eqref{eq:regStrongSol} stated for (analytically) strong solutions also imply $\trkla{\widetilde{\bs{d}}\times\Delta\widetilde{\bs{d}}}\times\widetilde{\bs{d}}\in L^p\trkla{\Omega;L^\infty\trkla{0,T;W^{1,\infty}\trkla{\D}^3}}$.
If such solutions exist, they are unique and coincide with the energy-variational solutions: 
\begin{theorem}\label{thm:weakstronguniqueness}
Let $\bs{d}$ be a probabilistically strong analytically energy-variational   solution as specified in Theorem \ref{thm:mainResult} and let $\widetilde{\bs{d}}$ and $\mathfrak{t}$ be an analytically and stochastically strong solution as defined in Definition \ref{def:strongSolution} with $\bs{d}_0=\widetilde{\bs{d}}_0$ $\mathds{P}$-almost surely. Then $\bs{d}=\widetilde{\bs{d}}$ a.e.~in $\D$ for all $t\in[0,\mathfrak{t}]$ and $\mathds{P}$-almost surely. 
\end{theorem}
The proof Theorem \ref{thm:weakstronguniqueness} can be found in Section \ref{sec:weakstrong}.

\section{Uniform regularity results for fully discrete solutions\label{sec:apriori}}
In this section, we derive uniform estimates for the fully discrete solutions to \eqref{eq:discscheme}.
Our first estimate will provide bounds for the energy:
\begin{lemma}\label{lem:energy}
Let assumptions \ref{item:T}, \ref{item:S}, \ref{item:I}, \ref{item:W1}, \ref{item:W2}, and \ref{item:W3} hold true. Then for $p\in[0,\infty)$ the estimate
\begin{multline}
\expected{\max_{0\leq\tilde{N}\leq N}\norm{\bs{d}\h^{\tilde{N}}}_{H^1\trkla{\D}}^{2p}}+\expected{\rkla{\sum_{n=0}^{N-1}\norm{\bs{d}\h\nn-\bs{d}\h\no}_{H^1\trkla{\D}}^2}^p}\\
+\expected{\rkla{\sum_{n=0}^{N-1}\tau\norm{\bs{d}\h\nh\times\Delta\h\bs{d}\h\nn}\h^2}^p}\leq C\,
\end{multline}
holds true with a constant $C$ depending on $p$ but not on $h$ nor $\tau$.
\end{lemma}
\begin{proof}
To prove the claim, we first need an estimate on $\norm{\bs{d}\h\nn-\bs{d}\h\no}\h$.
Thus, we choose $\bs{\psi}\h=\bs{d}\h\nn-\bs{d}\h\no$ in \eqref{eq:discscheme} and deduce after absorption
\begin{align}\label{eq:proof:energy:0}
\begin{split}
\norm{\bs{d}\h\nn-\bs{d}\h\no}\h^2\leq&\, C\tau^2\norm{\bs{d}\h\nh}_{L^\infty\trkla{\D}}^2\norm{\bs{d}\h\nh\times\Delta\h\bs{d}\h\nn}\h^2+ C\norm{\bs{d}\h\nh\times\sinc{n+1}}\h^2\\
\leq&\,C\tau^2\norm{\bs{d}\h\nh\times\Delta\h\bs{d}\h\nn}\h^2 +C\norm{\sinc{n+1}}\h^2\,.
\end{split}
\end{align}
Here, we used the $L^\infty$-bound on $\bs{d}\h\nh$ established in \eqref{eq:lengthconstraint}.
Next, we recover the discrete energy.
Choosing $\bs{\psi}\h=-\Delta\h\bs{d}\h\nn$, we obtain
\begin{align}
\begin{split}\label{eq:proof:energy:0.5}
\tfrac12\norm{\nabla\bs{d}\h\nn}_{L^2\trkla{\D}}^2&\,+\tfrac12\norm{\nabla\bs{d}\h\nn-\nabla\bs{d}\h\no}_{L^2\trkla{\D}}^2-\tfrac12\norm{\nabla\bs{d}\h\no}_{L^2\trkla{\D}}^2+\tau\norm{\bs{d}\h\nh\times\Delta\h\bs{d}\h\nn}\h^2\\
=&\,\rkla{-\Delta\h\bs{d}\h\nn,\bs{d}\h\nh\times\sinc{n+1}}\h\,.
\end{split}
\end{align}
To derive an estimate for the right-hand side, we adapt the lines of \cite{Banas2013}:
\begin{align}
\begin{split}\label{eq:proof:energy:1}
&\rkla{-\Delta\h\bs{d}\h\nn,\bs{d}\h\nh\times\sinc{n+1}}\h=\rkla{\nabla\Ih{\bs{d}\h\nh\times\sinc{n+1}},\nabla\bs{d}\h\nn}\\
&=-\rkla{\nabla\trkla{\id-\Ihop}\gkla{\bs{d}\h\nh\times\Ih{\sinc{n+1}}},\nabla\bs{d}\h\nn}\\
&\qquad+\rkla{\nabla\trkla{\bs{d}\h\nh\times\Ih{\sinc{n+1}}},\nabla\bs{d}\h\nn}\\
&=-\rkla{\nabla\trkla{\id-\Ihop}\gkla{\bs{d}\h\nh\times\Ih{\sinc{n+1}}},\nabla\bs{d}\h\nn-\nabla\bs{d}\h\no}\\
&\qquad-\rkla{\nabla\trkla{\id-\Ihop}\gkla{\bs{d}\h\no\times\Ih{\sinc{n+1}}},\nabla\bs{d}\h\no}\\
&\qquad-\tfrac12\rkla{\nabla\trkla{\id-\Ihop}\gkla{\trkla{\bs{d}\h\nn-\bs{d}\h\no}\times\Ih{\sinc{n+1}}},\nabla\bs{d}\h\no}\\
&\qquad+\rkla{\nabla\trkla{\bs{d}\h\nh\times\Ih{\sinc{n+1}}},\nabla\bs{d}\h\nn}=:\mathfrak{r}_1+\mathfrak{r}_2+\mathfrak{r}_3+\mathfrak{r}_4\,.
\end{split}
\end{align}
For the first term, we deduce
\begin{align}
\begin{split}
\mathfrak{r}_1\leq&\,C\norm{\nabla\trkla{\id-\Ihop}\gkla{\bs{d}\h\nh\times\Ih{\sinc{n+1}}}}_{L^2\trkla{\D}}^2+\alpha\norm{\nabla\bs{d}\h\nn-\nabla\bs{d}\h\no}_{L^2\trkla{\D}}^2\,.
\end{split}
\end{align}
While the second term can be absorbed on the left-hand side, we need to control the first term. 
Using Lemma \ref{lem:interpolationerror}, a standard inverse estimate, and the constraint $\norm{\bs{d}\h\nh}_{L^\infty\trkla{\D}}\leq 1$ (cf.~\eqref{eq:lengthconstraint}), we deduce
\begin{align}
    \begin{split}
        &\norm{\nabla\rkla{\id-\Ihop}\gkla{\bs{d}\h\nh\times\Ih{\sinc{n+1}}}}_{L^2\trkla{\D}}^2\\
        &~\leq Ch^2\norm{\nabla\bs{d}\h\nh}_{L^\infty\trkla{\D}}^2\norm{\nabla\Ih{\sinc{n+1}}}_{L^2\trkla{\D}}^2\\
        &~\leq C\norm{\bs{d}\h\nh}_{L^\infty\trkla{\D}}^2\norm{\nabla\Ih{\sinc{n+1}}}_{L^2\trkla{\D}}^2\leq C\norm{\nabla\Ih{\sinc{n+1}}}_{L^2\trkla{\D}}^2\,.
    \end{split}
\end{align}

For the third term on the right-hand side of \eqref{eq:proof:energy:1} we obtain using Lemma \ref{lem:interpolationerror}
\begin{align}
\begin{split}
\mathfrak{r}_3=&\,-\tfrac12\rkla{\nabla\trkla{\id-\Ihop}\gkla{\trkla{\bs{d}\h\nn-\bs{d}\h\no}\times\Ih{\sinc{n+1}}},\nabla\bs{d}\h\no}\\
\leq&\, \tfrac12\norm{\nabla\trkla{\id-\Ihop}\gkla{\trkla{\bs{d}\h\nn-\bs{d}\h\no} \times\Ih{\sinc{n+1}}}}_{L^1\trkla{\D}}\norm{\nabla\bs{d}\h\no}_{L^\infty\trkla{\D}}\\
\leq&\, C\norm{\nabla\bs{d}\h\nn-\nabla\bs{d}\h\no}_{L^2\trkla{\D}}\norm{\bs{d}\h\no}_{L^\infty\trkla{\D}}\norm{\nabla\Ih{\sinc{n+1}}}_{L^2\trkla{\D}}\\
\leq&\,\alpha\norm{\nabla\bs{d}\h\nn-\nabla\bs{d}\h\no}_{L^2\trkla{\D}}^2+C\norm{\nabla\Ih{\sinc{n+1}}}_{L^2\trkla{\D}}^2
\end{split}
\end{align}
due to a standard inverse estimate.
As we will keep the term $\mathfrak{r}_2$ for the moment, only $\mathfrak{r}_4$ remains to be estimated.
To deal with the last term, we use the identity $\trkla{\bs{a}\times\bs{b}}\cdot\bs{a}=0$ and the norm equivalence \eqref{eq:normequivalence} to compute
\begin{align}
\begin{split}
\mathfrak{r}_4=&\,\tfrac12\sum_{q=1}^d\rkla{\rkla{\frac{\partial\bs{d}\h\no}{\partial x_q}\times\Ih{\sinc{n+1}}},\frac{\partial\bs{d}\h\nn}{\partial x_q}} +\sum_{q=1}^d\rkla{\bs{d}\h\nh\times\frac{\partial\Ih{\sinc{n+1}}}{\partial x_q},\frac{\partial\bs{d}\h\nn}{\partial x_q}}\\
=&\,\tfrac12\sum_{q=1}^d\rkla{\rkla{\frac{\partial\bs{d}\h\no}{\partial x_q}\times\Ih{\sinc{n+1}}},\frac{\partial\trkla{\bs{d}\h\nn-\bs{d}\h\no}}{\partial x_q}} +\sum_{q=1}^d\rkla{\bs{d}\h\nh\times\frac{\partial\Ih{\sinc{n+1}}}{\partial x_q},\frac{\partial\bs{d}\h\no}{\partial x_q}}\\
&\,+\sum_{q=1}^d\rkla{\bs{d}\h\nh\times\frac{\partial\Ih{\sinc{n+1}}}{\partial x_q},\frac{\partial\trkla{\bs{d}\h\nn-\bs{d}\h\no}}{\partial x_q}}\\
\leq&\,\alpha\norm{\nabla\bs{d}\h\nn-\nabla\bs{d}\h\no}_{L^2\trkla{\D}}^2 +C\sum_{q=1}^d\norm{\frac{\partial\bs{d}\h\no}{\partial x_q}\times\Ih{\sinc{n+1}}}_{L^2\trkla{\D}}^2 \\
&\,+ C\sum_{q=1}^d\norm{\bs{d}\h\nh\times\frac{\partial\Ih{\sinc{n+1}}}{\partial x_q}}_{L^2\trkla{\D}}^2 +\sum_{q=1}^d\rkla{\bs{d}\h\no\times\frac{\partial\Ih{\sinc{n+1}}}{\partial x_q},\frac{\partial\bs{d}\h\no}{\partial x_q}}\\
&\,+\alpha \norm{\bs{d}\h\nn-\bs{d}\h\no}\h^2+ C\sum_{q=1}^d\norm{\frac{\partial\Ih{\sinc{n+1}}}{\partial x_q}\times\frac{\partial\bs{d}\h\no}{\partial x_q}}_{L^2\trkla{\D}}^2\,.
\end{split}
\end{align}
Next, we can estimate $\norm{\bs{d}\h\nh\times\tfrac{\partial\Ih{\sinc{n+1}}}{\partial x_q}}_{L^2\trkla{\D}}^2\leq C\norm{\nabla\Ih{\sinc{n+1}}}_{L^2\trkla{\D}}^2$ using the $L^\infty$-bound \eqref{eq:lengthconstraint}.
Thus, \eqref{eq:proof:energy:0.5} reads after adding $\alpha\norm{\bs{d}\h\nn-\bs{d}\h\no}_{L^2\trkla{\D}}^2$ on both sides, substituting \eqref{eq:proof:energy:0}, and choosing $\alpha>0$ sufficiently small
\begin{align}\label{eq:proof:energy:intermediate}
\begin{split}
\tfrac12&\,\norm{\nabla\bs{d}\h\nn}_{L^2\trkla{\D}}^2+\tfrac14\norm{\nabla\bs{d}\h\nn-\nabla\bs{d}\h\no}_{L^2\trkla{\D}}^2+\alpha\norm{\bs{d}\h\nn-\bs{d}\h\no}_{L^2\trkla{\D}}^2+\tau\tfrac12\norm{\bs{d}\h\nh\times\Delta\h\bs{d}\h\nn}\h^2\\
&\leq\tfrac12\norm{\nabla\bs{d}\h\no}_{L^2\trkla{\D}}^2+ C\norm{\nabla\Ih{\sinc{n+1}}}_{L^2\trkla{\D}}^2 \\
&\qquad-\rkla{\nabla\trkla{\id-\Ihop}\gkla{\bs{d}\h\no\times\Ih{\sinc{n+1}}},\nabla\bs{d}\h\no} + \sum_{q=1}^d\rkla{\bs{d}\h\no\times\frac{\partial\Ih{\sinc{n+1}}}{\partial x_q},\frac{\partial\bs{d}\h\no}{\partial x_q}}\\
&\qquad+C\sum_{q=1}^d\norm{\frac{\partial\bs{d}\h\no}{\partial x_q}\times\Ih{\sinc{n+1}}}_{L^2\trkla{\D}}^2 +C\sum_{q=1}^d\norm{\frac{\partial\Ih{\sinc{n+1}}}{\partial x_q}\times\frac{\partial\bs{d}\h\no}{\partial x_q}}_{L^2\trkla{\D}}^2\\
&\qquad+C\alpha\norm{\sinc{n+1}}\h^2\,.
\end{split}
\end{align}

Summing \eqref{eq:proof:energy:intermediate} from $n=0$ to $\tilde{N}-1\leq N-1$ and taking the $p$-th power, we obtain for the expected value
\begin{align}
\begin{split}
&\expected{\norm{\nabla \bs{d}\h^{\tilde{N}}}_{L^2\trkla{\D}}^{2p}} + \expected{\rkla{\sum_{n=0}^{\tilde{N}-1}\norm{\bs{d}\h\nn-\bs{d}\h\no}_{H^1\trkla{\D}}^2}^p} \\
&\qquad+\expected{\rkla{\sum_{n=0}^{\tilde{N}-1}\tau\norm{\bs{d}\h\nh\times\Delta\h\bs{d}\h\nn}\h^2}^p}\\
&\quad\leq C\expected{\norm{\nabla\bs{d}\h^0}_{L^2\trkla{\D}}^{2p}} +C\expected{\rkla{\sum_{n=0}^{\tilde{N}-1}\norm{\Ih{\sinc{n+1}}}_{H^1\trkla{\D}}^2}^p}\\
&\qquad +C\expected{\max_{1\leq M\leq \tilde{N}}\abs{\sum_{n=0}^{M-1}\rkla{\nabla\trkla{\id-\Ihop}\gkla{\bs{d}\h\no\times\Ih{\sinc{n+1}}},\nabla\bs{d}\h\no}}^p} \\
&\qquad+C\sum_{q=1}^d\expected{\max_{1\leq M\leq \tilde{N}}\abs{\sum_{n=0}^{M-1}\rkla{\bs{d}\h\no\times\frac{\partial\Ih{\sinc{n+1}}}{\partial x_q},\frac{\partial \bs{d}\h\no}{\partial x_q}}}^p} \\
&\qquad+C\sum_{q=1}^d\expected{\rkla{\sum_{n=0}^{\tilde{N}-1}\norm{\frac{\partial\bs{d}\h\no}{\partial x_q}\times\Ih{\sinc{n+1}}}_{L^2\trkla{\D}}^2}^p}\\
&\qquad+C\sum_{q=1}^d\expected{\rkla{\sum_{n=0}^{\tilde{N}-1}\norm{\frac{\partial\Ih{\sinc{n+1}}}{\partial x_q}\times\frac{\partial\bs{d}\h\no}{\partial x_q}}_{L^2\trkla{\D}}^2}^p}\\
&\quad=:C\expected{\norm{\nabla\bs{d}\h^0}_{L^2\trkla{\D}}^{2p}} + \sum_{r=1}^5\mathfrak{R}_r\tekla{\tilde{N}-1}\,.
\end{split}
\end{align}
Standard arguments and \ref{item:W3} show that 
\begin{align}\label{eq:estimatestochastic}
\mathfrak{R}_1\tekla{\tilde{N}-1} \leq C\sum_{n=0}^{\tilde{N}-1}\tau \rkla{\sum_{i=1}^3\sum_{l\in\mathds{Z}\h}\norm{\lambda_{li}\Ih{\g{li}}}_{H^1\trkla{\D}}^2}^p\leq C\,.
\end{align}
For $\mathfrak{R}_2\tekla{\tilde{N}-1}$ and $\mathfrak{R}_3\tekla{\tilde{N}-1}$, we apply the Burkholder--Davis--Gundy inequality, Hölder's inequality, assumption \ref{item:W3}, and \eqref{eq:lengthconstraint} to deduce
\begin{align}
\begin{split}
\mathfrak{R}_2\tekla{\tilde{N}-1}\leq&\,C\sum_{n=0}^{\tilde{N}-1}\tau\expected{\rkla{\norm{\nabla\bs{d}\h\no}_{L^2\trkla{\D}}^2\sum_{i=1}^3\sum_{l\in\mathds{Z}\h}\norm{\nabla\trkla{\id-\Ihop}\gkla{\bs{d}\h\no\times\Ih{\lambda_{li}\g{li}}}}_{L^2\trkla{\D}}^2}^{p/2}}\\
\leq&\,C\sum_{n=0}^{\tilde{N}-1}\tau\expected{\norm{\nabla\bs{d}\h\no}_{L^2\trkla{\D}}^{2p}\rkla{Ch\sum_{i=1}^3\sum_{l\in\mathds{Z}\h}\norm{\nabla\Ih{\lambda_{li}\g{li}}}_{L^\infty\trkla{\D}}^2}^{p/2}}\\
\leq&\,Ch^{p/2}\sum_{n=0}^{\tilde{N}-1}\tau\expected{\norm{\nabla\bs{d}\h\no}_{L^2\trkla{\D}}^{2p}}\,,
\end{split}
\end{align}
\begin{align}
\begin{split}
\mathfrak{R}_3\tekla{\tilde{N}-1}\leq&\,C\sum_{q=1}^d\sum_{n=0}^{\tilde{N}-1}\tau\expected{\rkla{\sum_{i=1}^3\sum_{l\in\mathds{Z}\h}\abs{\rkla{\bs{d}\h\no\times\frac{\partial\Ih{\lambda_{li}\g{li}}}{\partial x_q},\frac{\partial\bs{d}\h\no}{\partial x_q}}}^2}^{p/2}}\\
\leq&\,C\sum_{n=0}^{\tilde{N}-1}\tau\expected{\rkla{\norm{\bs{d}\h\no}_{L^\infty\trkla{\D}}^2\norm{\nabla\bs{d}\h\no}_{L^2\trkla{\D}}^2\sum_{i=1}^3\sum_{l\in\mathds{Z}\h}\norm{\lambda_{li}\nabla\Ih{\g{li}}}_{L^2\trkla{\D}}^2}^{p/2}}\\
\leq&\,C\sum_{n=0}^{\tilde{N}-1}\tau\expected{\norm{\nabla\bs{d}\h\no}_{L^2\trkla{\D}}^{2p}} +C\,,
\end{split}
\end{align}
For $\mathfrak{R}_4\tekla{\tilde{N}-1}$ and $\mathfrak{R}_5\tekla{\tilde{N}-1}$, we obtain
\begin{align}
\begin{split}
\mathfrak{R}_4\tekla{\tilde{N}-1}\leq&\,C\sum_{q=1}^d\sum_{n=0}^{\tilde{N}-1}\tau \expected{\rkla{\sum_{i=1}^3\sum_{l\in\mathds{Z}\h}\norm{\frac{\partial\bs{d}\h\no}{\partial x_q}\times\Ih{\lambda_{li}\g{li}}}_{L^2\trkla{\D}}^2}^p}\\
\leq&\,C\sum_{n=0}^{\tilde{N}-1}\tau\expected{\norm{\nabla\bs{d}\h\no}_{L^2\trkla{\D}}^{2p}\rkla{\sum_{i=1}^3\sum_{l\in\mathds{Z}\h}\norm{\lambda_{li}\Ih{\g{li}}}_{L^\infty\trkla{\D}}^2}^p}\\
\leq&\,C\sum_{n=0}^{\tilde{N}-1}\tau\expected{\norm{\nabla\bs{d}\h\no}_{L^2\trkla{\D}}^{2p}}\,,
\end{split}
\end{align}
\begin{align}
\begin{split}
\mathfrak{R}_5\tekla{\tilde{N}-1}\leq&\,C\sum_{q=1}^d\sum_{n=0}^{\tilde{N}-1}\tau\expected{\rkla{\sum_{i=1}^3\sum_{l\in\mathds{Z}\h}\norm{\frac{\partial\Ih{\lambda_{li}\g{li}}}{\partial x_q}\times\frac{\partial\bs{d}\h\no}{\partial x_q}}_{L^2\trkla{\D}}^2}^p}\\
\leq&\,C\sum_{n=0}^{\tilde{N}-1}\tau\expected{\norm{\nabla\bs{d}\h\no}_{L^2\trkla{\D}}^{2p}\rkla{\sum_{i=1}^3\sum_{l\in\mathds{Z}\h}\norm{\nabla\Ih{\lambda_{li}\g{li}}}_{L^\infty\trkla{\D}}^2}^p}\\
\leq&\,C\sum_{n=0}^{\tilde{N}-1}\tau\expected{\norm{\nabla\bs{d}\h\no}_{L^2\trkla{\D}}^{2p}}\,.
\end{split}
\end{align}
Consequently, a Gronwall argument provides
\begin{align}
\max_{0\leq\tilde{N}\leq N}\expected{\norm{\nabla\bs{d}\h^{\tilde{N}}}_{L^2\trkla{\D}}^{2p}} \leq C\trkla{T,\bs{d}\h^0}\,.
\end{align}
Summing \eqref{eq:proof:energy:intermediate} from $n=0$ to $\tilde{N}-1\leq N-1$ and taking the $p$-th power and the maximum over $\tilde{N}\in\tgkla{0,\ldots,N}$, we obtain for the expected value
\begin{align}
\begin{split}
&\expected{\max_{0\leq\tilde{N}\leq N}\norm{\nabla\bs{d}\h^{\tilde{N}}}_{L^2\trkla{\D}}^{2p}}+\expected{\rkla{\sum_{n=0}^{N-1}\norm{\bs{d}\h\nn-\bs{d}\h\no}_{H^1\trkla{\D}}^2}^p}\\
&\qquad+\expected{\rkla{\sum_{n=0}^{N-1}\tau\norm{\bs{d}\h\nh\times\Delta\h\bs{d}\h\nn}\h^2}^p}\\
&\quad\leq C\expected{\norm{\nabla\bs{d}\h^0}_{L^2\trkla{\D}}^{2p}}+\sum_{r=1}^5\mathfrak{R}_r\tekla{N-1}\\
&\quad\leq C\expected{\norm{\nabla\bs{d}\h^0}_{L^2\trkla{\D}}^{2p}} + C\trkla{T}+C\sum_{n=0}^{N-1}\tau\expected{\norm{\nabla\bs{d}\h\no}_{L^2\trkla{\D}}^{2p}}\leq C\,.
\end{split}
\end{align}
Recalling $\norm{\bs{d}\h\nn}_{L^\infty\trkla{\D}}=\norm{\bs{d}\h^0}_{L^\infty\trkla{\D}}$ completes the estimate on the full $H^1\trkla{\D}$-norm.
\end{proof}
Next, we establish Nikolskii-type estimates for $\trkla{\bs{d}\h\no}_{n}$. 
On the one hand, this guarantees that the limit function $\bs{d}$ will be Hölder continuous with respect to time (cf.~\cite{Banas2013, Simon1990}).
On the other hand, we will use the special case $k=1$ in Lemma \ref{lem:stochasticerror} and Lemma \ref{lem:B} to show that our approximation of the Stratonovich integral indeed converges towards an \Ito~integral with the proper correction terms.

\begin{lemma}\label{lem:Nikolskii}
Let assumptions \ref{item:T}, \ref{item:S}, \ref{item:I}, \ref{item:W1}, \ref{item:W2}, and \ref{item:W3} hold true.
Then, the solution $\trkla{\bs{d}\h\no}_n$ satisfies
\begin{align}
\expected{\sum_{m=0}^{N-k}\tau\norm{\bs{d}\h^{m+k}-\bs{d}\h^m}\h^{2q}}\leq C\trkla{k\tau}^q
\end{align}
for all $q\geq1$ and all $k=0,\ldots,N$ with a constant $C$ depending on $q$ but not on $h$, $\tau$, or $k$.
\end{lemma}
\begin{proof}
Summing \eqref{eq:discscheme} from $n=m$ to $m+k-1$, choosing $\bs{\psi}\h:=\bs{d}\h^{m+k}-\bs{d}\h^m$, taking the $q$-th power, multiplying by $\tau$, and summing from $m=0$ to $N-k$ yields
\begin{align}
\begin{split}
&\expected{\sum_{m=0}^{N-k}\tau\norm{\bs{d}\h^{m+k}-\bs{d}\h^m}\h^{2q}}\\
&\qquad\leq C\expected{\sum_{m=0}^{N-k}\tau\abs{\tau\sum_{n=m}^{m+k-1}\rkla{\bs{d}\h\nh\times\Delta\h\bs{d}\h\nn,\bs{d}\h\nh\times\trkla{\bs{d}\h^{m+k}-\bs{d}\h^m}}\h}^q}\\
&\qquad\qquad+C\expected{\sum_{m=0}^{N-k}\tau\abs{\sum_{n=m}^{m+k-1}\rkla{\bs{d}\h^{m+k}-\bs{d}\h^m,\bs{d}\h\no\times\sinc{n+1}}\h}^q}\\
&\qquad\qquad+C\expected{\sum_{m=0}^{N-k}\tau\abs{\sum_{n=m}^{m+k-1}\rkla{\bs{d}\h^{m+k}-\bs{d}\h^m,\trkla{\bs{d}\h\nn-\bs{d}\h\no}\times\sinc{n+1}}\h}^q}\\
&\qquad=:I+II+III\,.
\end{split}
\end{align}
To estimate $I$, we use Hölder's inequality, Young's inequality with $0<\alpha<\!\!<1$, \eqref{eq:lengthconstraint}, and Lemma \ref{lem:energy} to deduce
\begin{align}
\begin{split}
I\leq&\, C\expected{\sum_{m=0}^{N-k}\tau\rkla{\sum_{n=m}^{m+k-1}\tau\norm{\bs{d}\h\nh\times\Delta\h\bs{d}\h\nn}\h\norm{\bs{d}\h\nh}_{L^\infty\trkla{\D}}\norm{\bs{d}\h^{m+k}-\bs{d}\h^m}\h}^q}\\
\leq&\,\alpha\expected{\sum_{m=0}^{N-k}\tau\norm{\bs{d}\h^{m+k}-\bs{d}\h^m}\h^{2q}}\\
&\,+C\expected{\sum_{m=0}^{N-k}\tau \rkla{\sum_{n=m}^{m+k-1} \tau\norm{\bs{d}\h\nh\times\Delta\h\bs{d}\h\nn}\h\norm{\bs{d}\h\nh}_{L^\infty\trkla{\D}}}^{2q}}\\
\leq&\,\alpha\expected{\sum_{m=0}^{N-k}\tau\norm{\bs{d}\h^{m+k}-\bs{d}\h^m}\h^{2q}}\\
&\,+C\expected{\max_{0\leq n\leq N-1}\norm{\bs{d}\h\nh}_{L^\infty\trkla{\D}}^{2q} \sum_{m=0}^{N-k}\tau\trkla{k\tau}^q \rkla{\sum_{n=m}^{m+k-1}\tau\norm{\bs{d}\h\nh\times\Delta\h\bs{d}\h\nn}\h^2}^q }\\
\leq&\,\alpha\expected{\sum_{m=0}^{N-k}\tau\norm{\bs{d}\h^{m+k}-\bs{d}\h^m}\h^{2q}} +C\trkla{k\tau}^q\,.
\end{split}
\end{align}
Applying Young's inequality, the Burkholder--Davis--Gundy inequality, and \eqref{eq:lengthconstraint} provides
\begin{align}
\begin{split}
II\leq&\,\alpha\expected{\sum_{m=0}^{N-k}\tau\norm{\bs{d}\h^{m+k}-\bs{d}\h^m}\h^{2q}} +C\expected{\sum_{m=0}^{N-k}\tau\norm{\sum_{n=m}^{m+k-1}\bs{d}\h\no\times\sinc{n+1}}\h^{2q}}\\
\leq&\,\alpha\expected{\sum_{m=0}^{N-k}\tau\norm{\bs{d}\h^{m+k}-\bs{d}\h^m}\h^{2q}} +C\sum_{m=0}^{N-k}\tau\trkla{k\tau}^{q-1}\sum_{n=m}^{m+k-1}\tau\expected{\rkla{\sum_{i=1}^d\sum_{l\in\mathds{Z}\h}\norm{\bs{d}\h\no\times\trkla{\lambda_{li}\g{li}}}\h^2}^q}\\
\leq&\,\alpha\expected{\sum_{m=0}^{N-k}\tau\norm{\bs{d}\h^{m+k}-\bs{d}\h^m}\h^{2q}} +C\trkla{k\tau}^{q}\,.
\end{split}
\end{align}

Using Hölder's inequality, standard Sobolev embeddings, Young's inequality, and Lemma \ref{lem:energy}, we compute
\begin{align}
\begin{split}
III\leq&\,C\expected{\sum_{m=0}^{N-k}\tau \rkla{\sum_{n=m}^{m+k-1}\norm{\bs{d}\h^{m+k}-\bs{d}\h^m}\h\norm{\bs{d}\h\nn-\bs{d}\h\no}_{H^1\trkla{\D}}\norm{\Ih{\sinc{n+1}}}_{H^1\trkla{\D}}}^q}\\
\leq&\,C\expected{\sum_{m=0}^{N-k}\tau\rkla{\trkla{k\tau}^{1/2}\sum_{n=m}^{m+k-1}\norm{\bs{d}\h^{m+k}-\bs{d}\h^m}\h\norm{\bs{d}\h\nn-\bs{d}\h\no}_{H^1\trkla{\D}}^2 }^q}\\
&\,+C\expected{\sum_{m=0}^{N-k}\tau\rkla{\trkla{k\tau}^{-1/2}\sum_{n=m}^{m+k-1}\norm{\bs{d}\h^{m+k}-\bs{d}\h^m}\h\norm{\Ih{\sinc{n+1}}}_{H^1\trkla{\D}}^2}^q}\\
\leq&\,\alpha\expected{\sum_{m=0}^{N-k}\tau\norm{\bs{d}\h^{m+k}-\bs{d}\h^m}\h^{2q}}+C\expected{\sum_{m=0}^{N-k}\tau\trkla{k\tau}^q\rkla{\sum_{n=m}^{m+k-1}\norm{\bs{d}\h\nn-\bs{d}\h\no}_{H^1\trkla{\D}}^2}^{2q}}\\
&\,+C\expected{\sum_{m=0}^{N-k}\trkla{k\tau}^{-q}\rkla{\sum_{n=m}^{m+k-1}\norm{\Ih{\sinc{n+1}}}_{H^1\trkla{\D}}^2}^{2q}}\\
\leq&\,\alpha\expected{\sum_{m=0}^{N-k}\tau\norm{\bs{d}\h^{m+k}-\bs{d}\h^m}\h^{2q}} +C\trkla{k\tau}^q\,.
\end{split}
\end{align}
\end{proof}

\section{Limit of the energy estimate}\label{sec:limenergy}
In this section, we shall prove the following convergence result:
\begin{theorem}\label{thm:convergence}
Let $\bs{W}$ be a cylindrical Wiener process which is defined on a complete filtered probability space $\trkla{\Omega,\mathcal{A},\trkla{\mathcal{F}_t}_{t\geq0},\mathds{P}}$ and let the assumptions \ref{item:W1}, \ref{item:W2}, \ref{item:W3}, \ref{item:T}, \ref{item:S}, \ref{item:I}, \ref{item:C}, and \ref{item:Z} hold true.
Then, for $p\in[1,\infty)$, $q\in\trkla{1,\infty}$ there exists a progressively $\mathcal{F}$-measurable process
\begin{align}
\begin{split}
\bs{d}\in& L^\infty\trkla{\Omega\times\trkla{0,T}\times\D}\cap L^{p}_{\weakstartop}\trkla{\Omega;L^\infty\trkla{0,T;H^1\trkla{\D}}}\\
&\,\cap L^{2q}\trkla{\Omega;C^{0,\trkla{q-1}/\trkla{2q}}\rkla{\tekla{0,T};L^2\trkla{\D}}}
\end{split}
\end{align}
and progressively measurable processes $ \bs{q} \in L^{p}(\Omega; L^2(\D\times (0,T)))$ and $ E \in L^{p}(\Omega; \mathfrak{D}([0,T]))$. 
These processes can be obtained as limits of suitable subsequences of discrete solutions to \eqref{eq:discscheme} in the sense that
\begin{subequations}\label{eq:convergence}
\begin{align}
\bs{d}\h\tpm\stackrel{*}{\rightharpoonup}&\,\bs{d}&&\text{in~}L^\infty\trkla{\Omega\times\trkla{0,T}\times\D}^3\,,\label{eq:envarconvergence:weakstar}\\
\bs{d}\h\tpm\stackrel{*}{\rightharpoonup}&\,\bs{d}&&\text{in~}L^{p}_{\weakstartop}\trkla{\Omega;L^\infty\trkla{0,T;H^1\trkla{\D}}}^3\,,\label{eq:envarconvergence:weakgrad}\\
\bs{d}\h\ti\times\nabla \bs{d}\h\ti\stackrel{*}{\rightharpoonup}&\,\bs{q}&&\text{in~}L^{p}_{\weakstartop}\trkla{\Omega;L^\infty\trkla{0,T;L^2\trkla{\D}}}^{3\times d}\,,\label{eq:envarconvergence:dissipation}\\
\frac{1}{2}\| \nabla \bs d\h\tp \|_{L^2(\D)}^2 \stackrel{*}{\rightharpoonup}& \, E &&\text{in~}L^{p}_{w^*}\trkla{\Omega;L^\infty(0,T)} \,,\label{eq:envarconvergence:energy}\\
\Ih{\bs{d}\h\ti\times\Delta\h\bs{d}\h\ti}\rightharpoonup&\,\div\tgkla{\bs{q}}&&\text{in~}L^{p}\trkla{\Omega;L^2\trkla{0,T;L^2\trkla{\D}}}^3\,,\label{eq:envarconvergence:dissipation2}
\end{align}
and
\begin{align}
\sum_{i=1}^3\sum_{l\in\mathds{Z}\h}\tabs{\lambda_{li}}^2\Ih{\rkla{\bs{d}\h\ti\times\g{li}}\times\g{li}}\rightharpoonup&\sum_{i=1}^3\sum_{l\in\mathds{Z}}\tabs{\lambda_{li}}^2\rkla{\bs{d}\times\g{li}}\times\g{li}&&\text{in~}L^p\trkla{\Omega\times\trkla{0,T}\times\D}\,.\label{eq:envarconvergence:stratcorrection}
\end{align}
\end{subequations}
Furthermore, the limits $(\bs{d}, \bs q, E) $ are an energy-variational solution in the sense of Definition~\ref{def:envar2}. 
\end{theorem}
Before we prove this theorem in Section \ref{sec:conv}, we need to establish some auxiliary results.
The main ingredient for the proof of Theorem \ref{thm:convergence} will be a suitably weighted (in time and probability) energy estimate, which has to be written in terms of time-discrete approximations of \Ito-integrals.\\

We choose $\bs{\psi}\h=-\Delta\h\bs{d}\h\nh$ in \eqref{eq:discscheme}. Further, as \eqref{eq:discscheme} holds true on pointwise on all nodes of $\Th$, we may multiply by $\bs{\varphi}\no:=\bs{\varphi}\trkla{t\no}$, where $\bs{\varphi}$ is a test process in the sense of Definition \ref{def:testprocess}.

The difference of these testing procedures reads
\begin{multline}
\tfrac12\norm{\nabla\bs{d}\h\nn}_{L^2\trkla{\D}}^2-\tfrac12\norm{\nabla\bs{d}\h\no}_{L^2\trkla{\D}}^2-\trkla{\bs{d}\h\nn-\bs{d}\h\no,\bs{\varphi}\no} +\tau\norm{\bs{d}\h\nh\times\Delta\h\bs{d}\h\nh}\h^2\\
+\tau\rkla{\bs{d}\h\nh\times\Delta\h\trkla{\bs{d}\h\nn-\bs{d}\h\nh},\bs{d}\h\nh\times\Delta\h\bs{d}\h\nh}\h \\
+ \tau\rkla{\Ih{\trkla{\bs{d}\h\nh\times\Delta\h\bs{d}\h\nh}\times\bs{d}\h\nh},\bs{\varphi}\no}\\
+\tau\rkla{\Ih{\trkla{\bs{d}\h\nh\times\Delta\h\trkla{\bs{d}\h\nn-\bs{d}\h\nh}}\times\bs{d}\h\nh},\bs{\varphi}\no}\\
=\rkla{\bs{d}\h\nh\times\Delta\h\bs{d}\h\nh,\sinc{n+1}}\h +\rkla{\Ih{\sinc{n+1}\times\bs{d}\h\nh},\bs{\varphi}\no}\,.
\end{multline}
Now, we multiply the resulting equality by $\eta\no=\eta\trkla{t\no}$ for $\eta\in C_0^\infty\trkla{[0,T);[0,\infty)}$ and sum over $n$. Applying the discrete integration by parts formula
\begin{align}
\sum_{n=0}^{N-1} (a\nn-a\no)b\no = -\sum_{n=0}^{N-1} a\nn (b\nn - b\no) - a^0 b^0  
\end{align}
for all $\tgkla{a\no,b\no}_{n\in\tgkla{0,\ldots,N}}$ with $b_N=0$ and the discrete product rule $a\nn b\nn-a\no b\no=\trkla{a\nn-a\no}b\nn+a\no\trkla{b\nn-b\no}$, we find
\begin{align}\label{eq:weakEnergy}
\begin{split}
-\sum_{n=0}^{N-1}&\tau\frac{\eta\nn-\eta\no}\tau\ekla{\tfrac12\norm{\nabla\bs{d}\h\nn}_{L^2\trkla{\D}}^2-\trkla{\bs{d}\h\nn,\bs{\varphi}\nn}}-\eta^0\ekla{\tfrac12\norm{\nabla\bs{d}\h^0}_{L^2\trkla{\D}}^2-\trkla{\bs{d}\h^0,\bs{\varphi}^0}}\\
&+\sum_{n=0}^{N-1}\eta\no\rkla{\bs{d}\h\nn,\bs{\varphi}\nn-\bs{\varphi}\no}+\sum_{n=0}^{N-1}\tau\eta\no\norm{\bs{d}\h\nh\times\Delta\h\bs{d}\h\nh}\h^2\\
&+\sum_{n=0}^{N-1}\tau\eta\no\rkla{\Ih{\trkla{\bs{d}\h\nh\times\Delta\h\bs{d}\h\nh}\times\bs{d}\h\nh},\bs{\varphi}\no}\\
&+\sum_{n=0}^{N-1}\tau\eta\no\rkla{\bs{d}\h\nh\times\Delta\h\trkla{\bs{d}\h\nn-\bs{d}\h\nh},\bs{d}\h\nh\times\Delta\h\bs{d}\h\nh}\h\\
&+\sum_{n=0}^{N-1}\tau\eta\no \rkla{\Ih{\trkla{\bs{d}\h\nh\times\Delta\h\trkla{\bs{d}\h\nn-\bs{d}\h\nh}}\times\bs{d}\h\nh},\bs{\varphi}\no}\\
=&\,\sum_{n=0}^{N-1}\eta\no\ekla{\rkla{\bs{d}\h\nh\times\Delta\h\bs{d}\h\nh,\sinc{n+1}}\h-\rkla{\bs{d}\h\nh\times\Ih{\sinc{n+1}},\bs{\varphi\no}}}\,.
\end{split}
\end{align}
The error terms in the fourth and fifth line (i.e.~the last two terms on the left-hand side) will be denoted by
\begin{align}
    \mathfrak{S}_1:=&\,\sum_{n=0}^{N-1}\tau\eta\no\rkla{\bs{d}\h\nh\times\Delta\h\trkla{\bs{d}\h\nn-\bs{d}\h\nh},\bs{d}\h\nh\times\Delta\h\bs{d}\h\nh}\h\,,\\
    \mathfrak{S}_2:=&\,\sum_{n=0}^{N-1}\tau\eta\no\rkla{\Ih{\trkla{\bs{d}\h\nh\times\Delta\h\trkla{\bs{d}\h\nn-\bs{d}\h\nh}}\times\bs{d}\h\nh},\bs{\varphi}\no}\,.
\end{align}
These terms will be discussed in Lemma \ref{lem:diff} and Corollary \ref{cor:diffS} below.
To deal with the terms involving a discrete approximation of the Stratonovich integral or the test process $\bs{\varphi}$, we need to reformulate them as stochastic integrals in the sense of \Ito.\\
We start with the third term on the left-hand side of \eqref{eq:weakEnergy}.
By using $\bs{\varphi}\nn-\bs{\varphi}\no=\int_{t\no}^{t\nn} A\mathrm{d}t +\int_{t\no}^{t\nn}B\mathrm{d}\bs{W}$, we obtain
\begin{multline}
\sum_{n=0}^{N-1}\eta\no\rkla{\bs{d}\h\nn,\bs{\varphi}\nn-\bs{\varphi}\no}=\sum_{n=0}^{N-1}\eta\no\skla{\int_{t^n}^{t^{n+1}}A\dt,\bs{d}\h\nn}+\sum_{n=0}^{N-1}\eta\no\skla{\int_{t^n}^{t^{n+1}}B\mathrm{d}\bs{W},\bs{d}\h\no}\\
+\sum_{n=0}^{N-1}\eta\no\skla{\int_{t\no}^{t\nn}B\mathrm{d}\bs{W},\bs{d}\h\nn-\bs{d}\h\no}\,.
\end{multline}
Inserting \eqref{eq:discscheme} into the last term, we deduce
\begin{align*}
\sum_{n=0}^{N-1}&\eta\no\skla{\int_{t\no}^{t\nn}B\mathrm{d}\bs{W},\trkla{\bs{d}\h\nn-\bs{d}\h\no}} \\
=&\sum_{n=0}^{N-1}\eta\no\tau\skla{\int_{t\no}^{t\nn}B\mathrm{d}\bs{W},\Ih{\trkla{\bs{d}\h\nh\times\Delta\h\bs{d}\h\nn}\times\bs{d}\h\nh}} \\&
+ \sum_{n=0}^{N-1}\eta\no\skla{\int_{t\no}^{t\nn}B\mathrm{d}\bs{W},\Ih{\bs{d}\h\nh\times\sinc{n+1}}}\\
=&\sum_{n=0}^{N-1}\eta\no\sum_{i=1}^3\sum_{l\in\mathds{Z}\h}\ekla{\int_{t\no}^{t\nn}\skla{B_{li}\trkla{t}\g{li},\Ih{\bs{d}\h\no\times\lambda_{li}\g{li}}} \dt} \\&
+\sum_{n=0}^{N-1}\eta\no\tau\skla{\int_{t\no}^{t\nn}B\mathrm{d}\bs{W},\Ih{\trkla{\bs{d}\h\nh\times\Delta\h\bs{d}\h\nn}\times\bs{d}\h\nh}}\\&
+\sum_{n=0}^{N-1}\eta\no\sum_{i=1}^3\sum_{l\in\mathds{Z}\h}\skla{\ekla{\incbeta{n+1}{l}{i}\int_{t\no}^{t\nn}\!\!\!\!\!\!B_{li}\trkla{t}\mathrm{d}\beta_l^i -\int_{t\no}^{t\nn}\!\!\!\!\!\!B_{li}\trkla{t}\dt}\g{li} ,\lambda_{li}\Ih{\bs{d}\h\no\times\g{li}}}\\
&
+\sum_{n=0}^{N-1}\eta\no\sum_{\trkla{i,l}\neq\trkla{j,k}}\skla{\int_{t\no}^{t\nn}B_{kj}\trkla{t}\g{kj}\mathrm{d}\beta_k^j,\Ih{\bs{d}\h\no\times\lambda_{li}\g{li}\incbeta{n+1}{l}{i}}}
\\& 
+\frac 12 \sum_{n=0}^{N-1}\eta\no\sum_{i,j=1}^3 \sum_{l\in\mathbb Z\h}\sum_{k\in\mathbb Z}\skla{\int_{t\no}^{t\nn}B_{kj}\trkla{t}\g{kj}\mathrm{d}\beta_k^j,\Ih{(\bs{d}\h\nn-\bs{d}\h\no)\times\lambda_{li}\g{li}\incbeta{n+1}{l}{i}}}\\
=:&\sum_{n=0}^{N-1}\eta\no\sum_{i=1}^3\sum_{l\in\mathds{Z}\h}\ekla{\int_{t\no}^{t\nn}\skla{B_{li}\trkla{t}\g{li},\Ih{\bs{d}\h\tm\times\lambda_{li}\g{li}}}\dt}  -\mathfrak{E}_1-\mathfrak{E}_2-\mathfrak{E}_3-\mathfrak{E}_4\,.
\end{align*}

Here, we want to point out that in the definition of $\mathfrak{E}_3$ the index $l$ is taken from the finite set $\mathds{Z}\h$, while $k$ is taken from the infinite set $\mathds{Z}$.
For the first part of the right-hand side of \eqref{eq:weakEnergy} we use 
\begin{align}
&\rkla{\bs{d}\h\nh\times\Delta\h\bs{d}\h\nh,\sinc{n+1}}\h= \rkla{\nabla\Ih{\bs{d}\h\nh\times\sinc{n+1}},\nabla\bs{d}\h\nh}\notag\\
&~=\rkla{\nabla\Ih{\bs{d}\h\no\times\sinc{n+1}},\nabla\bs{d}\h\no} +\tfrac12\rkla{\nabla\Ih{\bs{d}\h\no\times\sinc{n+1}},\nabla\trkla{\bs{d}\h\nn-\bs{d}\h\no}}\notag\\
&\qquad +\tfrac12\rkla{\nabla\Ih{\trkla{\bs{d}\h\nn-\bs{d}\h\no}\times\sinc{n+1}},\nabla\bs{d}\h\nh}\notag\\
&~=\rkla{\nabla\Ih{\bs{d}\no\times\sinc{n+1}},\nabla\bs{d}\h\no}+\tfrac12\tau\sum_{i=1}^3\sum_{l\in\mathds{Z}\h}\norm{\nabla\Ih{\bs{d}\h\no\times\lambda_{li}\g{li}}}_{L^2\trkla{\D}}^2\notag\\
&\qquad+\tfrac12\tau\sum_{i=1}^3\sum_{l\in\mathds{Z}\h}\rkla{\nabla\Ih{\rkla{\bs{d}\h\nh\times\lambda_{li}\g{li}}\times\lambda_{li}\g{li}},\nabla\bs{d}\h\nh}\notag\\
&\qquad+\tfrac12\tau\rkla{\nabla\Ih{\bs{d}\h\no\times\sinc{n+1}},\nabla\Ih{\trkla{\bs{d}\h\nh\times\Delta\h\bs{d}\h\nn}\times\bs{d}\h\nh}}\notag\\
&\qquad+\tfrac12\tau\rkla{\nabla\Ih{\rkla{\rkla{\bs{d}\h\nh\times\Delta\h\bs{d}\h\nn}\times\bs{d}\h\nh}\times\sinc{n+1}},\nabla\bs{d}\h\nh}\notag\\
&\qquad+\tfrac12 \sum_{i=1}^3\sum_{l\in\mathds{Z}\h}\trkla{\tabs{\incbeta{n+1}{l}{i}}^2-\tau}\norm{\nabla\Ih{\bs{d}\h\no\times\lambda_{li}\g{li}}}_{L^2\trkla{\D}}^2\notag\notag\\
&\qquad+\tfrac12\sum_{\trkla{i,l}\neq\trkla{j,k}}\rkla{\nabla\Ih{\bs{d}\h\no\times\lambda_{li}\g{li}\incbeta{n+1}{l}{i}},\nabla\Ih{\bs{d}\h\no\times\lambda_{kj}\g{kj}\incbeta{n+1}{k}{j}}}\notag\\
&\qquad+\tfrac14\rkla{\nabla\Ih{\bs{d}\h\no\times\sinc{n+1}},\nabla\Ih{\trkla{\bs{d}\h\nn-\bs{d}\h\no}\times\sinc{n+1}}}\notag\\
&\qquad+\tfrac12\sum_{i=1}^3\sum_{l\in\mathds{Z}\h}\trkla{\tabs{\incbeta{n+1}{l}{i}}^2-\tau}\rkla{\nabla\Ih{\rkla{\bs{d}\h\nh\times\lambda_{li}\g{li}}\times\lambda_{li}\g{li}},\nabla\bs{d}\h\nh}\notag\\
&\qquad+\tfrac12\sum_{\trkla{i,l}\neq\trkla{j,k}}\rkla{\nabla\Ih{\rkla{\bs{d}\h\nh\times\lambda_{li}\g{li}\incbeta{n+1}{l}{i}}\times\lambda_{kj}\g{kj}\incbeta{n+1}{k}{j}},\nabla\bs{d}\h\nh}\notag\\
&~=:\rkla{\nabla\Ih{\bs{d}\h\no\times\sinc{n+1}},\nabla\bs{d}\h\no}+\tfrac12\tau\sum_{i=1}^3\sum_{l\in\mathds{Z}\h}\norm{\nabla\Ih{\bs{d}\h\no\times\lambda_{li}\g{li}}}_{L^2\trkla{\D}}^2\notag\\
&\qquad+\tfrac12\tau\sum_{i=1}^3\sum_{l\in\mathds{Z}\h}\rkla{\nabla\Ih{\rkla{\bs{d}\h\nh\times\lambda_{li}\g{li}}\times\lambda_{li}\g{li}},\nabla\bs{d}\h\nh} +\sum_{a=1}^7\mathfrak{B}_a\nn\,.\label{eq:discdiffItoStrat:energy}
\end{align}
Thus, we can write this term as
\begin{align}
\begin{split}
\sum_{n=0}^{N-1}&\eta\no\rkla{\bs{d}\h\nh\times\Delta\h\bs{d}\h\nh,\sinc{n+1}}\h\\
=&\, \sum_{i=1}^3\sum_{l\in\mathds{Z}\h}\int_0^T\eta\tm\rkla{\Ih{\bs{d}\h\tm\times\Delta\h\bs{d}\h\tm},\lambda_{li}\g{li}}\h\mathrm{d}\beta_l^i\\
&\,+\tfrac12\sum_{i=1}^3\sum_{l\in\mathds{Z}\h}\int_0^T\eta\tm\norm{\nabla\Ih{\bs{d}\h\tm\times\lambda_{li}\g{li}}}_{L^2\trkla{\D}}^2  \dt\\
&\,+\tfrac12\sum_{i=1}^3\sum_{l\in\mathds{Z}\h}\int_0^T\eta\tm\rkla{\nabla\Ih{\trkla{\bs{d}\h\ti\times\lambda_{li}\g{li}}\times\lambda_{li}\g{li}},\nabla\bs{d}\h\ti}\dt\\
&\,+\sum_{a=1}^7\sum_{n=0}^{N-1}\eta\no\mathfrak{B}_a\nn\,.
\end{split}
\end{align}
The last term on the right-hand side of \eqref{eq:weakEnergy}, can also be written as a discrete \Ito-integral with correction term by using \eqref{eq:discscheme}:
\begin{align}\label{eq:deffrakA}
 \begin{split}
    -\sum_{n=0}^{N-1}&\,\eta\no\rkla{\Ih{\bs{d}\h\nh\times\sinc{n+1}},\bs{\varphi}\no}\\
    =&\,-\sum_{n=0}^{N-1}\eta\no\rkla{\Ih{\bs{d}\h\no\times\sinc{n+1}},\bs{\varphi}\no}-\tfrac12\sum_{n=0}^{N-1}\eta\no\rkla{\Ih{\rkla{\bs{d}\h\nn-\bs{d}\h\no}\times\sinc{n+1}},\bs{\varphi}\no}\\
    =&\,-\sum_{n=0}^{N-1}\eta\no\rkla{\Ih{\bs{d}\h\no\times\sinc{n+1}},\bs{\varphi}\no}\\
    &\,-\tfrac12\sum_{n=0}^{N-1}\tau\eta\no\sum_{i=1}^3\sum_{l\in\mathds{Z}\h}\rkla{\Ih{\trkla{\bs{d}\h\nh\times\lambda_{li}\g{li}}\times\lambda_{li}\g{li}},\bs{\varphi}\no} \\
    &\,-\tfrac12\sum_{n=1}^{N-1}\eta\no\sum_{i=1}^{3}\sum_{l\in\mathds{Z}\h}\rkla{\tabs{\incbeta{n+1}{l}{i}}^2-\tau}\rkla{\Ih{\trkla{\bs{d}\h\nh\times\lambda_{li}\g{li}}\times\lambda_{li}\g{li}},\bs{\varphi}\no}\\
    &\,-\tfrac12\sum_{n=0}^{N-1}\eta\no\sum_{\trkla{i,l}\neq\trkla{j,k}}\incbeta{n+1}{l}{i}\incbeta{n+1}{k}{j}\rkla{\Ih{\trkla{\bs{d}\h\nh\times\lambda_{li}\g{li}}\times\lambda_{kj}\g{kj}},\bs{\varphi}\no}\\
    &\,-\tfrac12\sum_{n=0}^{N-1}\eta\no\rkla{\tau\Ih{\ekla{\trkla{\bs{d}\h\nh\times\Delta\h\bs{d}\h\nn}\times\bs{d}\h\nh}\times\sinc{n+1}},\bs{\varphi}\no}\\
    =:&\,-\sum_{n=0}^{N-1}\eta\no\rkla{\Ih{\bs{d}\h\no\times\sinc{n+1}},\bs{\varphi}\no}\\
    &\,-\tfrac12\sum_{n=0}^{N-1}\tau\eta\no\sum_{i=1}^3\sum_{l\in\mathds{Z}\h}\rkla{\Ih{\trkla{\bs{d}\h\nh\times\lambda_{li}\g{li}}\times\lambda_{li}\g{li}},\bs{\varphi}\no}\\
    &\,+\tfrac12\sum_{n=0}^{N-1}\eta\no\rkla{\mathfrak{A}_1\nn,\bs{\varphi}\no} +\tfrac12\sum_{n=0}^{N-1}\eta\no\rkla{\mathfrak{A}_2\nn,\bs{\varphi}\no} +\tfrac12\sum_{n=0}^{N-1}\eta\no\rkla{\mathfrak{A}_3\nn,\bs{\varphi}\no}\,.
     \end{split}
\end{align}
Combining the above transformations, multiplying the result by $\psi\in L^\infty\trkla{\Omega}$, and taking the expectation, we find that
\begin{align}
-&\expected{\psi\int_0^T\partial_t\eta\tl\ekla{\tfrac12\norm{\nabla\bs{d}\h\tp}_{L^2\trkla{\D}}^2-\trkla{\bs{d}\h\tp,\bs{\varphi}\tp}}\dt} \notag\\
&-\expected{\psi\eta\trkla{0}\ekla{\tfrac12\norm{\nabla\bs{d}\h^0}_{L^2\trkla{\D}}^2-\trkla{\bs{d}\h^0,\bs{\varphi}\trkla{0}}}}\notag\\
&+\expected{\psi\int_0^T\eta\tm\skla{A,\bs{d}\h\tp}\dt} +\sum_{i=1}^3\sum_{l\in\mathds{Z}}\expected{\psi\int_0^T\eta\tm\skla{B_{li}\g{li},\bs{d}\h\tm}\mathrm{d}\beta_l^i}\notag\\
&+\sum_{i=1}^3\sum_{l\in\mathds{Z}\h}\expected{\psi\int_0^T\eta\tm\skla{B_{li}\trkla{t}\g{li},\Ih{\bs{d}\h\ti\times\lambda_{li}\g{li}}}\dt}\notag\\
&+\expected{\psi\int_0^T\eta\tm\norm{\bs{d}\h\ti\times\Delta\h\bs{d}\h\ti}_h^2\dt}\notag\\
&+\expected{\psi\int_0^T\eta\tm\trkla{\Ih{\trkla{\bs{d}\h\ti\times\Delta\h\bs{d}\h\ti}\times\bs{d}\h\ti},\bs{\varphi}\tm}\dt}\notag\\
=&\,\sum_{i=1}^3\sum_{l\in\mathds{Z}\h}\expected{\psi\int_0^T\eta\tm\rkla{\Ih{\bs{d}\h\tm\times\Delta\h\bs{d}\h\tm},\lambda_{li}\g{li}}\h\mathrm{d}\beta_l^i}\notag\\
&+\tfrac12\sum_{i=1}^3\sum_{l\in\mathds{Z}\h}\expected{\psi\int_0^T\eta\tm\norm{\nabla\Ih{\bs{d}\h\tm\times\lambda_{li}\g{li}}}_{L^2\trkla{\D}}^2\dt}\notag\\
&+\tfrac12\sum_{i=1}^3\sum_{l\in\mathds{Z}\h}\expected{\psi\int_0^T    \eta\tm \rkla{\nabla\Ih{\trkla{\bs{d}\h\ti\times\lambda_{li}\g{li}}\times\lambda_{li}\g{li}},\nabla\bs{d}\h\ti}\dt}\notag\\
&-\sum_{i=1}^3\sum_{l\in\mathds{Z}\h}\expected{\psi\int_0^T\eta\tm\rkla{\Ih{\bs{d}\h\tm\times\lambda_{li}\g{li}},\bs{\varphi}\tm}\mathrm{d}\beta_l^i}\notag\\
&-\tfrac12\sum_{i=1}^3\sum_{l\in\mathds{Z}\h}\expected{\psi\int_0^T\eta\tm\rkla{\Ih{\trkla{\bs{d}\h\ti\times\lambda_{li}\g{li}}\times\lambda_{li}\g{li}},\bs{\varphi}\tm}\dt}\notag\\
&+\expected{\psi \mathfrak{E}_1} +\expected{\psi \mathfrak{E}_2} +\expected{\psi \mathfrak{E}_3}+\expected{\psi \mathfrak{E}_4}+\sum_{a=1}^7\expected{\psi \sum_{n=0}^{N-1}\eta\no\mathfrak{B}_a\nn}\notag\\
&+\sum_{a=1}^3\tfrac12\expected{\psi\sum_{n=0}^{N-1}\eta\no\trkla{\mathfrak{A}_a\nn,\bs{\varphi}\no}} - \expected{\psi \mathfrak{S}_1} -\expected{\psi\mathfrak{S}_2}\,.\label{eq:weakEnergy2}
\end{align}
In the next section, we shall discuss the error terms on the right-hand side of \eqref{eq:weakEnergy2}.

\subsection{Treatment of the residual terms}\label{sec:trafoerror}
We start by discussing the the terms related to $\mathfrak{S}_1$ and $\mathfrak{S}_2$.
\begin{lemma}\label{lem:diff}
Let assumptions \ref{item:T}, \ref{item:S}, \ref{item:I}, \ref{item:W1}, \ref{item:W2}, and \ref{item:W3} hold true.
Then, the estimate
\begin{align}
    \expected{\rkla{\sum_{n=0}^{N-1}\tau\norm{\Ih{\bs{d}\h\nh\times\Delta\h\trkla{\bs{d}\h\nh-\bs{d}\h\nn}}}_{L^2\trkla{\D}}^2}^p}\leq&\, C\rkla{\frac{\tau}{h^2}}^p\,\label{eq:diff3}
\end{align}
holds true for all $p\in[1,\infty)$ with a constant $C$ depending on $p$ but not on $h$ or $\tau$.
\end{lemma}
\begin{proof}
Combining Hölder's inequality and an inverse estimate, we obtain
\begin{align}
\begin{split}
&\expected{\rkla{\sum_{n=0}^{N-1}\tau\norm{\Ih{\bs{d}\h\nh\times\Delta\h\trkla{\bs{d}\h\nn-\bs{d}\h\no}}}_{L^2\trkla{\D}}^2}^p}\\
&~\leq C\rkla{\frac{\tau}{h^2}}^p\expected{\rkla{\sum_{n=0}^{N-1}\norm{\nabla\trkla{\bs{d}\h\nn-\bs{d}\h\no}}_{L^2\trkla{\D}}^2}^p}\leq C\rkla{\frac{\tau}{h^2}}^p\,.
\end{split}
\end{align}
This provides the result.
\end{proof}
Straightforward computations now show that also $\mathfrak{S}_1$ and $\mathfrak{S}_2$ vanish.
\begin{corollary}\label{cor:diffS}
Let the assumptions of Lemma \ref{lem:diff} hold true. Then 
    \begin{align}
         \expected{\abs{\mathfrak{S}_1}^p}&\leq C\rkla{\frac{\tau}{h^2}}^{p/2} \label{eq:diffS:1}\,, && \text{and}&\expected{\abs{\mathfrak{S}_2}^q}&\leq C\rkla{\frac{\tau}{h^2}}^{q/2} 
    \end{align}
    for any $p\in[1,\infty)$ and $q\in[1,\p/2)$ with a constant $C$ depending on $p$, $q$ but not on $h$ or $\tau$.
\end{corollary}

\begin{proof}
    To show the first estimate in \eqref{eq:diffS:1}, we combine Hölder's inequality and Lemma \ref{lem:diff} to obtain
\begin{align}
\begin{split}
&\expected{\max_{1\leq M\leq N}\abs{\sum_{n=0}^{M-1}\tau\rkla{\bs{d}\h\nh\times\Delta\h\bs{d}\h\nh,\bs{d}\h\nh\times\Delta_h\trkla{\bs{d}\h\nn-\bs{d}\h\nh}}\h}^p}\\
&\quad\leq C\mathds{E}\left[\rkla{\sum_{n=0}^{N-1}\tau\norm{\bs{d}\h\nh\times\Delta\h\bs{d}\h\nn}\h^2 +\sum_{n=0}^{N-1}\tau\norm{\bs{d}\h\nh\times\Delta\h\trkla{\bs{d}\h\nn-\bs{d}\h\nh}}\h^2}^{p/2}\right. \\
&\qquad\qquad\qquad\cdot\left.\rkla{\sum_{n=0}^{N-1}\tau\norm{\bs{d}\h\nh\times\Delta\h\rkla{\bs{d}\h\nn-\bs{d}\h\nh}}\h^2}^{p/2} \right]\\
&\quad\leq C\rkla{\frac{\tau}{h^2}}^{p/2}\,.
\end{split}
\end{align}
The second estimate in \eqref{eq:diffS:1}, can be shown with similar arguments
\begin{align}
\begin{split}
    \expected{\abs{\mathfrak{S}_2}^q}\leq &\,\expected{\rkla{\sum_{n=0}^{N-1}\tau\abs{\eta\no}^2\norm{\Ih{\bs{d}\h\nh\times\Delta\h\trkla{\bs{d}\h\nn-\bs{d}\h\nh}}}_{L^2\trkla{\D}}^2}^q}^{1/2}\\
    &\,\cdot\expected{\rkla{\sum_{n=0}^{N-1}\tau\norm{\bs{\varphi}\no}_{L^2\trkla{\D}}^2}^q}^{1/2}\leq C\rkla{\frac{\tau}{h^2}}^{q/2}\,.
    \end{split}
\end{align}
\end{proof}

As the following lemma shows, the residual terms $\mathfrak{A}\nn_1$, $\mathfrak{A}\nn_2$, and $\mathfrak{A}\nn_3$ will also vanish for $\tau\searrow0$:
\begin{lemma}\label{lem:stochasticerror}
Let assumptions \ref{item:T}, \ref{item:S}, \ref{item:I}, \ref{item:W1}, \ref{item:W2}, and \ref{item:W3} hold true.
Then, the residual terms $\trkla{\mathfrak{A}_1\nn,\mathfrak{A}_2\nn,\mathfrak{A}_3\nn}_n$ satisfy the estimates
\begin{subequations}\label{eq:stochasticerror}
\begin{align}
\expected{\sum_{n=0}^{N-1}\norm{\mathfrak{A}_1\nn}_{L^2\trkla{\D}}^2} +\max_{1\leq M\leq N}\expected{\norm{\sum_{n=0}^{M-1}\mathfrak{A}_1\nn}_{L^2\trkla{\D}}^2}\leq &\,C\tau\,,\label{eq:stochasticerror:A1}\\
\expected{\sum_{n=0}^{N-1}\norm{\mathfrak{A}_2\nn}_{L^2\trkla{\D}}^2} +\max_{1\leq M\leq N}\expected{\norm{\sum_{n=0}^{M-1}\mathfrak{A}_2^{n+1}}_{L^2\trkla{\D}}^2}\leq &\,C\tau\,,\label{eq:stochasticerror:A2}\\
\expected{\sum_{n=0}^{N-1}\tau\norm{\mathfrak{A}\nn_3}_{L^2\trkla{\D}}^{p}}\leq&\,C \tau^{1/2+p}\label{eq:stochasticerror:A3}
\end{align}
\end{subequations}
with a constant $C$ which is independent of $h$ and $\tau$.
\end{lemma}
\begin{proof}
To control $\mathfrak{A}_1$, we write
\begin{align}
\begin{split}
&\expected{\sum_{n=0}^{N-1}\norm{\mathfrak{A}_1\nn}_{L^2\trkla{\D}}^2}\\
&\leq C\sum_{n=0}^{N-1}\expected{\norm{\sum_{i=1}^3\sum_{l\in\mathds{Z}\h}\abs{\lambda_{li}}^2\rkla{\tabs{\incbeta{n+1}{l}{i}}^2-\tau}\Ih{\rkla{\bs{d}\h\no\times\g{li}}\times\g{li}}}_{L^2\trkla{\D}}^2}\\
&\quad+C\sum_{n=0}^{N-1}\expected{\norm{\sum_{i=1}^3\sum_{l\in\mathds{Z}\h}\abs{\lambda_{li}}^2\rkla{\tabs{\incbeta{n+1}{l}{i}}^2-\tau}\Ih{\rkla{\trkla{\bs{d}\h\nn-\bs{d}\h\no}\times\g{li}}\times\g{li}}}_{L^2\trkla{\D}}^2}\\
&=:C\sum_{n=0}^{N-1}\expected{\norm{\widetilde{\mathfrak{A}}_{1,1}\nn}_{L^2\trkla{\D}}^2}+C\sum_{n=0}^{N-1}\expected{\norm{\widetilde{\mathfrak{A}}_{1,2}\nn}_{L^2\trkla{\D}}^2}\,.
\end{split}
\end{align}
Due to the independence of the increments $\incbeta{n+1}{l}{i}$, \eqref{eq:lengthconstraint}, and \ref{item:W3}, we obtain
\begin{align}
\sum_{n=0}^{N-1}\expected{\norm{\widetilde{\mathfrak{A}}_{1,1}\nn}_{L^2\trkla{\D}}^2}\leq CN\tau^2\leq C\tau\,.
\end{align}
For the second term, we compute using Hölder's inequality, \ref{item:W3}, and Lemma \ref{lem:Nikolskii}
\begin{align}\label{eq:tmp:A1:2}
\begin{split}
&\sum_{n=0}^{N-1}\expected{\norm{\widetilde{\mathfrak{A}}_{1,2}\nn}_{L^2\trkla{\D}}^2}\\
&\leq C\sum_{n=0}^{N-1}\expected{\rkla{\sum_{i=1}^3\sum_{l\in\mathds{Z}\h} \norm{\tabs{\lambda_{li}}^2\trkla{\tabs{\incbeta{n+1}{l}{i}}^2-\tau}\Ih{\rkla{\trkla{\bs{d}\h\nn-\bs{d}\h\no}\times\g{li}}\times\g{li}}}_{L^2\trkla{\D}}}^2}\\
&\leq C\sum_{n=0}^{N-1}\sum_{i,j=1}^3\sum_{k,l\in\mathds{Z}\h}\mathds{E}\left[\norm{\bs{d}\h\nn-\bs{d}\h\no}_{L^2\trkla{\D}}^2\right.\\
&\qquad\qquad\qquad\cdot\left.\abs{\tabs{\incbeta{n+1}{l}{i}}^2-\tau}\abs{\tabs{\incbeta{n+1}{k}{j}}^2-\tau}\norm{\lambda_{li}\g{li}}_{L^\infty\trkla{\D}}^2\norm{\lambda_{kj}\g{kj}}_{L^\infty\trkla{\D}}^2\right]\\
&\leq C\sum_{n=0}^{N-1}\sum_{i,j=1}^3\sum_{k,l\in\mathds{Z}\h}\norm{\lambda_{li}\g{li}}_{L^\infty\trkla{\D}}^2\norm{\lambda_{kj}\g{kj}}_{L^\infty\trkla{\D}}^2\expected{\norm{\bs{d}\h\nn-\bs{d}\h\no}_{L^2\trkla{\D}}^4}^{1/2}\\
&\qquad\qquad\qquad\qquad\qquad\qquad\qquad\cdot\expected{\abs{\tabs{\incbeta{n+1}{l}{i}}^2-\tau}^4}^{1/4}\expected{\abs{\tabs{\incbeta{n+1}{k}{j}}^2-\tau}^4}^{1/4}\\
&\leq C\sum_{n=0}^{N-1}\tau^2\expected{\norm{\bs{d}\h\nn-\bs{d}\h\no}_{L^2\trkla{\D}}^4}^{{1/2}}\leq C\tau\rkla{\sum_{n=0}^{{N-1}}\tau\expected{\norm{\bs{d}\h\nn-\bs{d}\h\no}_{L^2\trkla{\D}}^4}}^{{1/2}}\leq C\tau^{{2}}\,,
\end{split}
\end{align}
which provides the first part of \eqref{eq:stochasticerror:A1}. To verify the second claim in \eqref{eq:stochasticerror:A1}, we use
\begin{multline}
\max_{1\leq M\leq N}\expected{\norm{\sum_{n=0}^{M-1}\mathfrak{A}_1^{n+1}}_{L^2\trkla{\D}}^2}\\
\leq\max_{1\leq M\leq N}\expected{\norm{\sum_{n=0}^{M-1}\widetilde{\mathfrak{A}}_{1,1}^{n+1}}_{L^2\trkla{\D}}^2}+\max_{1\leq M\leq N}\expected{\norm{\sum_{n=0}^{M-1}\widetilde{\mathfrak{A}}_{1,2}^{n+1}}_{L^2\trkla{\D}}^2}\,.
\end{multline}
Following the lines of \cite[Lemma 4.3]{Banas2013} (see also \cite[Lemma 2.23]{BanasBrzezniakNeklyudovProhl2014}), we infer that $\tgkla{\widetilde{\mathfrak{A}}\nn_{1,1}}_{n=0}^{N-1}$ is an $\trkla{L^2\trkla{\D}}^3$-valued martingale difference with respect to the filtration $\mathcal{F}_t$.
Thus, the estimate
\begin{align}\label{eq:martingale:difference}
\begin{split}
\expected{\norm{\sum_{n=0}^{M-1}\widetilde{\mathfrak{A}}_{1,1}^{n+1}}_{L^2\trkla{\D}}^2}=\sum_{n=0}^{M-1}\expected{\norm{\widetilde{\mathfrak{A}}\nn_{1,1}}_{L^2\trkla{\D}}^2}\leq CM\tau^2\leq C\tau
\end{split}
\end{align}
holds true for all $M\leq N$.
The second term satisfies the estimate
\begin{align}
\begin{split}
\expected{\norm{\sum_{n=0}^{M-1}\widetilde{\mathfrak{A}}_{1,2}^{n+1}}_{L^2\trkla{\D}}^2}\leq CM\sum_{n=0}^{M-1}\expected{\norm{\widetilde{\mathfrak{A}}_{1,2}\nn}_{L^2\trkla{\D}}^2}\leq C\tau
\end{split}
\end{align}
due to \eqref{eq:tmp:A1:2}.

To establish the estimate \eqref{eq:stochasticerror:A2}, we use similar arguments and decompose $\bs{d}\h\nh$ into $\bs{d}\h\no$ and $\tfrac12\trkla{\bs{d}\h\nn-\bs{d}\h\no}$. 
Using the independence of the stochastic increments, we obtain
\begin{align}\label{eq:tmp:A2:1}
\begin{split}
&\expected{\norm{\sum_{n=0}^{M-1}\sum_{\trkla{i,l}\neq\trkla{j,k}}\lambda_{li}\lambda_{kj}\incbeta{n+1}{l}{i}\incbeta{n+1}{k}{j}\Ih{\rkla{\bs{d}\h\no\times\g{kj}}\times\g{li}}}^2_{L^2\trkla{\D}}}\\
&=\sum_{n=0}^{M-1}\sum_{\trkla{i,l}\neq\trkla{j,k}}\expected{\abs{\lambda_{li}\lambda_{kj}}^2\rkla{\incbeta{n+1}{l}{i}\incbeta{n+1}{k}{j}}^2\norm{\Ih{\rkla{\bs{d}\h\no\times\g{kj}}\times\g{li}}}_{L^2\trkla{\D}}^2}\\
&\quad+\sum_{n=0}^{M-1}\sum_{\trkla{i,l}\neq\trkla{j,k}}\expected{\abs{\lambda_{li}\lambda_{kj}}^2\rkla{\incbeta{n+1}{l}{i}\incbeta{n+1}{k}{j}}^2\rkla{\Ih{\rkla{\bs{d}\h\no\times\g{kj}}\times\g{li}}, \Ih{\rkla{\bs{d}\h\no\times\g{li}}\times\g{kj}}}}\\
&\leq C\tau\,.
\end{split}
\end{align}
To deal with the second component of $\mathfrak{A}_2\nn$, we need the estimate
\begin{align}\label{eq:tmp:A2:2}
\begin{split}
&\expected{\norm{\Ih{\rkla{\trkla{\bs{d}\h\nn-\bs{d}\h\no}\times\sinc{n+1}}\times\sinc{n+1}}}_{L^2\trkla{\D}}^2}\\
&\quad\leq C\expected{\norm{\bs{d}\h\nn-\bs{d}\h\no}_{L^2\trkla{\D}}^2\norm{\sinc{n+1}}_{L^\infty\trkla{\D}}^4} \\
&\quad\leq C\expected{\norm{\bs{d}\h\nn-\bs{d}\h\no}_{L^2\trkla{\D}}^4}^{1/2}\expected{\norm{\sinc{n+1}}_{H^2\trkla{\D}}^8}^{1/2}\leq C\tau^2\expected{\norm{\bs{d}\h\nn-\bs{d}\h\no}_{L^2\trkla{\D}}^4}^{1/2}\,,
\end{split}
\end{align}
which is a consequence of Hölder's inequality, standard Sobolev embeddings, the Burkholder--Davis--Gundy inequality, and assumption \ref{item:W3}.
Thus, rearranging the terms in $\mathfrak{A}\nn_2$, applying the estimates \eqref{eq:tmp:A2:1} and \eqref{eq:tmp:A2:2}, and arguing as in \eqref{eq:tmp:A1:2} provides

\begin{align}
\begin{split}
&\expected{\norm{\sum_{n=0}^{M-1}\mathfrak{A}\nn_2}_{L^2\trkla{\D}}^2}\\
&\quad\leq C\expected{\norm{\sum_{n=0}^{M-1}\sum_{\trkla{i,l}\neq\trkla{j,k}}\lambda_{li}\lambda_{kj}\incbeta{n+1}{l}{i}\incbeta{n+1}{k}{j}\Ih{\rkla{\bs{d}\h\no\times\g{kj}}\times\g{li}}}_{L^2\trkla{\D}}^2}\\
&\qquad\quad +C\expected{\norm{\sum_{n=0}^{M-1}\sum_{i=1}^3\sum_{l\in\mathds{Z}\h}\tabs{\lambda_{li}}^2\tabs{\incbeta{n+1}{l}{i}}^2\Ih{\rkla{\trkla{\bs{d}\h\nn-\bs{d}\h\no}\times\g{li}}\times\g{li}}}_{L^2\trkla{\D}}^2}\\
&\qquad\quad+C M \sum_{n=0}^{M-1}\expected{\norm{\Ih{\rkla{\trkla{\bs{d}\h\nn-\bs{d}\h\no}\times\sinc{n+1}}\times\sinc{n+1}}}_{L^2\trkla{\D}}^2}\\
&\quad\leq C\tau +C\trkla{M\tau}\sum_{n=0}^{M-1}\tau\expected{\norm{\bs{d}\h\nn-\bs{d}\h\no}_{L^2\trkla{\D}}^4}^{1/2}\leq {C\tau}\,,
\end{split}
\end{align}
due to Lemma \ref{lem:Nikolskii}.
A straightforward modification of this argument shows
\begin{align}
\sum_{n=0}^{N-1}\expected{\norm{\mathfrak{A}_2\nn}_{L^2\trkla{\D}}^2}\leq C\tau\,.
\end{align}
Finally, for $\mathfrak{A}\nn_3$ we obtain using Young's inequality, the Burkholder--Davis--Gundy inequality, and \eqref{eq:lengthconstraint}
\begin{align*}
\expected{\sum_{n=0}^{N-1}\tau^{1/2}\norm{\mathfrak{A}\nn_3}_{L^2\trkla{\D}}^{p}}\leq&\, C\sum_{n=0}^{N-1}\tau^{2p}\expected{\norm{\Ih{\rkla{\bs{d}\h\nh\times\Delta\h\bs{d}\h\nn}\times\bs{d}\h\nh}}_{L^2\trkla{\D}}^{2p}}\\
&\, +C\sum_{n=0}^{N-1}\tau^{1/2}\tau^{1/2}\expected{\norm{\sinc{n+1}}_{H^2\trkla{\D}}^{2p}}\\
\leq&\, C\tau^p\expected{\rkla{\sum_{n=0}^{N-1}\tau\norm{\Ih{\bs{d}\h\nh\times\Delta\h\bs{d}\h\nn}}_{L^2\trkla{\D}}^2}^p}+C\sum_{n=0}^{N-1}\tau^{p+1}\\
\leq&\,C\tau^p\,.
\end{align*}
\end{proof}

\begin{lemma}\label{lem:B}
Let assumptions \ref{item:T}, \ref{item:S}, \ref{item:I}, \ref{item:W1}, \ref{item:W2}, and \ref{item:W3} hold true. Then the residual terms in \eqref{eq:discdiffItoStrat:energy} satisfy
\begin{subequations}\label{eq:B}
\begin{align}
\expected{\rkla{\sum_{n=0}^{N-1}\abs{\mathfrak{B}_1\nn}}^p}+\expected{\rkla{\sum_{n=0}^{N-1}\abs{\mathfrak{B}_2\nn}}^p}\leq& \,C\rkla{\frac{\tau^{1-\widehat{\epsilon}}}{h^2}}^p+C\tau^{\widehat{\epsilon}p}\,,\label{eq:B1B2}\\
\max_{1\leq M\leq N}\expected{\abs{\sum_{n=0}^{M-1}\mathfrak{B}_3\nn}^2}+\max_{1\leq M\leq N}\expected{\abs{\sum_{n=0}^{M-1}\mathfrak{B}_4\nn}^2}\leq&\,C\tau\,,\label{eq:B3B4}\\
\expected{\rkla{\sum_{n=0}^{N-1}\abs{\mathfrak{B}_5\nn}}^p}\leq&\,C\rkla{\frac{\tau^{1/2}}{h}}^p\,,\label{eq:B5}\\
\max_{1\leq M\leq N}\expected{\abs{\sum_{n=0}^{M-1}\mathfrak{B}_6\nn}^2}\leq&\,C\tau^{1/2}\,,\label{eq:B6}\\
\max_{1\leq M\leq N}\expected{\abs{\sum_{n=0}^{M-1}\mathfrak{B}_7\nn}^2}\leq&\,C\tau^{1/3}\,,\label{eq:B7}
\end{align}
\end{subequations}
with $0<\widehat{\epsilon}<1$.
\end{lemma}
\begin{proof}
Applying Young's inequality, an inverse estimate, the $L^\infty$-bound on $\bs{d}\h\nh$, the Burkholder--Davis--Gundy inequality, assumption \ref{item:W3}, and the regularity results obtained in Lemma \ref{lem:energy} yields
\begin{align}
\begin{split}
&\expected{\rkla{\sum_{n=0}^{N-1}\abs{\mathfrak{B}_1\nn}}^p}\\
&~\leq C\expected{\rkla{\sum_{n=0}^{N-1}\tau^{1-\widehat{\epsilon}}h^{-2}\norm{\nabla\Ih{\bs{d}\h\no\times\sinc{n+1}}}_{L^2\trkla{\D}}^2}^p}\\
&\qquad+C\expected{\rkla{\sum_{n=0}^{N-1}\tau^{1+\widehat{\epsilon}}h^2\norm{\nabla\Ih{\rkla{\bs{d}\h\nh\times\Delta\h\bs{d}\h\nn}\times\bs{d}\h\nh}}_{L^2\trkla{\D}}^2}^p}\\
&~\leq C\expected{\rkla{\sum_{n=0}^{N-1}\tau^{1-\widehat{\epsilon}}h^{-2}\norm{\nabla\Ih{\bs{d}\h\no\times\sinc{n+1}}}_{L^2\trkla{\D}}^2}^p}\\
&\qquad+C\expected{\rkla{\sum_{n=0}^{N-1}\tau^{1+\widehat{\epsilon}}\norm{
\bs{d}\h\nh\times\Delta\h\bs{d}\h\nn}
\h^2}^p}\\
&~\leq C\rkla{\frac{\tau^{1-\widehat{\epsilon}}}{h^2}}^p\sum_{n=0}^{N-1}\tau\expected{\rkla{\sum_{i=1}^3\sum_{l\in\mathds{Z}\h}\norm{\nabla\Ih{\bs{d}\h\no\times\lambda_{li}\g{li}}}_{L^2\trkla{\D}}^2}^p} +C\tau^{\widehat{\epsilon}p}\\
&~\leq C\rkla{\frac{\tau^{1-\widehat{\epsilon}}}{h^2}}^p+C\tau^{\widehat{\epsilon}p}\,.
\end{split}
\end{align}
Here, we used the estimate
\begin{align}\label{eq:tmp:nablaIh}
\begin{split}
\norm{\nabla\Ih{a\h b\h}}_{L^2\trkla{\D}}\leq&\, \norm{\nabla\trkla{a\h b\h}}_{L^2\trkla{\D}} +\norm{\nabla\trkla{\id-\Ihop}\gkla{a\h b\h}}_{L^2\trkla{\D}} \\
\leq&\, C\norm{a\h}_{H^1\trkla{\D}}\norm{b\h}_{W^{1,\infty}\trkla{\D}}\,
\end{split}
\end{align}
in the last step.
To verify the second part of \eqref{eq:B1B2}, we use a similar strategy and obtain
\begin{align}
\begin{split}
&\expected{\rkla{\sum_{n=0}^{N-1}\abs{\mathfrak{B}_2\nn}}^p}\\
&~\leq C\expected{\rkla{\sum_{n=0}^{N-1}\tau^{1+\widehat{\epsilon}}\norm{\Ih{\rkla{\bs{d}\h\nh\times\Delta\h\bs{d}\h\nn}\times\bs{d}\h\nh}}\h^2}^p}\\
&\qquad+C\expected{\rkla{\sum_{n=0}^{N-1}\tau^{1-\widehat{\epsilon}}\norm{\Delta\h\bs{d}\h\no\times\sinc{n+1}}\h^2}^p}\\
&\qquad+C\expected{\rkla{\sum_{n=0}^{N-1}\tau^{1-\widehat{\epsilon}}\norm{\Delta\h\trkla{\bs{d}\h\nn-\bs{d}\h\no}\times\sinc{n+1}}\h^2}^p}\\
&~\leq C\tau^{\widehat{\epsilon}p}+C\tau^{\trkla{1-\widehat{\epsilon}}p}\sum_{n=0}^{N-1}\tau\expected{\rkla{\sum_{i=1}^3\sum_{l\in\mathds{Z}\h}\norm{\Delta\h\bs{d}\h\no\times\lambda_{li}\g{li}}\h^2}^p}\\
&\qquad+C\rkla{\frac{\tau^{1-\widehat{\epsilon}}}{h^2}}^p\expected{\rkla{\sum_{n=0}^{N-1}\norm{\nabla\trkla{\bs{d}\h\nn-\bs{d}\h\no}}_{L^2\trkla{\D}}^2\norm{\sinc{n+1}}_{H^2\trkla{\D}}^2}^p}\\
&~\leq C\tau^{\widehat{\epsilon}p}+C\rkla{\frac{\tau^{1-\widehat{\epsilon}}}{h^2}}^p\,,
\end{split}
\end{align}
due to
\begin{align}
\begin{split}
&\expected{\rkla{\sum_{n=0}^{N-1}\norm{\nabla\bs{d}\h\nn-\nabla\bs{d}\h\no}_{L^2\trkla{\D}}^2\norm{\sinc{n+1}}_{H^2\trkla{\D}}^2}^p}^2\\
&\,\leq \expected{\max_{0\leq n\leq N-1}\norm{\bs{d}\h\nn-\bs{d}\h\no}_{H^1\trkla{\D}}^{2p}\rkla{\sum_{n=0}^{N-1}\norm{\sinc{n+1}}_{H^2\trkla{\D}}^2}^p}^2\\
&\,\leq \expected{\rkla{\sum_{n=0}^{N-1}\norm{\bs{d}\h\nn-\bs{d}\h\no}_{H^1\trkla{\D}}^2}^{2p}}\expected{\rkla{\sum_{n=0}^{N-1}\norm{\sinc{n+1}}_{H^2\trkla{\D}}}^{2p}}\leq C\,.
\end{split}
\end{align}
To deal with $\mathfrak{B}_3\nn$, we use that $\tgkla{\rkla{\tabs{\incbeta{n+1}{l}{i}}^2-\tau}}_{n,i,l}$ are mutually independent and satisfy $\expected{\rkla{\tabs{\incbeta{n+1}{l}{i}}^2-\tau}}=0$ to obtain
\begin{align}
\begin{split}
\expected{\abs{\sum_{n=0}^{M-1}\mathfrak{B}_3\nn}^2}\leq&\,C\sum_{n=0}^{N-1}\sum_{i=1}^3\sum_{l\in\mathds{Z}\h}\expected{\rkla{\tabs{\incbeta{n+1}{l}{i}}^2-\tau}^2\norm{\nabla\Ih{\bs{d}\h\no\times\lambda_{li}\g{li}}}_{L^2\trkla{\D}}^4}\\
\leq&\,C\tau\expected{\max_{0\leq n\leq N}\norm{\bs{d}\h\no}_{H^1\trkla{\D}}^2\sum_{n=0}^{N-1}\tau\norm{\bs{d}\h\no}_{H^1\trkla{\D}}^2}\leq C\tau\,.
\end{split}
\end{align}
For $\mathfrak{B}_4\nn$ the independence of the increments provides
\begin{align}
\begin{split}
\expected{\abs{\sum_{n=0}^{M-1}\mathfrak{B}_4\nn}^2}\leq&\,C\sum_{n=0}^{M-1}\sum_{\trkla{i,l}\neq\trkla{j,k}}\tau^2\expected{\norm{\nabla\Ih{\bs{d}\h\no\times\lambda_{li}\g{li}}}_{L^2\trkla{\D}}^2\norm{\nabla\Ih{\bs{d}\h\no\times\lambda_{ki}\g{kj}}}_{L^2\trkla{\D}}^2}\\
\leq&\,C\tau\expected{\max_{0\leq n\leq N}\norm{\bs{d}\h\no}_{H^1\trkla{\D}}^2\sum_{n=0}^{N-1}\tau\norm{\bs{d}\h\no}_{H^1\trkla{\D}}^2}\leq C\tau\,.
\end{split}
\end{align}
To establish \eqref{eq:B5}, we use Hölder's inequality, an inverse estimate, standard Sobolev embeddings, and Young's inequality to deduce
\begin{multline*}
\abs{\mathfrak{B}_5\nn}\leq C\frac{\tau^{1/2}}{h}\tau^{-1/4}\norm{\nabla\Ih{\bs{d}\h\no\times\sinc{n+1}}}_{L^2\trkla{\D}}\norm{\bs{d}\h\nn-\bs{d}\h\no}_{H^1\trkla{\D}}\tau^{-1/4}\norm{\sinc{n+1}}_{H^1\trkla{\D}}\\
\leq C\frac{\tau^{1/2}}{h}\rkla{\tau^{-1}\norm{\nabla\Ih{\bs{d}\h\no\times\sinc{n+1}}}_{L^2\trkla{\D}}^4+\norm{\bs{d}\h\nn-\bs{d}\h\no}_{H^1\trkla{\D}}^2+\tau^{-1}\norm{\Ih{\sinc{n+1}}}_{H^1\trkla{\D}}^4}\,.
\end{multline*}
Thus, an application of the Burkholder--Davis--Gundy inequality provides
\begin{align}
\begin{split}
\expected{\rkla{\sum_{n=0}^{N-1}\abs{\mathfrak{B}_5\nn}}^p}\leq&\,C\rkla{\frac{\tau^{1/2}}{h}}^pN^{p-1}\tau^{-p}\sum_{n=0}^{N-1}\expected{\norm{\nabla\Ih{\bs{d}\h\no\times\sinc{n+1}}}_{L^2\trkla{\D}}^{4p}}\\
&\,+C\rkla{\frac{\tau^{1/2}}{h}}^pN^{p-1}\tau^{-p}\sum_{n=0}^{N-1}\expected{\norm{\Ih{\sinc{n+1}}}_{H^1\trkla{\D}}^{4p}}\\
&\,+C\rkla{\frac{\tau^{1/2}}{h}}^p\expected{\rkla{\sum_{n=0}^{N-1}\norm{\bs{d}\h\nn-\bs{d}\h\no}_{H^1\trkla{\D}}^2}^p}\\
\leq&\,C\rkla{\frac{\tau^{1/2}}{h}}^p\,.
\end{split}
\end{align}
To establish \eqref{eq:B6}, we first write
\begin{align}
\begin{split}
\mathfrak{B}_6\nn=&\,\tfrac12\sum_{i=1}^3\sum_{l\in\mathds{Z}\h}\trkla{\tabs{\incbeta{n+1}{l}{i}}^2-\tau}\rkla{\nabla\Ih{\rkla{\bs{d}\h\nh\times\lambda_{li}\g{li}}\times\lambda_{li}\g{li}},\nabla\bs{d}\h\nh}\\
=&\,\tfrac12\sum_{i=1}^3\sum_{l\in\mathds{Z}\h}\trkla{\tabs{\incbeta{n+1}{l}{i}}^2-\tau}\rkla{\nabla\Ih{\rkla{\bs{d}\h\no\times\lambda_{li}\g{li}}\times\lambda_{li}\g{li}},\nabla\bs{d}\h\no}\\
&\,+\tfrac14\sum_{i=1}^3\sum_{l\in\mathds{Z}\h}\trkla{\tabs{\incbeta{n+1}{l}{i}}^2-\tau}\rkla{\nabla\Ih{\rkla{\bs{d}\h\no\times\lambda_{li}\g{li}}\times\lambda_{li}\g{li}},\nabla\trkla{\bs{d}\h\nn-\bs{d}\h\no}}\\
&\,+\tfrac14\sum_{i=1}^3\sum_{l\in\mathds{Z}\h}\trkla{\tabs{\incbeta{n+1}{l}{i}}^2-\tau}\rkla{\nabla\Ih{\rkla{\trkla{\bs{d}\h\nn-\bs{d}\h\no}\times\lambda_{li}\g{li}}\times\lambda_{li}\g{li}},\nabla\bs{d}\h\nh}\\
=:&\,\mathfrak{B}\nn_{6,1} + \mathfrak{B}\nn_{6,2} + \mathfrak{B}\nn_{6,3}\,.
\end{split}
\end{align}
To control the last term, we use Hölder's inequality, \eqref{eq:tmp:nablaIh}, Assumption \ref{item:W3}, Young's inequality, and the regularity results established in Lemma \ref{lem:energy}:
\begin{align}\label{sumestimate}
\begin{split}
\expected{\abs{\sum_{n=0}^{M-1}\mathfrak{B}_{6,3}\nn}^2}\leq&\,CM\mathds{E}\left[
\sum_{n=0}^{M-1}\!\sum_{i,j=1}^3\!\sum_{k,l\in\mathds{Z}\h}\trkla{\tabs{\incbeta{n+1}{l}{i}}^2-\tau}\trkla{\tabs{\incbeta{m+1}{k}{j}}^2-\tau}\right.\\
&\,\left.\cdot\norm{\lambda_{li}\g{li}}_{W^{1,\infty}\trkla{\D}}^2\norm{\lambda_{kj}\g{kj}}_{W^{1,\infty}\trkla{\D}}^2 \norm{\nabla\bs{d}\h\nh}_{L^2\trkla{\D}}^2\norm{\bs{d}\h\nn-\bs{d}\h\no}_{H^1\trkla{\D}}^2\right]\\
\leq&\,CM\sum_{n=0}^{M-1}\!\sum_{i,j=1}^3\!\sum_{k,l\in\mathds{Z}\h}\!\!\norm{\lambda_{li}\g{li}}_{W^{1,\infty}\trkla{\D}}^2\!\norm{\lambda_{kj}\g{kj}}_{W^{1,\infty}\trkla{\D}}^2 \!\expected{\norm{\nabla\bs{d}\h\nh}_{L^2\trkla{\D}}^{12}}^{1/6}\\
&\,\cdot\!\expected{\norm{\bs{d}\h\nn-\bs{d}\h\no}_{H^1\trkla{\D}}^4}^{1/2}\!\expected{\rkla{\tabs{\incbeta{n+1}{l}{i}}^2-\tau}^6}^{1/6}\!\expected{\rkla{\tabs{\incbeta{n+1}{k}{j}}^2-\tau}^6}^{1/6}\\
\leq&\,C\sum_{n=0}^{M-1}\tau\expected{\norm{\nabla\bs{d}\h\nh}_{L^2\trkla{\D}}^{12}}^{1/6}\expected{\norm{\bs{d}\h\nn-\bs{d}\h\no}_{H^1\trkla{\D}}^4}^{1/2}\\
\leq&\, C\tau^{1/2}\rkla{\sum_{n=0}^{M-1}\!\tau\expected{\norm{\nabla\bs{d}\h\nh}_{L^2\trkla{\D}}^{12}}^{1/3}\!\!\! +\expected{\rkla{\sum_{n=0}^{M-1}\!\norm{\bs{d}\h\nn-\bs{d}\h\no}_{H^1\trkla{\D}}^2}^{\!2}}}\\
\leq&\, C\tau^{1/2}\,.
\end{split}
\end{align}
A similar argumentation provides
\begin{align}
\expected{\abs{\sum_{n=0}^{M-1}\mathfrak{B}_{6,2}\nn}^2}\leq C\tau^{1/2}\,.
\end{align}
For the remaining term $\mathfrak{B}_{6,1}\nn$ we argue as in \eqref{eq:martingale:difference} and deduce
\begin{multline}
\expected{\abs{\sum_{n=0}^{M-1}\mathfrak{B}_{6,1}\nn}^2}\\
=\tfrac14\sum_{n=0}^{M-1}\sum_{i=1}^3\sum_{l\in\mathds{Z}\h}\expected{\trkla{\tabs{\incbeta{n+1}{l}{i}}^2-\tau}^2}\expected{\rkla{\nabla\Ih{\rkla{\bs{d}\h\no\times\lambda_{li}\g{li}}\times\lambda_{li}\g{li}},\nabla\bs{d}\h\no}^2}\\
\leq C\tau\sum_{n=0}^{M-1}\tau\sum_{i=1}^3\sum_{l\in\mathds{Z}\h}\norm{\lambda_{li}\g{li}}_{W^{1,\infty}\trkla{\D}}^4\expected{\norm{\bs{d}\h\no}_{H^1\trkla{\D}}^4}\leq C\tau\,.
\end{multline}
For the final term, we start with a similar decomposition:
\begin{align}
\begin{split}
\mathfrak{B}_7\nn=&\,\tfrac12\sum_{\trkla{i,l}\neq\trkla{j,k}}\rkla{\nabla\Ih{\rkla{\bs{d}\h\nh\times\lambda_{li}\g{li}\incbeta{n+1}{l}{i}}\times\lambda_{kj}\g{kj}\incbeta{n+1}{k}{j}},\nabla\bs{d}\h\nh}\\
=&\,\tfrac12\sum_{\trkla{i,l}\neq\trkla{j,k}}\rkla{\nabla\Ih{\rkla{\bs{d}\h\no\times\lambda_{li}\g{li}\incbeta{n+1}{l}{i}}\times\lambda_{kj}\g{kj}\incbeta{n+1}{k}{j}},\nabla\bs{d}\h\no}\\
&\,+\tfrac14\sum_{\trkla{i,l}\neq\trkla{j,k}}\rkla{\nabla\Ih{\rkla{\bs{d}\h\no\times\lambda_{li}\g{li}\incbeta{n+1}{l}{i}}\times\lambda_{kj}\g{kj}\incbeta{n+1}{k}{j}},\nabla\bs{d}\h\nn-\nabla\bs{d}\h\no}\\
&\,+\tfrac14\sum_{\trkla{i,l}\neq\trkla{j,k}}\rkla{\nabla\Ih{\rkla{\trkla{\bs{d}\h\nn-\bs{d}\h\no}\times\lambda_{li}\g{li}\incbeta{n+1}{l}{i}}\times\lambda_{kj}\g{kj}\incbeta{n+1}{k}{j}},\nabla\bs{d}\h\nh}\\
=:&\,\mathfrak{B}_{7,1}\nn+\mathfrak{B}_{7,2}\nn+\mathfrak{B}_{7,3}\nn\,.
\end{split}
\end{align}
For $\mathfrak{B}_{7,1}\nn$ we deduce by using the independence of the increments $\tgkla{\incbeta{n+1}{l}{i}}_{n,l,i}$
\begin{align}
\begin{split}
&\expected{\abs{\sum_{n=0}^{M-1}\mathfrak{B}_{7,1}\nn}^2}=\tfrac14\sum_{n=0}^{M-1}\sum_{\trkla{i,l}\neq\trkla{j,k}}\expected{\rkla{\nabla\Ih{\trkla{\bs{d}\h\no\times\lambda_{li}\g{li}\incbeta{n+1}{l}{i}}\times\lambda_{kj}\g{kj}\incbeta{n+1}{k}{j}},\nabla\bs{d}\h\no}^2}\\
&\qquad\quad+\tfrac14\sum_{n=0}^{M-1}\sum_{\trkla{i,l}\neq\trkla{j,k}}\mathds{E}\left[\rkla{\nabla\Ih{\trkla{\bs{d}\h\no\times\lambda_{li}\g{li}\incbeta{n+1}{l}{i}}\times\lambda_{kj}\g{kj}\incbeta{n+1}{k}{j}},\nabla\bs{d}\h\no}\right.\\ 
&\qquad\qquad\qquad\qquad\qquad\quad\cdot \left.\rkla{\nabla\Ih{\trkla{\bs{d}\h\no\times\lambda_{kj}\g{kj}\incbeta{n+1}{k}{j}}\times\lambda_{li}\g{li}\incbeta{n+1}{l}{i}},\nabla\bs{d}\h\no}\right]\\
&\quad\leq C\sum_{n=0}^{M-1}\sum_{\trkla{i,l}\neq\trkla{j,k}}\tau^2\norm{\lambda_{li}\g{li}}_{W^{1,\infty}\trkla{\D}}^2\norm{\lambda_{kj}\g{kj}}_{W^{1,\infty}\trkla{\D}}^2\expected{\norm{\bs{d}\h\no}_{H^1\trkla{\D}}^4}\leq C\tau\,.
\end{split}
\end{align}
As before, $\mathfrak{B}_{7,2}\nn$ and $\mathfrak{B}_{7,3}\nn$ can be treated in a similar manner.
Exemplarily, we show the computations for $\mathfrak{B}_{7,3}\nn$:
We first reduce the impact of the interpolation operator by writing
\begin{multline}\label{eq:tmp:b73:splitting}
    \norm{\nabla\Ih{\rkla{\trkla{\bs{d}\h\nn-\bs{d}\h\no}\times\sinc{n+1}}\times\sinc{n+1}}}_{L^2\trkla{\D}}\\ \leq \norm{\nabla\rkla{\Ih{\trkla{\bs{d}\h\nn-\bs{d}\h\no}\times\sinc{n+1}}\times\Ih{\sinc{n+1}}}}_{L^2\trkla{\D}}\\
    + \norm{\nabla\trkla{I-\Ihop}\gkla{\Ih{\trkla{\bs{d}\h\nn-\bs{d}\h\no}\times\sinc{n+1}}\times\Ih{\sinc{n+1}}}}_{L^2\trkla{\D}}\,.
\end{multline}
To derive an estimate for the first term, we use Hölder's inequality and the standard error estimates for $\Ihop$ for smooth functions (cf.~\cite{BrennerScott}). Thus, we obtain
\begin{align}
    \begin{split}
        &\norm{\nabla\rkla{\Ih{\trkla{\bs{d}\h\nn-\bs{d}\h\no}\times\sinc{n+1}}\times\Ih{\sinc{n+1}}}}_{L^2\trkla{\D}}\\
        &~\leq C\sum_{q=1}^d\norm{\rkla{\partial_q\Ih{\trkla{\bs{d}\h\nn-\bs{d}\h\no}\times\sinc{n+1}}}\times\Ih{\sinc{n+1}}}_{L^2\trkla{\D}}\\
        &\quad+C\sum_{q=1}^d\norm{\Ih{\trkla{\bs{d}\h\nn-\bs{d}\h\no}\times\sinc{n+1}}\times\rkla{\partial_q\Ih{\sinc{n+1}}}}_{L^2\trkla{\D}}\\
        &~\leq C\sum_{q=1}^d\norm{\partial_q\Ih{\trkla{\bs{d}\h\nn-\bs{d}\h\no}\times\sinc{n+1}}}_{L^2\trkla{\D}}\norm{\sinc{n+1}}_{L^\infty\rkla{\D}}\\
        &\quad+C\sum_{q=1}^d\norm{\Ih{\trkla{\bs{d}\h\nn-\bs{d}\h\no}\times\sinc{n+1}}}_{L^3\trkla{\D}}\norm{\partial_q\sinc{n+1}}_{L^6\trkla{\D}}\,.
    \end{split}
\end{align}
Next, we show that the remaining term on the right-hand side of \eqref{eq:tmp:b73:splitting} can be written in a similar form.
Combining Lemma \ref{lem:interpolationerror}, standard inverse estimates, and the error estimates for $\Ihop$ applied to smooth functions, we deduce
\begin{align}
    \begin{split}
        \norm{\nabla\trkla{\id-\Ihop}\gkla{\Ih{\trkla{\bs{d}\h\nn-\bs{d}\h\no}\times\sinc{n+1}}\times\Ih{\sinc{n+1}}}}_{L^2\trkla{\D}}\\
        \leq Ch\norm{\nabla\Ih{\trkla{\bs{d}\h\nn-\bs{d}\h\no}}}_{L^3\trkla{\D}}\norm{\nabla\Ih{\sinc{n+1}}}_{L^6\trkla{\D}}\\
        \leq C\norm{\Ih{\trkla{\bs{d}\h\nn-\bs{d}\h\no}\times\sinc{n+1}}}_{L^3\trkla{\D}}\norm{\sinc{n+1}}_{W^{1,6}\trkla{\D}}\,.
    \end{split}
\end{align}
Combining the above results, we get by recalling the standard Sobolev embeddings
\begin{align}\label{eq:B73}
\begin{split}
&\norm{\nabla\Ih{\rkla{\trkla{\bs{d}\h\nn-\bs{d}\h\no}\times\sinc{n+1}}\times\sinc{n+1}}}_{L^2\trkla{\D}}\\
&\quad \leq C\norm{\nabla\Ih{\trkla{\bs{d}\h\nn-\bs{d}\h\no}\times\sinc{n+1}}}_{L^2\trkla{\D}}\norm{\sinc{n+1}}_{L^\infty\trkla{\D}} \\
&\qquad+C\norm{\Ih{\trkla{\bs{d}\h\nn-\bs{d}\h\no}\times\sinc{n+1}}}_{L^3\trkla{\D}} \norm{\sinc{n+1}}_{W^{1,6}\trkla{\D}}\\
&\quad\leq C\norm{\Ih{\trkla{\bs{d}\h\nn-\bs{d}\h\no}\times\sinc{n+1}}}_{H^1\trkla{\D}}\norm{\sinc{n+1}}_{H^2\trkla{\D}}\\
&\quad \leq C\norm{\trkla{\bs{d}\h\nn-\bs{d}\h\no}}_{H^1\trkla{\D}}\norm{\sinc{n+1}}_{H^2\trkla{\D}}^2\,,
\end{split}
\end{align}
where the last step follows again by standard interpolation estimates. 
Thus, we compute similar to~\eqref{sumestimate} 
\begin{align}
    \begin{split}
        &\expected{\abs{\sum_{n=0}^{M-1}\mathfrak{B}_{7,3}\nn}^2}\\
        &\quad\leq C\expected{\abs{\sum_{n=0}^{M-1}\rkla{\nabla\Ih{\rkla{\trkla{\bs{d}\h\nn-\bs{d}\h\no}\times\sinc{n+1}}\times\sinc{n+1}},\nabla\bs{d}\h\nh}}^2}\\
        &\qquad+C\expected{\abs{\sum_{n=0}^{M-1}\sum_{i=1}^3\sum_{l\in\mathds{Z}\h}\rkla{\nabla\Ih{\rkla{\trkla{\bs{d}\h\nn-\bs{d}\h\no}\times\lambda_{li}\g{li}}\times\lambda_{li}\g{li}},\nabla\bs{d}\h\nh}\tabs{\incbeta{n+1}{l}{i}}^2}^2}\\
        &\quad\leq CM\sum_{n=0}^{M-1}\expected{\norm{\bs{d}\h\nn-\bs{d}\h\no}_{H^1\trkla{\D}}^2\norm{\sinc{n+1}}_{H^2\trkla{\D}}^4\norm{\nabla\bs{d}\h\nh}_{L^2\trkla{\D}}^2}\\
        &\qquad+CM\sum_{n=0}^{M-1}\sum_{i,j=1}^3\sum_{l,k\in\mathds{Z}\h}\mathds{E}\left[\norm{\bs{d}\h\nn-\bs{d}\h\no}_{H^1\trkla{\D}}^2\norm{\nabla\bs{d}\h\nh}_{L^2\trkla{\D}}^2\right.\\
        &\qquad\qquad\qquad\qquad\left.\phantom{\sum_{i=1}^3}\cdot\norm{\lambda_{li}\g{li}}_{H^2\trkla{\D}}^2\norm{\lambda_{kj}\g{kj}}_{H^2\trkla{\D}}^2\tabs{\incbeta{n+1}{l}{i}}^2\tabs{\incbeta{n+1}{k}{j}}^2\right]\\
        &\quad\leq CM\tau^2\sum_{n=0}^{M-1}\expected{\norm{\bs{d}\h\nn-\bs{d}\h\no}_{H^1\trkla{\D}}^6}^{1/3}\expected{\norm{\nabla\bs{d}\h\nh}_{L^2\trkla{\D}}^6}^{1/3}\\
        &\quad\leq C\expected{\tau\sum_{n=0}^{M-1}\norm{\bs{d}\h\nn-\bs{d}\h\no}_{H^1\trkla{\D}}^6}^{1/3} \leq C\tau^{1/3}\expected{\rkla{\sum_{n=0}^{M-1}\norm{\bs{d}\h\nn-\bs{d}\h\no}_{H^1\trkla{\D}}^2}^3}^{1/3}\\
        &\quad\leq C\tau^{1/3}\,.
    \end{split}
\end{align}
\end{proof}
\begin{lemma}\label{lem:E}
Let Assumptions \ref{item:T}, \ref{item:S}, \ref{item:I}, \ref{item:W1}, \ref{item:W2}, and \ref{item:W3} hold true and let $\widehat{\epsilon}$ be the parameter from Assumption \ref{item:C}.
Further, let $B$ satisfy the conditions listed in Definition \ref{def:testprocess}.
Then, the estimates
\begin{subequations}
\begin{align}
\expected{\abs{\mathfrak{E}_1}^2}+\expected{\abs{\mathfrak{E}_2}^2} + \expected{\abs{\mathfrak{E}_3}^2}\leq&\,C\frac{\tau}{h^2}\,,\\
\expected{\abs{\mathfrak{E}_4}^2}\leq&\,C\tau^{\widehat{\epsilon}/2}\,
\end{align}
\end{subequations}
hold true with a constant $C>0$ independent of $h$ and $\tau$.
\end{lemma}
Recalling Assumption \ref{item:C}, these estimates are sufficient to guarantee that $\mathfrak{E}_1$, $\mathfrak{E}_2$, $\mathfrak{E}_3$, and $\mathfrak{E}_4$ vanish when passing to the limit $\trkla{h,\tau}\searrow\trkla{0,0}$.
\begin{proof}
We start with the bounds on $\mathfrak{E}_1$. 
Estimating the $L^2$-inner product using the $H^1$ and $\trkla{H^1}^*$-norm, an inverse estimate, Hölder's inequality, the regularity results established in Lemma \ref{lem:energy}, Lemma \ref{lem:diff} and \eqref{eq:lengthconstraint}, and the Burkholder--Davis--Gundy inequality, we deduce
\begin{align}
\begin{split}
&\expected{\abs{\mathfrak{E}_1}^2}\leq \expected{\abs{\sum_{n=0}^{N-1}\eta\no\tau\norm{\Ih{\trkla{\bs{d}\h\nh\times\Delta\h\bs{d}\h\nn}\times\bs{d}\h\nh}}_{H^1\trkla{\D}}\norm{\int_{t\no}^{t\nn} B\mathrm{d}\bs{W}}_{*}}^2}\\
&\leq\frac{\tau}{h^2}\expected{ \rkla{\sum_{n=0}^{N-1}\tau\abs{\eta\no}^2\norm{\Ih{\rkla{\bs{d}\h\nh\times\Delta\h\bs{d}\h\nn}\times\bs{d}\h\nh}}_{L^2\trkla{\D}}^2}^2}^{1/2}\\
&\qquad\qquad\cdot\expected{\rkla{\sum_{n=0}^{N-1}\norm{\int_{t\no}^{t\nn} B\mathrm{d}\bs{W}}_{*}^2}^2}^{1/2}\\
&\leq C\frac{\tau}{h^2}\expected{\int_0^T\norm{B}_{L_2\trkla{\mathfrak{U};\trkla{H^1\trkla{\D}^3}^*}}^4\ds} \leq C\frac{\tau}{h^2}\,,
\end{split}
\end{align}
Therefore, we can estimate $\mathfrak{E}_{2}$ via
\begin{align}
\expected{\abs{\mathfrak{E}_{2}}^2}=&\,\sum_{i=1}^3\sum_{l\in\mathds{Z}\h}\sum_{n=0}^{N-1}\abs{\eta\no}^2\expected{\rkla{\Ih{\bs{d}\h\no\times\lambda_{li}\g{li}},\mathfrak{I}_{n,l,i}\g{li} }^2}\notag\\
\leq&\, \sum_{i=1}^3\sum_{l\in\mathds{Z}\h}\sum_{n=0}^{N-1}\abs{\eta\no}^2\expected{h^{-2}\norm{\Ih{\bs{d}\h\no\times\lambda_{li}\g{li}}}_{L^2\trkla{\D}}^2\norm{\mathfrak{I}_{n,l,i}\g{li}}_{L_2(\mathfrak U; (H^1(\D))^*)}^2}\notag\\
\leq&\,Ch^{-2}\sum_{i=1}^3\sum_{l\in\mathds{Z}\h}\sum_{n=0}^{N-1}\tabs{\eta\no}^2\rkla{\expected{\tabs{\incbeta{n+1}{l}{i}}^4} +\expected{\tau\int_{t\no}^{t\nn}\norm{B_{li}\g{li}}_{L_2(\mathfrak U; (H^1(\D))^*)}^4\dt} }\notag\\
&\qquad\qquad\qquad\qquad\qquad\cdot\norm{\g{li}\lambda_{li}}_{L^2\trkla{\D}}^2\expected{\norm{\bs{d}\h\no}_{L^\infty\trkla{\D}}^2} \notag\\
&\,+Ch^{-2}\sum_{i=1}^3\sum_{l\in\mathds{Z}\h}\sum_{n=0}^{N-1}\tabs{\eta\no}^2\expected{\tau\int_{t\no}^{t\nn}\norm{B_{li}\g{li}}_{L_2(\mathfrak U; (H^1(\D))^*)}^2\dt}\notag\\
&\qquad\qquad\qquad\qquad\qquad\cdot\norm{\g{li}\lambda_{li}}_{L^2\trkla{\D}}^2\expected{\norm{\bs{d}\h\no}_{L^\infty\trkla{\D}}^2}\notag \\
\leq&\,C\frac{\tau}{h^2}\,.
\end{align}
Here, we used an inverse estimate in the first step transforming the $H^1$-norm into the $L^2$-norm.
To deal with the remaining term $\mathfrak{E}_3$, we use the independence of the stochastic increments and obtain
\begin{align}
\begin{split}
&\expected{\abs{\mathfrak{E}_{3}}^2}=\,\sum_{n=0}^{N-1}\tabs{\eta\no}^2\sum_{\trkla{i,l}\neq\trkla{j,k}}\expected{\skla{\int_{t\no}^{t\nn}B_{kj}\g{kj}\mathrm{d}\beta_k^j,\Ih{\bs{d}\h\no\times\lambda_{li}\g{li}\incbeta{n+1}{l}{i}}}^2}\\
&\,+\sum_{n=0}^{N-1}\tabs{\eta\no}^2\sum_{\trkla{i,l}\neq\trkla{j,k}}\mathds{E}\left[\skla{\int_{t\no}^{t\nn}B_{kj}\g{kj}\mathrm{d}\beta_k^j,\Ih{\bs{d}\h\no\times\lambda_{li}\g{li}\incbeta{n+1}{l}{i}}}\right.\\
&\qquad\qquad\cdot \left.\skla{\int_{t\no}^{t\nn}B_{li}\g{li}\mathrm{d}\beta_l^i,\Ih{\bs{d}\h\no\times\lambda_{kj}\g{kj}\incbeta{n+1}{k}{j}}}\right]\\
\leq&\,C\sum_{n=0}^{N-1}\sum_{i=1}^3\sum_{l\in\mathds{Z}}\expected{\norm{\Ih{\bs{d}\h\no\times\lambda_{li}\g{li}}\incbeta{n+1}{l}{i}}_{H^1\trkla{\D}}^2}\sum_{j=1}^3\sum_{k\in\mathds{Z}}\expected{\norm{\int_{t\no}^{t\nn}B_{kj}\g{kj}\mathrm{d}\beta_k^j}^2_{*}}\\
&\,+C\sum_{n=0}^{N-1}\rkla{\sum_{i=1}^3\sum_{l\in\mathds{Z}}\expected{\norm{\Ih{\bs{d}\h\no\times\lambda_{li}\g{li}}\incbeta{n+1}{l}{i}}_{H^1\trkla{\D}}\norm{\int_{t\no}^{t\nn}B_{li}\g{li}\mathrm{d}\beta_l^i}_{*}}}^2\\
\leq&\,Ch^{-2}\sum_{n=0}^{N-1}\sum_{i=1}^3\sum_{l\in\mathds{Z}}\expected{\norm{\Ih{\bs{d}\h\no\times\lambda_{li}\g{li}}\incbeta{n+1}{l}{i}}_{L^2\trkla{\D}}^2}\sum_{j=1}^3\sum_{k\in\mathds{Z}}\expected{\norm{\int_{t\no}^{t\nn}B_{kj}\g{kj}\mathrm{d}\beta_k^j}^2_{*}}\\
&\,+Ch^{-2}\sum_{n=0}^{N-1}\rkla{\sum_{i=1}^3\sum_{l\in\mathds{Z}}\expected{\norm{\Ih{\bs{d}\h\no\times\lambda_{li}\g{li}}\incbeta{n+1}{l}{i}}_{L^2\trkla{\D}}\norm{\int_{t\no}^{t\nn}B_{li}\g{li}\mathrm{d}\beta_l^i}_{*}}}^2\\
\leq&\,C\frac{\tau}{h^2}\sum_{n=0}^{N-1}\expected{\norm{\bs{d}\h\no}_{L^\infty\trkla{\D}}^2}\sum_{i=1}^3\sum_{l\in\mathds{Z}}\norm{\Ih{\lambda_{li}\g{li}}}_{L^2\trkla{\D}}^2\sum_{j=1}^3\sum_{k\in\mathds{Z}}\expected{\int_{t\no}^{t\nn}\norm{B_{kj}\g{kj}}_{*}^2\dt}\\
&\,+Ch^{-2}\sum_{n=0}^{N-1}\!\expected{\norm{\bs{d}\h\no}_{L^\infty\trkla{\D}}^2}\rkla{\sum_{i=1}^3\!\sum_{l\in\mathds{Z}}\rkla{\!\tau\norm{\Ih{\lambda_{li}\g{li}}}_{L^2\trkla{\D}}^2\! +\expected{\int_{t\no}^{t\nn}\!\!\norm{B_{li}\g{li}}_{*}^2\dt}}}^2\\
\leq&\,C\frac{\tau}{h^2}\expected{\int_0^T\sum_{j=1}^3\sum_{k\in\mathds{Z}}\norm{B_{kj}\g{kj}}_*^2\dt} +Ch^{-2}\sum_{n=0}^{N-1}\rkla{\tau +\expected{\int_{t\no}^{t\nn}\!\sum_{i=1}^3\sum_{l\in\mathds{Z}}\norm{B_{li}\g{li}}_*^2\dt}}^2\\
\leq&\,C\frac{\tau}{h^2} +C\frac{\tau}{h^2}\expected{\int_0^T\rkla{\sum_{i=1}^3\sum_{l\in\mathds{Z}}\norm{B_{li}\g{li}}_*^2}^2\dt}\leq C\frac{\tau}{h^2}\,.
\end{split}
\end{align}

To deal with the remaining residual term $\mathfrak{E}_{4}$, we need the estimate
\begin{align}
&\norm{\Ih{\trkla{\bs{d}\h\nn-\bs{d}\h\no}\times\lambda_{li}\g{li}\incbeta{n+1}{l}{i}}}_{H^1\trkla{\D}}\notag\\
&~\leq \norm{\trkla{\bs{d}\h\nn-\bs{d}\h\no}\times\Ih{\lambda_{li}\g{li}\incbeta{n+1}{l}{i}}}_{H^1\trkla{\D}} \notag\\
&\qquad+  \norm{\trkla{1-\Ihop}\gkla{\trkla{\bs{d}\h\nn-\bs{d}\h\no}\times\Ih{\lambda_{li}\g{li}\incbeta{n+1}{l}{i}}}}_{H^1\trkla{\D}}\notag\\
&~\leq \norm{\bs{d}\h\nn-\bs{d}\h\no}_{L^6\trkla{\D}}\norm{\nabla\Ih{\lambda_{li}\g{li}\incbeta{n+1}{l}{i}}}_{L^3\trkla{\D}} \notag\\
&\qquad+ \norm{\nabla\trkla{\bs{d}\h\nn-\bs{d}\h\no}}_{L^2\trkla{\D}}\norm{\Ih{\lambda_{li}\g{li}\incbeta{n+1}{l}{i}}}_{L^\infty\rkla{\D}}\notag\\
&\qquad+Ch\norm{\nabla\rkla{\bs{d}\h\nn-\bs{d}\h\no}}_{L^6\trkla{\D}}\norm{\nabla\Ih{\lambda_{li}\g{li}\incbeta{n+1}{l}{i}}}_{L^3\trkla{\D}}\notag\\
&~\leq Ch^{-1/2}\norm{\bs{d}\h\nn-\bs{d}\h\no}_{H^1\trkla{\D}}\norm{\nabla\Ih{\lambda_{li}\g{li}\incbeta{n+1}{l}{i}}}_{L^2\trkla{\D}}\notag\\
&\qquad+\norm{\nabla\rkla{\bs{d}\h\nn-\bs{d}\h\no}}_{L^2\trkla{\D}}\norm{\lambda_{li}\g{li}\incbeta{n+1}{l}{i}}_{L^\infty\trkla{\D}}\notag\\
&~\leq C\trkla{1+h^{-1/2}}\norm{\bs{d}\h\nn-\bs{d}\h\no}_{H^1\trkla{\D}}\norm{\lambda_{li}\g{li}\incbeta{n+1}{l}{i}}_{H^2\trkla{\D}}\,,
\end{align}
which can be derived by using Hölder's inequality, Lemma \ref{lem:interpolationerror}, inverse estimates, standard error estimates for $\Ihop$ (see e.g.~\cite{BrennerScott}), and Sobolev embeddings.
Using this estimate with $\trkla{1+h^{-1/2}}^2\leq 2\trkla{1+h^{-1}}$, we obtain via Hölder's inequality (twice with the exponents $\trkla{2,2+\delta,\tfrac{4+2\delta}{\delta}}$ with $\delta>0$ specified below), the Burkholder--Davis--Gundy inequality, and Assumption \ref{item:W3}
\begin{align}
&\expected{\abs{\mathfrak{E}_{4}}^2}\notag\\
&\quad\leq C\expected{\abs{\sum_{n=0}^{N-1}\eta\no\skla{\sum_{j=1}^3\sum_{k\in\mathds{Z}} \int_{t\no}^{t\nn}B_{kj}\g{kj}\mathrm{d}\beta_k^j,\Ih{\trkla{\bs{d}\h\nn-\bs{d}\h\no}\times\sum_{i=1}^3\sum_{l\in\mathds{Z}\h}\lambda_{li}\g{li}\incbeta{n+1}{l}{i}}}}^2}\notag\\
&\quad\leq C\trkla{1+h^{-1}}N\sum_{n=0}^{N-1}\expected{\norm{\bs{d}\h\nn-\bs{d}\h\no}_{H^1\trkla{\D}}^2\norm{\sinc{n+1}}_{H^2\trkla{\D}}^2\norm{\int_{t\no}^{t\nn} B\mathrm{d}W}_*^2} \notag\\
&\quad\leq C\trkla{1+h^{-1}}N\sum_{n=0}^{N-1}\left(\expected{\norm{\int_{t\no}^{t\nn}B\mathrm{d}W}_*^4}^{1/2}\expected{\norm{\bs{d}\h\nn-\bs{d}\h\no}_{H^1\trkla{\D}}^{4+2\delta}}^{1/\trkla{2+\delta}}\right.\notag\\
&\qquad\qquad\qquad\qquad\qquad\left.\phantom{\ekla{\norm{\int_{t\no}^{t\nn}}}^2} \cdot\expected{\norm{\sinc{n+1}}_{H^2\trkla{\D}}^{\trkla{8+4\delta}/\delta}}^{\delta/\trkla{4+2\delta}}\right)\notag\\
&\quad\leq C\trkla{1+h^{-1}}N\sum_{n=0}^{N-1}\expected{\tau\int_{t\no}^{t\nn}\norm{B}_{L_2\trkla{L^2\trkla{\D}^3;\trkla{H^1\trkla{\D}^3}^*}}^4\ds}^{1/2}\!\!\expected{\norm{\bs{d}\h\nn-\bs{d}\h\no}_{H^1\trkla{\D}}^{4+2\delta}}^{1/\trkla{2+\delta}}\tau\notag\\
&\quad\leq C\rkla{\tau^{1/2-\widehat{\epsilon}} +\frac{\tau^{1/2-\widehat{\epsilon}}}{h}}\tau^{\widehat{\epsilon}}N^{\delta/\trkla{4+2\delta}}\expected{\int_0^T\norm{B}_{L_2\trkla{L^2\trkla{\D}^3;\trkla{H^1\trkla{\D}^3}^*}}^4\ds}^{1/2}\notag\\
&\qquad\qquad\qquad\qquad\qquad\qquad\qquad\qquad\cdot\expected{\sum_{n=0}^{N-1}\norm{\bs{d}\h\nn-\bs{d}\h\no}_{H^1\trkla{\D}}^{4+2\delta}}^{1/\trkla{2+\delta}} \notag\\
&\quad\leq C\rkla{\tau^{1/2-\widehat{\epsilon}} +\frac{\tau^{1/2-\widehat{\epsilon}}}{h}}\tau^{\widehat{\epsilon}}N^{\delta/\trkla{4+2\delta}} \,.
\end{align}
Choosing $\delta:=\trkla{4\widehat{\epsilon}}/\trkla{2+2\widehat{\epsilon}}$ (i.e.~$\delta/\trkla{4+2\delta}=\widehat{\epsilon}/2$) provides the result.
\end{proof}

\subsection{Reformulation of the correction terms}\label{sec:correctionterms}
In this section, we will derive an alternative formulation for the correction terms
\begin{multline}\label{eq:energiecorrection:terms_new}
    \tfrac12\sum_{i=1}^3\sum_{l\in\mathds{Z}\h}\expected{\psi\int_0^T\eta\tm\norm{\nabla\Ih{\bs{d}\h\tm\times\lambda_{li}\g{li}}}_{L^2\trkla{\D}}^2\dt} \\+\tfrac12\sum_{i=1}^3\sum_{l\in\mathds{Z}\h}\expected{\psi\int_0^T\eta\tm\rkla{\nabla\Ih{\trkla{\bs{d}\h\ti\times\lambda_{li}\g{li}}\times\lambda_{li}\g{li}},\nabla\bs{d}\h\ti}\dt}\,.
\end{multline}
As we have neither control over $\Delta\h\bs{d}\h\ti$ nor strong convergence of $\nabla\bs{d}\h\ti$, passing to the limit with these terms requires a careful reformulation.
We start by removing the interpolation operator $\Ihop$ and by aligning the evaluation in time.
For the first term in~\eqref{eq:energiecorrection:terms_new}, we write
\begin{multline}
\norm{\nabla\Ih{\bs{d}\h\tm\times\lambda_{li}\g{li}}}_{L^2\trkla{\D}}\leq \norm{\nabla\rkla{\bs{d}\h\ti\times\Ih{\lambda_{li}\g{li}}}}_{L^2\trkla{\D}} \\
+\tfrac12\norm{\nabla\rkla{\trkla{\bs{d}\h\tp-\bs{d}\h\tm}\times\Ih{\lambda_{li}\g{li}}}}_{L^2\trkla{\D}}
+\norm{\nabla\trkla{\id-\Ihop}\gkla{\bs{d}\h\tm\times\Ih{\lambda_{li}\g{li}}}}_{L^2\trkla{\D}}\,.
\end{multline}
According to Lemma~\ref{lem:interpolationerror}, the last term on the right-hand side is bounded by 
\begin{align}
Ch\norm{\nabla\bs{d}\h\tm}_{L^2\trkla{\D}}\norm{\lambda_{li}\nabla\Ih{\g{li}}}_{L^\infty\trkla{\D}}\,.
\end{align}
Here, switching from $\bs{d}\h\tm$ to $\bs{d}\h\ti$ allows us to combine the result with the second term in~\eqref{eq:energiecorrection:terms_new}.
Thus, computing the squares and using Young's inequality with coefficients $h^{1/2}$ and $\tau^{1/2}$, we can deduce
\begin{align}
\begin{split}
   & \norm{\nabla\Ih{\bs{d}\h\tm\times\lambda_{li}\g{li}}}_{L^2\trkla{\D}}^2\leq \norm{\nabla\rkla{\bs{d}\h\ti\times\Ih{\lambda_{li}\g{li}}}}_{L^2\trkla{\D}}^2 \\
   &\qquad +\trkla{h^{1/2}+\tau^{1/2}}\norm{\nabla\rkla{\bs{d}\h\ti\times\Ih{\lambda_{li}\g{li}}}}_{L^2\trkla{\D}}^2\\ 
    &\qquad+C\tau^{-1/2}\norm{\nabla\rkla{\trkla{\bs{d}\h\tp-\bs{d}\h\tm}\times\Ih{\lambda_{li}\g{li}}}}_{L^2\trkla{\D}}^2\\ 
    &\qquad+C h^{1/2}\norm{\nabla\bs{d}\h\tm}_{L^2\trkla{\D}}^2\norm{\lambda_{li}\g{li}}_{W^{1,\infty}\trkla{\D}}^2\\
    &~\leq \norm{\nabla\rkla{\bs{d}\h\ti\times\Ih{\lambda_{li}\g{li}}}}_{L^2\trkla{\D}}^2 + \trkla{h^{1/2}+\tau^{1/2}}\norm{\bs{d}\h\ti}_{H^1\trkla{\D}}^2\norm{\lambda_{li}\g{li}}_{W^{1,\infty}\trkla{\D}}^2\\
    &\qquad+C\tau^{-1/2}\norm{\bs{d}\h\tp-\bs{d}\h\tm}_{H^1\trkla{\D}}^2\norm{\lambda_{li}\g{li}}_{W^{1,\infty}\trkla{\D}}^2 \\
    &\qquad+Ch^{1/2}\norm{\bs{d}\h\tm}_{H^1\trkla{\D}}^2\norm{\lambda_{li}\g{li}}_{W^{1,\infty}\trkla{\D}}^2\,.
    \end{split}
\end{align}

For the second term in~\eqref{eq:energiecorrection:terms_new}, we obtain by similar arguments
\begin{align}
\begin{split}
&\rkla{\nabla\Ih{\rkla{\bs{d}\h\ti\times\lambda_{li}\g{li}}\times\lambda_{li}\g{li}},\nabla\bs{d}\h\ti}\\
&\quad=\rkla{\nabla\rkla{\rkla{\bs{d}\h\ti\times\Ih{\lambda_{li}\g{li}}}\times\Ih{\lambda_{li}\g{li}}},\nabla\bs{d}\h\ti}\\
&\qquad\quad-\rkla{\nabla\rkla{\trkla{\id-\Ihop}\gkla{\bs{d}\h\ti\times\Ih{\lambda_{li}\g{li}}}\times\Ih{\lambda_{li}\g{li}}},\nabla\bs{d}\h\ti}\\
&\qquad\quad-\rkla{\nabla\trkla{\id-\Ihop}\gkla{\Ih{\bs{d}\h\ti\times\Ih{\lambda_{li}\g{li}}}\times\Ih{\lambda_{li}\g{li}}},\nabla\bs{d}\h\ti}\\
&\quad\leq\rkla{\nabla\rkla{\rkla{\bs{d}\h\ti\times\Ih{\lambda_{li}\g{li}}}\times\Ih{\lambda_{li}\g{li}}},\nabla\bs{d}\h\ti}\\
&\quad\qquad+C\norm{\nabla\bs{d}\h\ti}_{L^2\trkla{\D}}\norm{\lambda_{li}\g{li}}_{W^{1,\infty}\trkla{\D}} \norm{\trkla{\id-\Ihop}\gkla{\bs{d}\h\ti\times\Ih{\lambda_{li}\g{li}}}}_{H^1\trkla{\D}}\\
&\qquad\quad +Ch\norm{\nabla\bs{d}\h\ti}_{L^2\trkla{\D}}\norm{\nabla\Ih{\bs{d}\h\ti\times\Ih{\lambda_{li}\g{li}}}}_{L^2\trkla{\D}}\norm{\lambda_{li}\nabla\Ih{\g{li}}}_{L^\infty\trkla{\D}}
\\
&\quad\leq\rkla{\nabla\rkla{\rkla{\bs{d}\h\ti\times\Ih{\lambda_{li}\g{li}}}\times\Ih{\lambda_{li}\g{li}}},\nabla\bs{d}\h\ti}\\
&\qquad\quad+ Ch\norm{\bs{d}\h\ti}_{H^1\trkla{\D}}^2\norm{\lambda_{li}\g{li}}_{W^{1,\infty}\trkla{\D}}^2\,,
\end{split}
\end{align}
due to Lemma~\ref{lem:interpolationerror}.
Here, we used~\eqref{eq:tmp:nablaIh} in the last step.
Finally, we use Lemma~\ref{lem} to rewrite the non-vanishing terms in a way that allows us to pass to the limit using only weak convergence of $\bs{d}\h\ti$ and the weak convergence of $\bs{d}\h\ti \times \Delta\h\bs{d}\h\ti$:
\begin{align}
\begin{split}
&\,\norm{\nabla\trkla{\bs{d}\h\ti\times\Ih{\lambda_{li}\g{li}}}}_{L^2\trkla{\D}}^2 +\rkla{\nabla\rkla{\rkla{\bs{d}\h\ti\times\Ih{\lambda_{li}\g{li}}}\times\Ih{\lambda_{li}\g{li}}},\nabla\bs{d}\h\ti}\\
&\quad=\sum_{j=1}^d\int_{\D}  |\bs d\h\ti|^2| \partial_j \Ih{\lambda_{li}\g{li}}|^2 - |\bs d\h\ti \cdot \partial_j \Ih{\lambda_{li}\g{li}} |^2\dx 
\\ &\qquad -\sum_{j=1}^d\int_{\D}  (\bs d\h\ti \times \partial_j \bs d\h\ti) \cdot( \Ih{\lambda_{li}\g{li}}\times \partial _j \Ih{\lambda_{li}\g{li}}  )
\dx\,.
\end{split}
\end{align}
To estimate the first term, we introduce the interpolation operator $\Ihop$ and use the triangle inequality such that $ 2|\bs d\h\ti| \leq | \bs d \h\tp|+|\bs d\h\tm|\leq 2  $. This implies 
\begin{equation*}
    \int_{\D}  |\bs d\h\ti|^2| \partial_j \Ih{\lambda_{li}\g{li}}|^2\dx \leq  \int_{\D}  |\partial_j \Ih{\lambda_{li}\g{li}}|^2\dx 
    \,.
\end{equation*}

Summarizing the previous transformations, we obtain
\begin{align}\label{eq:correction:alternativeform}
    \begin{split}
        &\tfrac12\sum_{i=1}^3\sum_{l\in\mathds{Z}\h}\expected{\psi\int_0^T\eta\tm\norm{\nabla\Ih{\bs{d}\h\tm\times\lambda_{li}\g{li}}}_{L^2\trkla{\D}}^2\dt}\\
        &\quad+\tfrac12\sum_{i=1}^3\sum_{l\in\mathds{Z}\h}\expected{\psi\int_0^T\eta\tm \rkla{\nabla\Ih{\trkla{\bs{d}\h\ti\times\lambda_{li}\g{li}}\times\lambda_{li}\g{li}},\nabla\bs{d}\h\ti}\dt}\\
        &\leq \tfrac12\sum_{i=1}^3\sum_{l\in\mathds{Z}\h}\sum_{j=1}^d\expected{\psi\int_0^T\eta\tm \rkla{\norm{\partial_j\Ih{\lambda_{li}\g{li}}}_{L^2\trkla{\D}}^2-\norm{\bs{d}\h\ti\cdot\partial_j\Ih{\lambda_{li}\g{li}}}_{L^2\trkla{\D}}^2}\dt}\\
        &\quad-\tfrac12\sum_{i=1}^3\sum_{l\in\mathds{Z}\h}\sum_{j=1}^d\expected{\psi\int_0^T\eta\tm\rkla{\bs{d}\h\ti\times\partial_j\bs{d}\h\ti,\Ih{\lambda_{li}\g{li}}\times\partial_j\Ih{\lambda_{li}\g{li}}}\dt} +\mathfrak{R}
    \end{split}
\end{align}
with a residual term satisfying
\begin{align}
    \begin{split}        
    \abs{\mathfrak{R}}\leq&\,C\trkla{h^{1/2}+\tau^{1/2}}\expected{\int_0^T\norm{\bs{d}\h\ti}_{H^1\trkla{\D}}^2\dt} +C\tau^{-1/2}\expected{\int_0^T\norm{\bs{d}\h\nn-\bs{d}\h\no}_{H^1\trkla{\D}}^2\dt}\\
    &\,+Ch^{1/2}\expected{\int_0^T\norm{\bs{d}\h\tm}_{H^1\trkla{\D}}^2\dt} +Ch\expected{\int_0^T\norm{\bs{d}\h\ti}_{H^1\trkla{\D}}^2\dt}\,.
    \end{split}
\end{align}
Thus, recalling the regularity results from Lemma \ref{lem:energy}, we note that $\mathfrak{R}$ vanishes for $\trkla{h,\tau}\searrow0$.

\subsection{Regularity of $\bs{d}-\int_0^\cdot\Phi\trkla{\bs{d}}\times\mathrm{d}\bs{W}$}\label{sec:timederivative}
Before discussing the limit of the energy inequality, we shall briefly verify \eqref{eq:d-dW}, i.e.~that the limit of the difference between the solution and the stochastic integral is differentiable with respect to time.
As we only establish convergence towards energy-variational solutions, we can not use that the limit process satisfies a certain SPDE to deduce the regularity. 
Therefore, we need to derive this regularity result already in the discrete setting.
We start by defining the following piecewise linear approximation of an \Ito-integral:
\begin{align}
&&\bs{S}\h\tl\trkla{t}&:=\trkla{\bs{S}\h\nn-\bs{S}\h\no}\frac{t-t\no}{\tau}+\bs{S}\h\no&&\text{for~} t\in\tekla{t\no,t\nn}\,,\\
\text{with}&&\bs{S}\h^m&:=\sum_{n=0}^{m-1}\sum_{i=1}^d\sum_{l\in\mathds{Z}\h}\Ih{\bs{d}\h\no\times\lambda_{li}\g{li}}\int_{t\no}^{t\nn}\mathrm{d}\beta_l^i&&\text{for~}m\in\tgkla{1,\ldots,N}\,.
\end{align}
With these approximations of the \Ito-integral, we can use similar arguments as in \eqref{eq:deffrakA} and write \eqref{eq:discscheme} as
\begin{multline}\label{eq:interpolatedscheme}
\bs{d}\h\tl\trkla{t}-\bs{d}\h\tl\trkla{0}-\int_0^t\Ih{\rkla{\bs{d}\h\ti\times\Delta\h\bs{d}\h\tp}\times\bs{d}\h\ti}\ds \\
= \tfrac12\int_0^t\sum_{i=1}^3\sum_{l\in\mathds{Z}\h}\tabs{\lambda_{li}}^2\Ih{\trkla{\bs{d}\h\ti\times\g{li}}\times\g{li}}\ds +\bs{S}\h\tl\trkla{t} \\
-\frac1{2\tau}\int_0^t\mathfrak{A}_1\tp\ds-\frac1{2\tau}\int_0^t\mathfrak{A}_2\tp\ds-\frac1{2\tau}\int_0^t\mathfrak{A}_3\tp\ds\,,
\end{multline}
where the residual terms $\mathfrak{A}_1\tp$, $\mathfrak{A}_2\tp$, and $\mathfrak{A}_3\tp$ where implicitly defined by \eqref{eq:deftimeinterpolation} and \eqref{eq:deffrakA}.
Further, the approximation $\bs{S}\h\tl$ satisfies the following regularity results:
\begin{lemma}\label{lem:boundS}
Let assumptions \ref{item:T}, \ref{item:S}, \ref{item:I}, \ref{item:W1}, \ref{item:W2}, and \ref{item:W3} hold true. Then the estimate
\begin{align}\label{eq:boundS}
\norm{\bs{d}\h\tl-\bs{S}\h\tl}_{L^2\trkla{\Omega;H^1\trkla{0,T;L^2\trkla{\D}}}}\leq C\,
\end{align}
holds true with a constant $C$ which is independent of $h$ and $\tau$.\\
Further, there exists a subsequence such that
\begin{align}\label{eq:convS}
        \bs{d}\h\tl-\bs{S}\h\tl\rightharpoonup\bs{d}-\sum_{i=1}^3\sum_{l\in\mathds{Z}}\int_0^\cdot \bs{d}\times\lambda_{li}\g{li}\mathrm{d}\beta_l^i&&\text{in~}L^2\trkla{\Omega;H^1\trkla{0,T;L^2\trkla{\D}}}\,.
    \end{align}
\end{lemma}
\begin{proof}
The bound in $L^2\trkla{\Omega;L^2\trkla{0,T;L^2\trkla{\D}}}$ is a direct consequence of the \Ito~isometry and Lemma \ref{lem:energy}.
From \eqref{eq:interpolatedscheme}, we infer
\begin{align}
\begin{split}
\partial_t\rkla{\bs{d}\h\tl-\bs{S}\h\tl}=&\,\Ih{\rkla{\bs{d}\h\ti\times\Delta\h\bs{d}\h\tp}\times\bs{d}\h\ti}+\tfrac12\sum_{i=1}^3\sum_{l\in\mathds{Z}\h}\tabs{\lambda_{li}}^2\Ih{\rkla{\bs{d}\h\ti\times\g{li}}\times\g{li}}\\
&\,-\tau^{-1}\mathfrak{A}_1\tp-\tau^{-1}\mathfrak{A}_2\tp-\tau^{-1}\mathfrak{A}_3\tp\,.
\end{split}
\end{align}
As the right-hand side is uniformly bounded in $L^2\trkla{\Omega;L^2\trkla{0,T;L^2\trkla{\D}}}$ due to Lemma \ref{lem:energy} and Lemma \ref{lem:stochasticerror}, we obtain \eqref{eq:boundS}.
To establish the weak convergence expressed in \eqref{eq:convS}, we follow the approach in \cite{BanasMetzger2026_arxiv} and first remove the interpolation operator $\Ihop$ and the linearization with respect to time: 
\begin{align}
\begin{split}
&\expected{\norm{\bs{S}\h\tl\trkla{t}-\int_0^t\Phi\trkla{\bs{d}\h\tm}\times\mathrm{d}\bs{W}}_{L^2\trkla{\D}}^2}\\
&\quad\leq C\expected{\norm{\bs{S}\h\tl\trkla{t}-\sum_{i=1}^3\sum_{l\in\mathds{Z}\h}\int_0^t\Ih{\bs{d}\h\tm\times\trkla{\lambda_{li}\g{li}}}\mathrm{d}\beta_l^i}_{L^2\trkla{\D}}^2}\\
&\qquad+C\expected{\norm{\sum_{i=1}^3\sum_{l\in\mathds{Z}\h}\int_0^t\trkla{\id-\Ihop}\gkla{\bs{d}\h\tm\times\Ih{\lambda_{li}\g{li}}}\mathrm{d}\beta_l^i}_{L^2\trkla{\D}}^2}\\
&\qquad+C\expected{\norm{\sum_{i=1}^3\sum_{l\in\mathds{Z}\h}\int_0^t\bs{d}\h\tm\times\trkla{\id-\Ihop}\gkla{\lambda_{li}\g{li}}\mathrm{d}\beta_l^i}_{L^2\trkla{\D}}^2}\\
&\qquad+C\expected{\norm{\sum_{i=1}^3\sum_{l\in\mathds{Z}\setminus\mathds{Z}\h}\int_0^t\bs{d}\h\tm\times\trkla{\lambda_{li}\g{li}}\mathrm{d}\beta_l^i}_{L^2\trkla{\D}}^2}\\
&\quad=:\mathfrak{F}_1+\mathfrak{F}_2+\mathfrak{F}_3+\mathfrak{F}_4\,.
\end{split}
\end{align}
Concerning $\mathfrak{F}_1$, we use \Ito's isometry, assumption \ref{item:W2}, and Lemma \ref{lem:energy} to compute
\begin{align}
\begin{split}
\mathfrak{F}_1\leq&\,C\expected{\norm{\sum_{i=1}^3\sum_{l\in\mathds{Z}\h}\int_{t\no}^{t\nn}\Ih{\bs{d}\h\tm\times\trkla{\lambda_{li}\g{li}}}\mathrm{d}\beta_l^i}_{L^2\trkla{\D}}^2}\\
&\,+C\expected{\norm{\sum_{i=1}^3\sum_{l\in\mathds{Z}\h}\int_{t\no}^{t}\Ih{\bs{d}\h\tm\times\trkla{\lambda_{li}\g{li}}}\mathrm{d}\beta_l^i}_{L^2\trkla{\D}}^2}\\
\leq&\,C\expected{\sum_{i=1}^3\sum_{l\in\mathds{Z}\h}\tau\norm{\Ih{\bs{d}\h\tm\times\trkla{\lambda_{li}\g{li}}}}_{L^2\trkla{\D}}^2}\leq C\tau\,.
\end{split}
\end{align}
For $\mathfrak{F}_2$, we again apply \Ito's isometry, Lemma \ref{lem:interpolationerror}, and inverse estimate, \eqref{eq:lengthconstraint}, and assumption \ref{item:W2}. 
\begin{align}
\begin{split}
\mathfrak{F}_2= &\, C\expected{\sum_{i=1}^3\sum_{l\in\mathds{Z}\h}\int_0^t\norm{\trkla{\id-\Ihop}\gkla{\bs{d}\h\tm\times\Ih{\lambda_{li}\g{li}}}}_{L^2\trkla{\D}}^2\ds}\\
\leq&\,Ch^2\expected{\sum_{i=1}^3\sum_{l\in\mathds{Z}\h}\int_0^t\norm{\bs{d}\h\tm}_{L^\infty\trkla{\D}}^2\norm{\Ih{\lambda_{li}\g{li}}}_{H^1\trkla{\D}}^2\ds}\leq Ch^2\,.
\end{split}
\end{align}
Similarly, we deduce via standard error estimates for $\Ihop$ (see e.g.~\cite{BrennerScott})
\begin{align}
\begin{split}
\mathfrak{F}_3=C\expected{\sum_{i=1}^3\sum_{l\in\mathds{Z}\h}\int_0^t\norm{\bs{d}\h\tm\times\trkla{\id-\Ihop}\gkla{\lambda_{li}\g{li}}}_{L^2\trkla{\D}}^2\ds}
\leq Ch^2\,.
\end{split}
\end{align}
The remaining term $\mathfrak{F}_4$ also vanishes for $h\searrow0$.
Indeed, using \Ito's isometry we obtain
\begin{align}
\begin{split}
\mathfrak{F}_4\leq&\, C\expected{\sum_{i=1}^d\sum_{l\in\mathds{Z}\setminus\mathds{Z}\h}\int_0^t\norm{\bs{d}\h\tm}_{L^\infty\trkla{\D}}^2\norm{\lambda_{li}\g{li}}_{L^\infty\trkla{\D}}^2\ds}\\
\leq&\, C\sum_{i=1}^3\sum_{l\in\mathds{Z}\setminus\mathds{Z}\h}\tabs{\lambda_{li}}^2\norm{\g{li}}_{L^\infty\trkla{\D}}^2\,,
\end{split}
\end{align}
which vanishes due to \ref{item:W3} and \ref{item:Z}.
As the stochastic integral is a linear and continuous mapping from $L^2\trkla{\Omega\times\trkla{0,T}\times\D}$ to $C\trkla{0,T;L^2\trkla{\Omega;L^2\trkla{\D}}}$, we can use the uniform bounds on $\bs{d}\h\tm$ to identify a weakly* converging subsequence, which provides
\begin{align}\label{eq:tmp:convergence:stochintegralweak}
\int_0^\cdot\bs{d}\h\tm\times\mathrm{d}\mathcal{Q}^{1/2}\bs{W}\rightharpoonup \int_0^\cdot\bs{d}\times\mathrm{d}\mathcal{Q}^{1/2}\bs{W}
\end{align}
in $C\trkla{0,T;L^2\trkla{\Omega;L^2\trkla{\D}}}$.
Thus, the desired result follows from \eqref{eq:boundS} and \eqref{eq:tmp:convergence:stochintegralweak}. 
\end{proof}

\subsection{Convergence of the scheme}\label{sec:conv}
This section is devoted to the conclusion of the proof of Theorem~\ref{thm:convergence}, which mainly consist of passing to the limit in all relevant terms. 
\begin{proof}[Proof of Theorem \ref{thm:convergence}]
As a first step, we shall establish weak convergence results that allow us to pass to the limit in our discrete equations -- albeit without completely identifying all limit terms.
The first four convergences~\eqref{eq:envarconvergence:weakstar}--\eqref{eq:envarconvergence:energy} immediately follow from the provided \textit{a priori} estimates.
To establish \eqref{eq:envarconvergence:dissipation2}, we first notice that there exists $ \bs a \in L^{2p}\trkla{\Omega;L^2\trkla{0,T;L^2\trkla{\D;\mathds R^{3}}}}$ such that 
$$
\begin{aligned}
\Ih{\bs{d}\h\ti\times\Delta\h\bs{d}\h\ti}\rightharpoonup&\,\bs a &&\text{in~}L^{2p}\trkla{\Omega;L^2\trkla{0,T;L^2\trkla{\D;\mathds R^{3}}}}\,.    
\end{aligned}
$$
In the following, we want to identify $\bs a$.
By $ \mathcal{P}\h: L^2(\D)^3 \to  \mathds V\h$, we denote the $L^2$-projection onto the finite element space given by 
$ (\Proj{u},v ) = ( u, v) $ for all $ v \in \mathds V\h$. 
For $\widehat{\varphi}\in W^{1,\infty}(\D)^3$, we consider 
\begin{align*}
    \int_{\D} \Ih{\bs d\h\ti \times \Delta\h\bs d \h \ti} \cdot \widehat{\varphi} \dx 
    ={}& \int_{\D} \Ih{\bs d\h\ti \times \Delta\h\bs d \h \ti} \cdot \Proj{\widehat{\varphi}} \dx \\
    ={}& \int_{\D}(I - \Ihop)\gkla{  \Ih{\bs d\h\ti \times \Delta\h\bs d \h \ti} \cdot \Proj{\widehat{\varphi}} }\dx 
    \\& + \rkla{\bs d\h\ti \times \Delta\h\bs d \h \ti, \Proj{\widehat{\varphi}}}\h\,.
\end{align*}
Applying Lemma \ref{lem:interpolationerror} and the standard error estimates for $\Ihop$, we deduce
\begin{multline}\label{eq:tmp:weakconv:error1}
    \abs{\iD\trkla{\id-\Ihop}\gkla{\Ih{\bs{d}\h\ti\times\Delta\h\bs{d}\h\ti}\cdot\Proj{\widehat{\varphi}}}\dx}\\
    \leq Ch\norm{\Ih{\bs{d}\h\ti\times\Delta\h\bs{d}\h\tp}}_{L^2\trkla{\D}}\norm{\widehat{\varphi}}_{H^1\trkla{\D}}\,,
\end{multline}

since the projection operator $\mathcal P\h$ is $H^1(\D)$-stable~\cite[Lemma~1.131]{Ern2004}. 
From the definition of the discrete Laplace operator, we find that  
$$
\begin{aligned}
    \rkla{\bs d\h\ti \times \Delta\h\bs d \h \ti, \Proj{\widehat{\varphi}}}\h &= -\rkla{ \Delta\h\bs d \h \ti, \bs d\h\ti\times  \Proj{\widehat{\varphi}}}\h 
    = \rkla{\nabla \bs d \h \ti, \nabla \Ih{\bs d\h\ti \times\Proj{\widehat{\varphi}}}  }
    \\&= \rkla{\nabla \bs d \h \ti, \nabla \{\bs d\h\ti \times\Proj{\widehat{\varphi}}\}  } +  \rkla{\nabla \bs d \h \ti, \nabla ( I - \Ihop)\{\bs d\h\ti\times  \Proj{\widehat{\varphi}}\}  }\,.
\end{aligned}
$$
Using Lemma~\ref{lem:interpolationerror},  inverse estimates~\cite[Thm.~4.5.11]{BrennerScott}, and the stability of  $\mathcal P\h$ in $W^{1,4}(\D)$, which follows from standard arguments using inverse estimates and the stability of the Clement interpolation operator \cite[Section 1.6]{Ern2004}, we find the bound

\begin{align}\label{eq:tmp:weakconv:error2}
\begin{split}
    \rkla{\nabla \bs d \h \ti, \nabla ( I - \Ihop)\{\bs d\h\ti\times  \Proj{\widehat{\varphi}}\}  } \leq {}&
Ch \| \nabla \bs d\h\ti \| _{L^2(\D)}\| \nabla \bs d\h\ti \| _{L^4(\D)} \| \nabla \Proj{\widehat{\varphi}}\|_{L^4(\D)} \\
\leq {}&
Ch^{1-d/4} \| \nabla \bs d\h\ti \| _{L^2(\D)}^2  \| \nabla \widehat{\varphi}\|_{L^4(\D)}
\,.
\end{split}
\end{align}

Moreover, it holds due to the properties of the cross product
$$
\begin{aligned}
    \rkla{\nabla \bs d \h \ti, \nabla \{\bs d\h\ti \times\Proj{\widehat{\varphi}}\}  } &
    = \sum_{j=1}^3\left[\rkla{\partial_j \bs d \h \ti, \partial_j \bs d\h\ti \times \Proj{\widehat{\varphi}}} + \rkla{\partial_j \bs d \h \ti,  \bs d\h\ti \times \partial_j \Proj{\widehat{\varphi}}}\right] 
    \\
    &= - \sum_{j=1}^3
    \rkla{\bs d \h \ti \times \partial_j \bs d \h \ti , \partial_j \Proj{\widehat{\varphi}} }\,.
\end{aligned}
$$

This implies that for all $ \psi \in L^\infty(\Omega)$ and $ \eta\in L^\infty(0,T)$ the error terms estimated in \eqref{eq:tmp:weakconv:error1} and \eqref{eq:tmp:weakconv:error2} vanish and
\begin{equation}
\begin{aligned}
     0 & =  \lim_{\tau,h\to 0}  \expected{\psi\int_0^T \eta\left[ \rkla{\Ih{\bs d\h\ti \times \Delta\h\bs d \h \ti}, \widehat{\varphi}}+ \sum_{j=1}^3
    \rkla{\bs d \h \ti \times \partial_j \bs d \h \ti , \partial_j \Proj{\widehat{\varphi}} }\right]  } \\& = 
    \expected{\psi\int_0^T\eta \int_{\D} \bs a \cdot \widehat{\varphi} + \bs q : \nabla \widehat{\varphi} \dx\dt }
    \,,
\end{aligned}
\end{equation}
which is nothing else than~\eqref{eq:envarconvergence:dissipation2}.\\
Next, we want to show the formulation for the identification of~$\bs q$.
Next, we want to show that the limit $\bs{q}$ is part of the energy-variational solution, \textit{i.e.}, it satisfies \eqref{eq:formulationq}.
Therefore, we consider $\rkla{\bs d\h\ti \times \nabla\bs d \h \ti, \bs \Psi}
$
for $\bs \Psi \in (L^\infty(\D))^{3 \times d}$. By the usual binomial formula, we find
\begin{equation}\label{eq:qformuladiscrete}
    \begin{aligned}
     \rkla{\bs d\h\ti \times \nabla\bs d \h \ti, \bs \Psi} &=     - \rkla{ \nabla\bs d \h \ti, \bs d\h\ti \times\bs \Psi} \\
     &=  \frac{1}{2} \left\| \nabla\bs d \h \ti- \bs d\h\ti \times\bs \Psi\right\|_{L^2(\D)}^2 -  \frac{1}{2}\left\| \nabla\bs d \h \ti\right\|_{L^2(\D)}^2 -  \frac{1}{2}\left\| \bs d\h\ti \times\bs \Psi\right\|_{L^2(\D)}^2 
\\&= 
 \frac{1}{2}\iD |\nabla\bs d \h \ti- \bs d\h\ti \times\bs \Psi|^2 - | \nabla\bs d \h\ti|^2 - |\bs d\h\ti|^2 |\bs \Psi|^2 + |\bs d\h\ti\cdot\bs \Psi|^2\dx\,.
\end{aligned}
\end{equation}
For the second to the last term, we observe by $  {|\bs d\h\ti|^2}\leq  1$ that
$$
\begin{aligned}
    -\int_{\D} |\bs d\h\ti|^2 |\bs \Psi|^2 \dx 
    &\geq -  \int_{\D}  |\bs \Psi|^2 \dx 
    \,.
\end{aligned}
$$
Multiplying~\eqref{eq:qformuladiscrete} by $ \eta \in L^\infty(0,T)$ and $ \psi\in L^\infty(\Omega)$ with $ \eta, \psi \geq 0$, we find by the convergences~\eqref{eq:envarconvergence:weakgrad} and~\eqref{eq:envarconvergence:energy} as well as the lower semicontinuity of the appearing quadratic functions that
\begin{align*}
    \liminf_{\tau,h\searrow 0}&\, \expected{\psi \int_0^T \eta \iD (\bs d\h\ti \times \nabla\bs d \h \ti): \bs \Psi \dx\dt}\\
    & \geq \frac{1}{2}\expected{\psi \int_0^T \eta \iD| \nabla \bs d - \bs d \times \bs{\Psi} |^2 -|\bs{\Psi} |^2 + | \bs d \cdot \bs \Psi |^2 \dx - E \dt  }
    \\
    &=\frac{1}{2}\expected{\psi \int_0^T \eta \iD | \nabla \bs d|^2 - 2 ( \nabla \bs d , \bs d \times \bs{\Psi}) + |\bs d \times \bs{\Psi} |^2 -|\bs{\Psi} |^2 + | \bs d \cdot \bs \Psi |^2 \dx - 2E \dt  }
    \\
    &=\expected{\psi \int_0^T \eta \iD\bs d \times \nabla \bs d:  \bs{\Psi}  +\frac{1}{2}( | \bs d |^2 -1)|\bs{\Psi} |^2  \dx + \dfrac{1}{2}\| \nabla \bs d \|_{L^2(\D)}^2 -  E \dt  }\,,
\end{align*}
which implies~\eqref{eq:formulationq} via the convergence~\eqref{eq:envarconvergence:dissipation}. 

From the lower semicontinuity of the $L^2$-norm, we find that 
$$
\liminf_{\tau,h\searrow 0}\expected{\psi \int_0^T  \eta  \tfrac{1}{2}\| \nabla \bs d \h\tp \|_{L^2(\D)}^2   
 \dt}  \geq \expected{\psi \int_0^T  \eta  \tfrac{1}{2}\| \nabla \bs d  \|_{L^2(\D)}^2   
 \dt}
$$
for all $ \psi \in L^\infty(\Omega)$ and $ \eta \in L^\infty(0,T)$ with $ \psi , \eta \geq 0$ a.s.~in $\Omega$ and a.e.~in $(0,T)$. This gives rise to $ 
E \geq \tfrac{1}{2}\| \nabla \bs d  \|_{L^2(\D)}^2  
$ a.e.~in $(0,T)$ and a.s.~in $\Omega$.

Going back to~\eqref{eq:weakEnergy2}, we will pass to the limit in this formulation, combining all the previous calculations. Now we will treat all the appearing terms in~\eqref{eq:weakEnergy2}.  
First, we observe that the error terms in the last four lines vanish as $ h,\tau$ go to zero due to Lemma~\ref{lem:diff}, Lemma~\ref{lem:stochasticerror}, Lemma~\ref{lem:B}, and Lemma~\ref{lem:E}.

In order to pass to the limit in the terms 
$$
\expected{\psi\int_0^T\eta\tm\left[\norm{\bs{d}\h\ti\times\Delta\h\bs{d}\h\ti}_h^2+ \trkla{\Ih{\trkla{\bs{d}\h\ti\times\Delta\h\bs{d}\h\ti}\times\bs{d}\h\ti},\bs{\varphi}\tm}\right]\dt}\,,
$$
we will add the suitable terms in order to induce lower semicontinuity of the formulation according to Definition~\ref{def:envar2}.  
First, we observe that 
\begin{align}
\begin{split}
\trkla{\Ih{\trkla{\bs{d}\h\ti\times\Delta\h\bs{d}\h\ti}\times\bs{d}\h\ti},\bs{\varphi}\tm} =  {}&\trkla{\Ih{\trkla{\bs{d}\h\ti\times\Delta\h\bs{d}\h\ti}}\times\bs{d}\h\ti,{\bs{\varphi}\tm}}
\\&+ \trkla{( \mathcal{I}\h-I)\gkla{\Ih{\trkla{\bs{d}\h\ti\times\Delta\h\bs{d}\h\ti}}\times\bs{d}\h\ti},{\bs{\varphi}\tm}}
\end{split}
\end{align}
and for the second term, we may estimate using Lemma \ref{lem:interpolationerror}
$$
\begin{aligned}
    &\trkla{( \mathcal{I}\h-I)\left[\Ih{\trkla{\bs{d}\h\ti\times\Delta\h\bs{d}\h\ti}}\times\bs{d}\h\ti\right],{\bs{\varphi}\tm}}\\
    &\leq {}\norm{( \mathcal{I}\h-I)\left[\Ih{\trkla{\bs{d}\h\ti\times\Delta\h\bs{d}\h\ti}}\times\bs{d}\h\ti\right]}_{L^{4/3}(\D)}\norm{{\bs{\varphi}\tm}}_{L^4(\D)}
    \\
    & \leq {} C h \norm{\Ih{\trkla{\bs{d}\h\ti\times\Delta\h\bs{d}\h\ti}}}_{L^2(\D)}\norm{\nabla\bs{d}\h\ti}_{L^4(\D)}\norm{{\bs{\varphi}\tm}}_{L^4(\D)} 
   \\&  \leq {} C h ^{1-d/4}\norm{\Ih{\trkla{\bs{d}\h\ti\times\Delta\h\bs{d}\h\ti}}}_{L^2(\D)}\norm{\nabla\bs{d}\h\ti}_{L^2(\D)}{\norm{{\bs{\varphi}\tm}}_{L^4(\D)}} =:Ch^{1-d/4}\widehat{\mathfrak{R}}\,,
\end{aligned}
$$

where we used H\"older's inequality in the first, Lemma~\ref{lem:interpolationerror} in the second, and inverse estimates~\cite[Thm.~4.5.11]{BrennerScott} in the third inequality, respectively. 
Using that $ \|a_h\|\h \geq \|a\h\|_{L^2(\D)}$ for all $ a\h \in U\h$ implies $ 
\norm{\bs{d}\h\ti\times\Delta\h\bs{d}\h\ti}_h^2 \geq \norm{\Ih{\bs{d}\h\ti\times\Delta\h\bs{d}\h\ti}}_{L^2(\D)}^2
$.

Observing again by $ | \bs d \h \ti| \leq 1$ that 
$$ 
\begin{aligned}
-\frac{1}{4} \norm{\bs d \h\ti \times {\bs{\varphi}\tm}}_{L^2(\D)}^2={}& - \frac{1}{4} \norm{ | \bs d \h\ti | {\bs{\varphi}\tm}}_{L^2(\D)}^2 + \frac{1}{4} \norm{\bs d \h\ti \cdot {\bs{\varphi}\tm}}_{L^2(\D)}^2\\\geq{}&  - \frac{1}{4} \norm{  {\bs{\varphi}\tm}}_{L^2(\D)}^2 + \frac{1}{4} \norm{\bs d \h\ti \cdot {\bs{\varphi}\tm}}_{L^2(\D)}^2\,,     
\end{aligned}
$$
we may add and subtract $\frac{1}{4} \norm{\bs d \h\ti \times {\bs{\varphi}\tm}}_{L^2(\D)}^2$ such that
$$
\begin{aligned}
   &\norm{\bs{d}\h\ti\times\Delta\h\bs{d}\h\ti}_h^2+ \trkla{\Ih{\trkla{\bs{d}\h\ti\times\Delta\h\bs{d}\h\ti}\times\bs{d}\h\ti},\bs{\varphi}\tm}
 \\
 &\geq  \norm{\Ih{\bs{d}\h\ti\times\Delta\h\bs{d}\h\ti}}_
 {L^2(\D)}^2+ \trkla{\Ih{\bs{d}\h\ti\times\Delta\h\bs{d}\h\ti},\bs{d}\h\ti\times{\bs{\varphi}\tm}} - C h^{1-d/4}\widehat{\mathfrak{R}}
 \\
&\geq \norm{\Ih{\bs{d}\h\ti\times\Delta\h\bs{d}\h\ti}+\tfrac{1}{2}\bs d \h\ti \times {\bs{\varphi}\tm} }_
 {L^2(\D)}^2 \\&\quad - \frac{1}{4} \norm{ {\bs{\varphi}\tm}}_{L^2(\D)}^2 + \frac{1}{4} \norm{\bs d \h\ti \cdot {\bs{\varphi}\tm}}_{L^2(\D)}^2 -Ch^{1-d/4}\widehat{\mathfrak{R}} \,.
\end{aligned}
$$

From the weak convergences~\eqref{eq:envarconvergence:dissipation2} and \eqref{eq:envarconvergence:weakstar}, we find  
\begin{align}
\begin{split}
   \liminf_{h,\tau \searrow 0}& \,\mathds{E}\left[\psi\int_0^T  \eta\tm\left[\norm{\Ih{\bs{d}\h\ti\times\Delta\h\bs{d}\h\ti}+ \tfrac{1}{2}\bs d \h\ti \times {\bs{\varphi}\tm}}_{L^2(\D)}^2 \right.\right. \\
 &\qquad\left.\phantom{\int_0^T}\left. +\tfrac{1}{4} \norm{\bs d \h\ti \cdot {\bs{\varphi}\tm}}_{L^2(\D)}^2  \right]\dt\right]
 \\
 &\geq \expected{\psi\int_0^T  \eta\left[\norm{\div \bs q+ \tfrac{1}{2}\bs d  \times \bs{\varphi}}_{L^2(\D)}^2 
 + \frac{1}{4} \norm{\bs d  \cdot \bs{\varphi}}_{L^2(\D)}^2  \right]\dt}\,.
 \end{split}
\end{align}
From classical binomial formulas and the rules from vector calculus~\eqref{crossproduct}, we find 
$$
\begin{aligned}
    &\norm{\div  \bs q + \tfrac{1}{2}\bs d  \times \bs{\varphi}}_{L^2(\D)}^2 
- \frac{1}{4}\| \bs{\varphi} \|_{L^2(\D)}^2  + \frac{1}{4} \norm{\bs d  \cdot \bs{\varphi}}_{L^2(\D)}^2\\
 &= \norm{\div  \bs q }_{L^2(\D)}^2   + \rkla{\div  \bs q,  \bs d  \times \bs{\varphi} }
 + \frac{1}{4} \norm{\bs d  \times \bs{\varphi}}_{L^2(\D)}^2
 - \frac{1}{4}\| \bs{\varphi} \|_{L^2(\D)}^2 
 + \frac{1}{4} \norm{\bs d  \cdot \bs{\varphi}}_{L^2(\D)}^2\\
 &= \norm{\div  \bs q }_{L^2(\D)}^2   + \rkla{\div  \bs q,  \bs d  \times \bs{\varphi} }
 - \frac{1}{4} \iD (1-|\bs d|^2) |\bs{\varphi}|^2 \dx 
\,.
\end{aligned}
$$
The weak* convergence~\eqref{eq:envarconvergence:energy} and \eqref{eq:envarconvergence:weakstar} together with the strong convergence of the initial values $ \bs d\h^0 \to \bs d_0$ in $L^p(\Omega; H^1(\D;\mathds R^3))$ implies the convergence 
\begin{multline}
    -\expected{\psi\int_0^T\partial_t\eta\tl\ekla{\tfrac12\norm{\nabla\bs{d}\h\tp}_{L^2\trkla{\D}}^2-\trkla{\bs{d}\h\tp,\bs{\varphi}\tp} +\eta\trkla{0}\ekla{\tfrac12\norm{\nabla\bs{d}\h^0}_{L^2\trkla{\D}}^2-\trkla{\bs{d}\h^0,\bs{\varphi}\trkla{0}}}}\dt}\\
\\
\longrightarrow 
-\expected{\psi\int_0^T\partial_t\eta\ekla{E-\trkla{\bs{d},\bs{\varphi}}}\dt+ \eta\trkla{0}\ekla{\tfrac12\norm{\nabla\bs{d}_0}_{L^2\trkla{\D}}^2-\trkla{\bs{d}_0,\bs{\varphi}\trkla{0}}}\dt}\,
\end{multline}
as $\tau, h \searrow 0$ for all nonnegative $\psi \in L^\infty(\Omega) $ and all $ \eta \in C^1_c([0,T))$. 
The weak* convergence~\eqref{eq:envarconvergence:weakgrad} also allows to pass to the limit in the coupling terms with the test process
\begin{multline*}
    \expected{\psi\int_0^T\eta\tm\langle \bs{d}\h\tp,A\rangle \dt}
+\sum_{i=1}^3\sum_{l\in\mathds{Z}\h}\expected{\psi\int_0^T\eta\tm\skla{\Ih{\bs{d}\h\ti\times\lambda_{li}\g{li}},B_{li}\trkla{t}\g{li}}\dt}\\
\longrightarrow
\expected{\psi\int_0^T\eta\langle\bs{d},A\rangle \dt}   
+\sum_{i=1}^3\sum_{l\in\mathds{Z}\h}\expected{\psi\int_0^T\eta\skla{\bs{d}\times\lambda_{li}\g{li},B_{li}\trkla{t}\g{li}}\dt}\,
\end{multline*}
as $\tau,h\searrow 0$ for all for all nonnegative $\psi \in L^\infty(\Omega) $ and all $ \eta \in C^1_c([0,T))$. 
Here, we again used Lemma~\ref{lem:interpolationerror} to deduce 
\begin{equation}\label{interpolest}
    \norm{( I -\mathcal{I}\h)\{ \bs{d}\h\ti\times\g{li}\}}_{H^1(\D)} \leq 
 C h \| \nabla \bs{d}\h\ti \|_{L^2(\D)} \| \nabla \g{li} \|_{L^\infty(\D)}
\end{equation} such that 

\begin{align} 
\begin{split}\sum_{i=1}^3&\sum_{l\in\mathds{Z}\h}\expected{\psi\int_0^T\eta\tm\skla{(\mathcal{I}\h-I)\{\bs{d}\h\ti\times\lambda_{li}\g{li}\},B_{li}\trkla{t}\g{li}}\dt}\\ 
&\leq C h \sum_{i=1}^3\sum_{l\in\mathds{Z}\h}\expected{\psi\int_0^T\eta\tm
\| \nabla \bs{d}\h\ti \|_{L^2(\D)} \| \lambda_{li} \nabla \g{li} \|_{L^\infty(\D)} \norm{
B_{li}\trkla{t}\g{li}}_{*}\dt}\\
&\leq Ch\expected{\psi\int_0^T\norm{\nabla\bs{d}\h\ti}_{L^2\trkla{\D}}^2\dt\sum_{i=1}^3\sum_{l\in\mathds{Z}\h}\norm{\lambda_{li}\g{li}}_{W^{1,\infty}\trkla{\D}}^2}\\
&\qquad+ Ch\expected{\psi\int_0^T\sum_{i=1}^3\sum_{l\in\mathds{Z}\h}\norm{B_{li}\trkla{t}\g{li}}_*^2\dt}\,.
\end{split}
\end{align}
Using the linearity of the stochastic integral and the weak* convergences~\eqref{eq:envarconvergence:dissipation2} and \eqref{eq:envarconvergence:weakgrad}, we find that 
\begin{align*}
 \expected{\psi\int_0^T\eta\tm\rkla{\Ih{\bs{d}\h\tm\times\Delta\h\bs{d}\h\tm},\lambda_{li}\g{li}}\h\mathrm{d}\beta_l^i}
&\longrightarrow
\expected{\psi\int_0^T\eta\rkla{\div\tgkla{\bs{q}},\lambda_{li}\g{li}}  \mathrm{d}\beta_l^i}\,,
\intertext{where the difference between the mass-lumped and the $L^2$ inner product vanishes due to standard interpolation estimates,  as well as}
\sum_{i=1}^3\sum_{l\in\mathds{Z}}\expected{\psi\int_0^T\eta \tm \skla{B_{li}\g{li},\bs{d}\h\tm}\mathrm{d}\beta_l^i}
&\longrightarrow
\sum_{i=1}^3\sum_{l\in\mathds{Z}}\expected{\psi\int_0^T\eta \skla{B_{li}\g{li},\bs{d}}\mathrm{d}\beta_l^i}\,
\end{align*}
as $\tau,h\searrow 0$ for all for all nonnegative $\psi \in L^\infty(\Omega) $ and all $ \eta \in C^1_c([0,T))$. 
The weak convergence~\eqref{eq:envarconvergence:weakstar} together with the interpolation estimate~\eqref{interpolest} for $ \bs{d}\ti\h$ and similar for $\bs {d}\tm\h$ also implies the convergences 
\begin{align}\label{eq:tmp:sec5:conv}
    -&\sum_{i=1}^3\sum_{l\in\mathds{Z}\h}\expected{\psi\int_0^T\eta\tm\rkla{\Ih{\bs{d}\h\tm\times\lambda_{li}\g{li}},\bs{\varphi}\tm}\mathrm{d}\beta_l^i}\nonumber\\
    &\longrightarrow 
    -\sum_{i=1}^3\sum_{l\in\mathds{Z}}\expected{\psi\int_0^T\eta\rkla{\bs{d}\times\lambda_{li}\g{li},\bs{\varphi}}\mathrm{d}\beta_l^i}
    \intertext{as well as}
    -&\tfrac12\sum_{i=1}^3\sum_{l\in\mathds{Z}\h}\expected{\psi\int_0^T\eta\tm\rkla{\Ih{\trkla{\bs{d}\h\ti\times\lambda_{li}\g{li}}\times\lambda_{li}\g{li}},\bs{\varphi}\tm}\dt}\nonumber\\
    &\longrightarrow
-\tfrac12\sum_{i=1}^3\sum_{l\in\mathds{Z}}\expected{\psi\int_0^T\eta\rkla{{\trkla{\bs{d}\times\lambda_{li}\g{li}}\times\lambda_{li}\g{li}},\bs{\varphi}}\dt}
\end{align}
as $\tau,h\searrow 0$ for all for all nonnegative $\psi \in L^\infty(\Omega) $ and all $ \eta \in C^1_c([0,T))$. 
Exemplarily, we will show the convergence \eqref{eq:tmp:sec5:conv}. We start by writing
\begin{align}
\begin{split}
    -\sum_{i=1}^3\sum_{l\in\mathds{Z}\h}&\expected{\psi\int_0^T\eta\tm\rkla{\Ih{\bs{d}\h\tm\times\lambda_{li}\g{li}},\bs{\varphi}\tm}\mathrm{d}\beta_l^i}\\
    =&\, -\sum_{i=1}^3\sum_{l\in\mathds{Z}}\expected{\psi\int_0^T\eta\tm\rkla{\bs{d}\h\tm\times\lambda_{li}\g{li},\bs{\varphi}\tm}\mathrm{d}\beta_l^i} \\
    &\,+\sum_{i=1}^3\sum_{l\in\mathds{Z}\setminus\mathds{Z}\h}\expected{\psi\int_0^T\eta\tm\rkla{\bs{d}\h\tm\times\lambda_{li}\g{li},\bs{\varphi}\tm}\mathrm{d}\beta_l^i}\\
    &\,-\sum_{i=1}^3\sum_{l\in\mathds{Z}\h}\expected{\psi\int_0^T\eta\tm\rkla{\bs{d}\h\tm\times\trkla{\id-\Ihop}\tgkla{\lambda_{li}\g{li}},\bs{\varphi}\tm}\mathrm{d}\beta_l^i}\\
    &\,-\sum_{i=1}^3\sum_{l\in\mathds{Z}\h}\expected{\psi\int_0^T\eta\tm\rkla{\trkla{\id-\Ihop}\gkla{\bs{d}\h\tm\times\Ih{\lambda_{li}\g{li}}},\bs{\varphi}\tm}\mathrm{d}\beta_l^i}\\
    =:&\,\mathfrak{T}_1+\mathfrak{T}_2+\mathfrak{T}_3+\mathfrak{T}_4\,.
    \end{split}
\end{align}
Using Hölder's inequality, the \Ito~isometry and Lemma \ref{lem:interpolationerror}, we obtain
\begin{align}
    \begin{split}
        \abs{\mathfrak{T}_4}\leq&\,\expected{\tabs{\psi}^2}^{1/2}\expected{\rkla{\sum_{i=1}^3\sum_{l\in\mathds{Z}\h}\int_0^T\eta\tm\rkla{\trkla{\id-\Ihop}\gkla{\bs{d}\h\tm\times\Ih{\lambda_{li}\g{li}}},\bs{\varphi}\tm}\mathrm{d}\beta_l^i}^2}^{1/2}\\
        \leq&\,Ch^2\expected{\int_0^T\tabs{\eta\tm}^2\norm{\nabla\bs{d}\h\tm}_{L^2\trkla{\D}}^2\sum_{i=1}^3\sum_{l\in\mathds{Z}\h}\norm{\lambda_{li}\g{li}}_{W^{1,\infty}\trkla{\D}}^2\norm{\bs{\varphi}\tm}_{L^2\trkla{\D}}^2\dt}^{1/2}\\
        \leq&\,Ch^2\expected{\int_0^T\tabs{\eta\tm}^2\norm{\nabla\bs{d}\h\tm}_{L^2\trkla{\D}}^2\norm{\bs{\varphi}\tm}_{L^2\trkla{\D}}^2\dt}^{1/2}\leq Ch^2\,.
    \end{split}
\end{align}
Similarly, we obtain
\begin{align}
    \begin{split}
        \abs{\mathfrak{T}_3}\leq&\,\expected{\sum_{i=1}^3\sum_{l\in\mathds{Z}\h}\int_0^T\tabs{\eta\tm}^2\norm{\bs{d}\h\tm}_{L^\infty\trkla{\D}}^2\norm{\trkla{\id-\Ihop}\gkla{\lambda_{li}\g{li}}}_{L^2\trkla{\D}}^2\norm{\bs{\varphi}\tm}_{L^2\trkla{\D}}^2\dt}^{1/2}\\
        \leq&\,Ch^4\,
    \end{split}
    \intertext{and}
    \begin{split}
        \abs{\mathfrak{T}_2}\leq&\,\expected{\tabs{\psi}^2}^{1/2}\expected{\sum_{i=1}^3\sum_{l\in\mathds{Z}\setminus\mathds{Z}\h} \int_0^T\tabs{\eta\tm}^2\abs{\rkla{\bs{d}\h\tm\times\lambda_{li}\g{li},\bs{\varphi}\tm}}^2\dt}^{1/2}\\
        \leq&\,C\expected{\int_0^T\tabs{\eta\tm}^2\norm{\bs{d}\h\tm}_{L^\infty\trkla{\D}}^2\norm{\bs{\varphi}\tm}_{L^2\trkla{\D}}^2\sum_{i=1}^3\sum_{l\in\mathds{Z}\setminus\mathds{Z}\h}\norm{\lambda_{li}\g{li}}_{L^2\trkla{\D}}^2}^{1/2}\,,
    \end{split}
\end{align}
which vanishes due to Assumption \ref{item:W3}.
Thus, we can use that the stochastic integral is linear and continuous and pass to the limit in $\mathfrak{T}_1$ using the weak convergence of $\bs{d}\h\tm$ and the strong convergence of $\eta\tm$ and $\bs{\varphi}\tm$.

Finally, we can pass to the limit in the \Ito-Stratonovich correction terms by using a time-interpolated version of the alternative formulation derived in \eqref{eq:correction:alternativeform} and the weak convergence~\eqref{eq:envarconvergence:dissipation2} and~\eqref{eq:envarconvergence:weakstar}:
\begin{multline}\label{eq:limsupestimate}
\limsup_{h,\tau \searrow 0 } \expected{\psi\int_0^T \eta^{\tau,-}\frac{1}2\sum_{i=1}^3\sum_{l\in\mathds{Z}\h}\tabs{\lambda_{li}}^2\sum_{q=1}^d\int_{\D} | \partial_q \Ih{\g{li}}|^2 - |\bs d\h\ti \cdot \partial_q \Ih{\g{li}} |^2\dx\dt}
\\
\,+\limsup_{h,\tau \searrow 0 }\expected{-\psi\int_0^T\eta^{\tau,-}\frac{1}2\sum_{i=1}^3\sum_{l\in\mathds{Z}\h}\tabs{\lambda_{li}}^2\sum_{q=1}^d\rkla{\bs{d}\h\ti\times \partial_q \bs d\h\ti ,\Ih{\g{li}}\times\partial_q\Ih{\g{li}}}\dt}
\\ \leq{}\expected{\psi
\int_0^T \eta\frac{1}2\sum_{i=1}^3\sum_{l\in\mathds{Z}}\tabs{\lambda_{li}}^2\int_{\D} | \nabla  \g{li}|^2 - |\bs d \cdot \nabla  {\g{li}} |^2 - \bs q : \trkla{\g{li}\times\nabla {\g{li}}}  \dx\dt }
\,.
\end{multline}
Here, we used the estimates
\begin{align}
    \begin{split}
        &\abs{\iD\abs{\partial_q\Ih{\g{li}}}^2-\abs{\partial_q\g{li}}^2\dx}=\abs{\iD\rkla{\partial_q\Ih{\g{li}}+\partial_q\g{li}}\cdot\rkla{\partial_q\Ih{\g{li}}-\partial_q\g{li}}\dx}\\
        &~\leq \rkla{\partial_q\Ih{\g{li}}_{L^2\trkla{\D}}+\norm{\partial_q\g{li}}_{L^2\trkla{\D}}}\norm{\partial_q\Ih{\g{li}}-\partial_q\g{li}}_{L^2\trkla{\D}}\leq Ch\norm{\g{li}}_{H^2\trkla{\D}}^2\,,
    \end{split}
\end{align} 
\begin{align}
    \begin{split}
        &\abs{\iD\rkla{\abs{\bs{d}\h\ti\cdot\partial_q\Ih{\g{li}}}^2}-\abs{\bs{d}\h\ti\cdot\partial_q\g{li}}^2\dx}\\
        &~=\abs{\iD\rkla{\bs{d}\h\ti\cdot\rkla{\partial_q\Ih{\g{li}}+\partial_q\g{li}}}\rkla{\bs{d}\h\ti\cdot\rkla{\partial_q\Ih{\g{li}}-\partial_q\g{li}}}\dx}\\
        &~\leq Ch\norm{\bs{d}\h\ti}_{L^\infty\trkla{\D}}^2\norm{\g{li}}_{H^2\trkla{\D}}^2\,,
    \end{split}
\end{align}
and
\begin{align}
    \begin{split}
        &\abs{\rkla{\bs{d}\h\ti\times\partial_q\bs{d}\h\ti,\ekla{\Ih{\g{li}\times\partial_q\Ih{\g{li} -\g{li}\times\partial_q\g{li}}}}}} \\
        &~\leq \norm{\bs{d}\h\ti\times\partial_q\bs{d}\h\ti}_{L^2\trkla{\D}}\norm{\Ih{\g{li}}-\g{li}}_{L^\infty\rkla{\D}}\norm{\partial_q\Ih{\g{li}}}_{L^2\trkla{\D}}\\
        &\quad + \norm{\bs{d}\h\ti\times\partial_q\bs{d}\h\ti}_{L^2\trkla{\D}}\norm{\g{li}}_{L^\infty\trkla{\D}}\norm{\partial_q\Ih{\g{li}}-\partial_q\g{li}}_{L^2\trkla{\D}}\\
        &~\leq Ch\norm{\bs{d}\h\ti\times\partial_q\bs{d}\h\ti}_{L^2\trkla{\D}}\norm{\g{li}}_{W^{1,\infty}\trkla{\D}}\norm{\g{li}}_{H^2\trkla{D}}\,.
    \end{split}
\end{align}
Due to assumption \ref{item:W3} and the uniform bounds on $\norm{\bs{d}\h\ti\times\partial_q\bs{d}\h\ti}_{L^2\trkla{\D}}$, these estimates allow us to remove the interpolation operator on the left-hand side of \eqref{eq:limsupestimate} and use the convexity of the term $ \bs d \mapsto | \bs d \cdot \partial _q \g{li}| ^2$ for fixed $ \g{li}$.  

With all these convergences at hand, we can pass to the limit in the formulation~\eqref{eq:weakEnergy2} and find that for all $ \psi \in L^\infty(\Omega)$ and $ \eta \in C_c^1([0,T))$ with $ \psi, \eta \geq 0$ a.s.~in $\Omega$ and a.e.~in $(0,T)$ it holds that
\begin{align}
\begin{split}
-&\expected{\psi\left[\int_0^T\partial_t\eta\ekla{E-\trkla{\bs{d},\bs{\varphi}}}\dt-\eta\trkla{0}\ekla{\tfrac12\norm{\nabla\bs{d}_0}_{L^2\trkla{\D}}^2-\trkla{\bs{d}_0,\bs{\varphi}_0}}\right]}\\
&+\expected{\psi\int_0^T\eta\skla{A,\bs{d}}\dt} +\sum_{i=1}^3\sum_{l\in\mathds{Z}}\expected{\psi\int_0^T\eta\skla{B_{li}\g{li},\bs{d}}\mathrm{d}\beta_l^i}\\
&+\sum_{i=1}^3\sum_{l\in\mathds{Z}}\expected{\psi\int_0^T\eta\skla{B_{li}\g{li},{\bs{d}\times\lambda_{li}\g{li}}}\dt}\\
&+\expected{\psi\int_0^T\eta\rkla{\norm{\div  \bs q }_{L^2(\D)}^2   + \rkla{\div \bs q,  \bs d  \times \bs{\varphi} }
 - \frac{1}{4} \iD (1-|\bs d|^2) |\bs{\varphi}|^2 \dx}\dt}\\
\leq{}&\sum_{i=1}^3\sum_{l\in\mathds{Z}}\expected{\psi\int_0^T\eta\rkla{\div \bs q,\lambda_{li}\g{li}}  \mathrm{d}\beta_l^i}
\\
&+\expected{\psi\int_0^T \frac{1}2\sum_{i=1}^3\sum_{l\in\mathds{Z}}\tabs{\lambda_{li}}^2\sum_{q=1}^d\int_{\D} | \partial_q \g{li}|^2 - |\bs d \cdot \partial_q {\g{li}} |^2\dx\dt}
\\
&-\expected{\psi\int_0^T\frac{1}2\sum_{i=1}^3\sum_{l\in\mathds{Z}}\tabs{\lambda_{li}}^2\sum_{q=1}^d\rkla{\bs q  , {\g{li}}\times\partial_q {\g{li}}}   \dt}
 \\&+\sum_{i=1}^3\sum_{l\in\mathds{Z}}\expected{\psi\int_0^T\eta\rkla{{\lambda_{li}\g{li}\times\bs{d}},\bs{\varphi}}\mathrm{d}\beta_l^i
 }
 \\&-\frac{1}{2}\sum_{i=1}^3\sum_{l\in \mathds{Z}} \expected{
 \psi \int_0^T \eta \iD |\lambda_{li}|^2 \rkla{(\bs d \times\g{li} )\times \g{li} } \cdot \bs \varphi \dx \dt}\,.
\end{split}
\end{align}
Now, we may apply Lemma~\ref{lem:invar} and infer the formulation~\eqref{eq:envarineq2} from Definition~\ref{def:envar2}.
Indeed, choosing $ g = E - (\bs d , \bs \varphi)$, $h\,:\,\mathfrak{U}^3\rightarrow\mathds{R}$ defined by
\begin{align}
    h(\mathfrak{u}):=\sum_{i=1}^3\sum_{l\in\mathds{Z}}\trkla{\g{li},\mathfrak{u}}\ekla{\skla{\bs{d},B_{li}\g{li}}-\trkla{\div\bs{q},\lambda_{li}\g{li}}-\trkla{\lambda_{li}\g{li}\times\bs{d},\bs{\varphi}}}&&\text{for all~}\mathfrak{u}\in\mathfrak{U}\,,
\end{align} 
and including all the remaining terms into $f$, 
we find that Lemma~\ref{lem:invar} is applicable. 
Especially for the choice of $h$, we observe that 
$$
\begin{aligned}
    \| h \| _{L^1(\Omega;L^2(0,T;L_2(\mathfrak U; \mathds R)))} \leq {}&
C \| \bs d \|_{L^p(\Omega;L^\infty(0,T; H^1(\D)^3)} \norm{B}_{L^4\trkla{\Omega\times\trkla{0,T};L_2\trkla{\mathfrak{U};\trkla{H^1\trkla{\D}^3}^*}}} \\
\\
&+ C\| \div \bs q \|_{L^p(\Omega; L^2(0,T;L^2(\D)))} \sum_{i=1}^3\sum_{l\in\mathds Z}  \lambda _{li}^2 \| \g{li}\|_{L^2(\D)}^2 \\&+ C
    \| \bs d \|_{L^\infty(\Omega \times \D \times (0,T))} \| \bs{\varphi} \| _{L^2(\Omega; L^\infty (0,T;L^2(\D)))} \sum_{i=1}^3\sum_{l\in\mathds Z}  \lambda _{li}^2 \| \g{li}\|_{L^2(\D)}^2 \,.
\end{aligned}
$$

Choosing in the resulting formulation~\eqref{eq:envarineq2} the test process $\bs{\varphi} \equiv 0$, we find that for
the function $$
\begin{aligned}
F(t):={}& E(t) - \sum_{i=1}^3\sum_{l\in\mathds{Z}}  \int_0^t\rkla{\div  \bs q,\lambda_{li}\g{li}}  \mathrm{d}\beta_l^i 
\\&
-\sum_{i=1}^3\sum_{l\in\mathds{Z}} \left[  \int_0^t\tabs{\lambda_{li}}^2\sum_{q=1}^d\int_{\D} | \partial_q \g{li}|^2 - |\bs d \cdot \partial_q {\g{li}} |^2\dx
- \rkla{\bs q  , {\g{li}}\times\partial_q {\g{li}}}   \dt
\right]     
\end{aligned}
$$
it holds $ F(t) - F(s-) \leq 0 $, which implies that $F$
is a monotone function and as such of bounded variation. As the two time integrals are continuous this implies that, choosing the correct representation, $E$ is a cádlág function.
\end{proof}

\section{Weak-strong uniqueness}\label{sec:weakstrong}
In this section, we conclude our analysis by proving Theorem \ref{thm:weakstronguniqueness}.
Thus, we have to show that, if a sufficiently regular process $\widetilde{\bs{d}}$ exists that is a strong solution in the sense of Definition \ref{def:strongSolution}, then this process is unique and coincides with the energy-variational solutions established in Theorem \ref{thm:mainResult}.
\begin{proof}[Proof of Theorem \ref{thm:weakstronguniqueness}]
    We start by introducing an additional stopping time 
    \begin{align}
        \tau_L:=\inf\gkla{t\in\trkla{0,\mathfrak{t}}\,:\norm{\widetilde{\bs{d}}}_{W^{2,\infty}\trkla{\D}} + \norm{\widetilde{\bs{d}}\times\Delta\widetilde{\bs{d}}}_{W^{1,\infty}\trkla{\D}} >L}&&\text{for}~L>0
    \end{align}
    with the convention $\tau_L=\mathfrak{t}$ if the set is empty.
    Recalling \eqref{eq:regStrongSol}, we deduce via Tschebyscheff's inequality
    \begin{multline}
        \mathds{P}\ekla{\tau_L<\mathfrak{t}}\leq \mathds{P}\ekla{\sup_{t\in\tekla{0,\mathfrak{t}}}\norm{\widetilde{\bs{d}}}_{W^{2,\infty}\trkla{\D}} +\sup_{t\in\tekla{0,\mathfrak{t}}}\norm{\widetilde{\bs{d}}\times\Delta\widetilde{\bs{d}}}_{W^{1,\infty}\trkla{\D}}\geq L}\\
        \leq \frac1L\expected{\sup_{t\in\tekla{0,\mathfrak{t}}}\norm{\widetilde{\bs{d}}}_{W^{2,\infty}\trkla{\D}}+\sup_{t\in\tekla{0,\mathfrak{t}}}\norm{\widetilde{\bs{d}}\times\Delta\widetilde{\bs{d}}}_{W^{1,\infty}\trkla{\D}}}\rightarrow 0
    \end{multline}
    for $L\rightarrow\infty$ which guarantees $\tau_L\rightarrow\mathfrak{t}$ in probability. Therefore, it suffices to show the claim in $\trkla{0,\tau_L}$ for fixed $L$.\\
    
\newcommand{\dd}{\widetilde{\bs{d}}}
\newcommand{\dr}{\,\mathrm{d}r}
    Let $\dd$ be the strong solution according to Definition~\ref{def:strongSolution}. 
    Applying \Ito's formula to the energy $ \tfrac{1}{2}\| \nabla \dd \|_{L^2(\D)}^2 $, we obtain 
    \begin{multline}
        \label{strongenergy}
            \tfrac{1}{2}\|\nabla \dd\|_{L^2(\D)}^2 \Big|_{s\wedge \tau_L}^{t\wedge\tau_L} + \int_{s\wedge \tau_L}^{t\wedge\tau_L} \| \dd \times \Delta \dd \|_{L^2(\D)}^2  \dr\\
            -\int_{s\wedge \tau_L}^{t\wedge\tau_L}    \tfrac{1}{2}\sum_{i=1}^3\sum_{l\in\mathds Z} \left[  \|\lambda_{li}\nabla \g{li}\|_{L^2(\D)}^2 - \| \dd \cdot\lambda_{li} \nabla \g{li}\|^2_{L^2(\D)} - \tabs{\lambda_{li}}^2( \dd \times \nabla \dd , \g{li}\times \nabla \g{li})  \right]\dr
            \\
            = \sum_{i=1}^3\sum_{l\in\mathds{Z}}\int_{s\wedge \tau_L}^{t\wedge\tau_L} (\dd \times \Delta \dd , \lambda_{li}\g{li})\mathrm{d}\beta_l^i\,,
    \end{multline}
    due to Lemma \ref{lem}.
    Moreover, we observe that $ | \dd|=1 $ a.e. in $\D \times (0,T)$ and~$\mathds P$-a.s.. 

    Now we choose as the test process in Definition~\ref{def:envar2} the stochastic process given by $\bs \varphi:= -\Delta_N \dd +\dd $ such that $ A = (-\Delta_N+ I) \left[ (  ( \dd \times \Delta \dd )\times\dd + \tfrac{1}{2}\sum_{i=1}^3\sum_{l\in\mathds Z} (( \dd \times \lambda_{li}\g{li})\times \lambda_{li}\g{li})\right]  $ and $B\,:\mathfrak{U}^3\rightarrow\trkla{H^1\trkla{\D}^3}^*$ is defined via
    \begin{align}
        \mathfrak{U}^3\ni\mathfrak{u}\mapsto \trkla{-\Delta_N+\id}\sum_{i=1}^3\sum_{l\in\mathds{Z}}\dd\times\lambda_{li}\g{li}\trkla{\g{li},\mathfrak{u}}\,.
    \end{align}
    We recall that $-\Delta_N$ denotes the Neumann--Laplace operator defined in Section~\ref{sec:notation}. 
    
From the additional regularity assumed for the strong solutions in Definition~\ref{def:strongSolution}, we observe that 
\begin{align}
    A \in L^p_{w^*}(\Omega; L^\infty(0,T;(H^1(\D)^3)^*))&&\text{and}&& B\in L_{\weakstartop}^p\trkla{\Omega;L^\infty\trkla{0,T;L_2\trkla{L^2\trkla{\D}^3;\trkla{H^1\trkla{\D}^3}^*}}}\,.
\end{align}
Moreover, for $\bs \varphi:= -\Delta_N \dd +\dd$, we find by the definition of the Neumann Laplacian and the fact that $ \dd$ fulfills homogeneous Neumann boundary conditions, we deduce via 
\begin{align}
\langle -\Delta_N \dd, \bs a \rangle = \int_{\D} \nabla \dd: \nabla \bs a \dx = - \int_{\D} \Delta \dd \cdot \bs a \dx + \int_{\partial \D} (\bs n \cdot \nabla) \dd \cdot \bs a \mathrm d S = - \int_{\D} \Delta \dd \cdot \bs a \dx 
\end{align}
for all $\bs a \in H^1(\D)$ 
that $ -\Delta_N \dd$ can be pointwise identified with $-\Delta \dd$ as an element in $L^2(\D)$.
From the Definition~\ref{def:strongSolution}, we observe that $ -\Delta \dd \in L^p(\Omega ; C([0,T]; L^\infty(\mathcal{D})))$.
Therefore, $\bs \varphi $ is an admissible test process in the sense of Definition~\ref{def:testprocess}. 
    
    This leads to the inequality
 \begin{align}
        \label{eq:envarineq2inserted}
    \begin{split}
    E\Big|_{(s\wedge \tau_L)-}^{t\wedge\tau_L}& + \int_{s\wedge \tau_L}^{t\wedge\tau_L}\int_{\D} |\div \bs q |^2 \dx\dtau 
   -\sum_{i=1}^3\sum_{l\in\mathds{Z}}\int_{s\wedge \tau_L}^{t\wedge\tau_L}\trkla{\bs q,\lambda_{li}\nabla \g{li}}\mathrm{d}\beta_l^i\\
   & +\rkla{\bs{d},\Delta\dd-\dd}\Big|_{s\wedge \tau_L}^{t\wedge \tau_L} + \int_{s\wedge \tau_L}^{t\wedge\tau_L} \iD  \trkla{\div \bs q}  \cdot ( \bs d \times (\dd - \Delta \dd )) \dx \dtau\\
   &+ \int_{s\wedge \tau_L}^{t\wedge\tau_L} \iD
   \frac{1}{4} |\dd - \Delta \dd|^2 \left[  |\bs d|^2 - 1\right] \dx \dtau\\
 \leq {} & \tfrac12\sum_{i=1}^3\sum_{l\in\mathds{Z}}\tabs{\lambda_{li}}^2\int_{s\wedge \tau_L}^{t\wedge\tau_L} \| \nabla \g{li}\|_{L^2\trkla{\D}}^2 - \|\bs d \cdot \nabla \g{li} \|_{L^2\trkla{\D}}^2 -  \trkla{\bs q , \g{li}\times \nabla  \g{li}  }\dtau
\\
&- \int_{s\wedge \tau_L}^{t\wedge\tau_L} \left\langle (-\Delta_N+ \id) \left[ (  ( \dd \times \Delta \dd )\times\dd + \tfrac{1}{2}\sum_{i=1}^3\sum_{l\in\mathds Z} \tabs{\lambda_{li}}^2(( \dd \times \g{li})\times \g{li})\right] , \bs d \right \rangle  \dtau  \end{split}\notag\\
   \begin{split}& - 
 \sum_{i=1}^3
\sum_{l \in\mathds{Z}}\tabs{\lambda_{li}}^2\int_{s\wedge \tau_L}^{t\wedge\tau_L} \left \langle (-\Delta_N+\id)(\dd \times \g{li}) , ( \bs d \times \g{li})\right \rangle  \dtau 
\\&-\sum_{i=1}^3\sum_{l\in\mathds Z}\int_{s\wedge \tau_L}^{t\wedge\tau_L} \left \langle  (-\Delta_N+\id)\lambda_{li}(\dd \times  \g{li}) , \bs d \right \rangle  \mathrm{d}\beta_l^i
 \\
& -\frac{1}{2}\sum_{i=1}^3\sum_{l\in \mathds{Z}} |\lambda_{li}|^2 \int_{s\wedge \tau_L}^{t\wedge\tau_L}\iD  (\id-\Delta_N)\dd\cdot   (\bs d \times\g{li} )\times \g{li}  \dx  \dtau \\
&-\sum_{i=1}^3\sum_{l\in \mathds{Z}} \int_{s\wedge \tau_L}^{t\wedge\tau_L}  \lambda_{li}\iD (\id-\Delta_N)\dd\cdot   (\bs d \times \g{li})\dx  \mathrm{d}\beta_l^i\,.
\end{split}
    \end{align}

   Now, we add both relations~\eqref{strongenergy} and~\eqref{eq:envarineq2inserted}, add and subtract $|\D|$, 
   and take expectations in order to infer 
 \begin{align}
        \label{eq:envarineq2insertedtwo}
    \begin{split}
    &\expected{\left[ E - ( \nabla \bs d , \nabla \dd) + \tfrac{1}{2}\| \nabla \dd \|_{L^2(\D)}^2 + |\D| - (\bs d , \dd )  \right]\Big|_{(s\wedge \tau_L)-}^{t\wedge\tau_L}} \\
    &+ \expected{\int_{s\wedge \tau_L}^{t\wedge\tau_L}\int_{\D} |\div \bs q |^2 +| \dd \times \Delta \dd | ^2- \trkla{\div \bs q}  \cdot ( \bs d \times  \Delta \dd )  \dx\dr 
    }
   \\
   & +\expected{ \int_{s\wedge \tau_L}^{t\wedge\tau_L} \iD    
   \frac{1}{4} |\dd - \Delta \dd|^2 \left[  |\bs d|^2 - 1\right] - \nabla 
    \left[ ( \dd \times ( \dd \times \Delta \dd )
    \right]: \nabla \bs d\dx \dr}\\
 \leq {} & \expected{\tfrac12\sum_{i=1}^3\sum_{l\in\mathds{Z}}\tabs{\lambda_{li}}^2\int_{s\wedge \tau_L}^{t\wedge\tau_L}\rkla{ \norm{\nabla \g{li}}_{L^2\trkla{\D}}^2 - \norm{\bs d \cdot \nabla \g{li} }_{L^2\trkla{\D}}^2 -  \rkla{\bs q , ( \g{li}\times \nabla   \g{li} )} 
}\dr}
\\
&+\expected{\int_{s\wedge \tau_L}^{t\wedge\tau_L}    \tfrac{1}{2}\sum_{i=1}^3\sum_{l\in\mathds Z}
\tabs{\lambda_{li}}^2 \left[  \|\nabla \g{li}\|_{L^2(\D)}^2 - \| \dd \cdot \nabla \g{li}\|^2_{L^2(\D)} - ( \dd \times \nabla \dd , \g{li}\times \nabla \g{li})  \right]\dr}\\
&-\expected{ \int_{s\wedge \tau_L}^{t\wedge\tau_L} \int_{\D} \nabla \left[   \tfrac{1}{2}\sum_{i=1}^3\sum_{l\in\mathds Z} \tabs{\lambda_{li}}^2(( \dd \times \g{li})\times \g{li})\right] :\nabla \bs d \dx \dr } \\
& {\bs{-}}\expected{
  \sum_{i=1}^3
\sum_{l\in\mathds{Z}}\tabs{\lambda_{li}}^2\int_{s\wedge \tau_L}^{t\wedge\tau_L} \rkla{\nabla(\dd \times \g{li}) ,\nabla ( \bs d \times \g{li})} \dr }\\
& -\expected{\tfrac{1}{2}\sum_{i=1}^3\sum_{l\in \mathds{Z}} \int_{s\wedge \tau_L}^{t\wedge\tau_L}\iD |\lambda_{li}|^2 \nabla [ (\bs d \times\g{li} )\times \g{li} ] :\nabla\dd \dx \dr }\\
 &-\expected{ \int_{s\wedge \tau_L}^{t\wedge\tau_L} \int_{\D} \div \bs q  \cdot ( \bs d \times   \dd ) \dx\dr}
\end{split}\notag \\
\begin{split}&+\expected{ \int_{s\wedge \tau_L}^{t\wedge\tau_L} \int_{\D}  \left[ ( \dd \times ( \dd \times \Delta \dd ) - \tfrac{1}{2}\sum_{i=1}^3\sum_{l\in\mathds Z} \tabs{\lambda_{li}}^2(( \dd \times \g{li})\times \g{li})\right] \cdot \bs d \dx \dr }  \\
& {\bs{-}}\expected{ 
 \sum_{i=1}^3
\sum_{l\in\mathds{Z}}\tabs{\lambda_{li}}^2\int_{s\wedge \tau_L}^{t\wedge\tau_L} \rkla{\dd \times \g{li} ,  \bs d \times \g{li}} \dr }\\
& -\expected{\tfrac{1}{2}\sum_{i=1}^3\sum_{l\in \mathds{Z}} \int_{s\wedge \tau_L}^{t\wedge\tau_L}\iD |\lambda_{li}|^2 \trkla{(\bs d \times\g{li} )\times \g{li}} \cdot \dd \dx \dr }\,.
\end{split}
    \end{align} 
First, we observe that 
$$
\frac{1}{2}\rkla{( \dd \times \g{li})\times \g{li} , \bs d} +\rkla{\dd \times \g{li} ,  \bs d \times \g{li}}+\frac{1}{2}\rkla{\trkla{(\bs d \times\g{li} )\times \g{li}} , \dd} = 0
$$
vanishes due to the properties of the cross product. 
Moreover, we rewrite  
\begin{multline}
E - ( \nabla \bs d , \nabla \dd) + \tfrac{1}{2}\| \nabla \dd \|_{L^2(\D)}^2 + |\D|- (\bs d , \dd ) \\
= \left[E - \tfrac{1}{2}\| \nabla \bs d \|_{L^2(\D)}^2\right] + \frac{1}{2}\left[\iD1 - |\bs d |^2\dx\right] + \left[\frac{1}{2}\| \bs d - \dd \|_{H^1(\D)}^2\right] \,,
\end{multline}
where each of the three summands on the right-hand side is nonnegative. 
    For the remaining terms on the left-hand side of~\eqref{eq:envarineq2insertedtwo}, we find 
\begin{equation}
    \begin{aligned}
        |\div \bs q |^2&+| \dd \times \Delta \dd | ^2 - \trkla{\div \bs q}  \cdot ( \bs d \times  \Delta \dd ) 
        \\
        ={}&
        |\div \bs q - \dd \times  \Delta \dd |^2  + ( \div \bs q) \cdot (\dd \times \Delta \dd )   
        \\
        & +  (\div \bs q - \dd \times \Delta \dd ) \cdot ((\dd -  \bs d) \times  \Delta \dd )+ ( \dd \times \Delta \dd) \cdot (( \dd - \bs d ) \times \Delta \dd ) \,.
    \end{aligned}
\end{equation}
After an integration-by-parts, we may use the formulation for $\bs q$, \textit{i.e.,}~\eqref{eq:formulationq}, 
\begin{align*}
        - ( \bs q , \nabla (\dd\times\Delta \dd) ) 
    ={}& \| \nabla (\dd\times\Delta \dd)\|^2_{L^\infty(\D)}  \left  ( \bs q , -\frac{\nabla (\dd\times\Delta \dd)}{\| \nabla (\dd\times\Delta \dd)\|_{L^\infty(\D)}} \right  )
    \\
    \geq {}&
    - (\bs d \times \nabla \bs d , \nabla (\dd\times\Delta \dd) )  \\& -  \| \nabla (\dd\times\Delta \dd)\|_{L^\infty(\D)} \left[ E - \tfrac{1}{2}\| \nabla \bs d \|_{L^2(\D)}^2+ \int_{\D}\left[ 1 - |\bs d | ^2\right] \dx  \right] \,.
\end{align*}
If $\norm{\nabla\trkla{\dd\times\Delta\dd}}_{L^\infty\trkla{\D}}=0$, the above estimate is trivial.
Thus, we end up with 
\begin{align}
\begin{split}
     & \int_{\D} |\div \bs q |^2 +| \dd \times \Delta \dd | ^2- \div \bs q  \cdot ( \bs d \times  \Delta \dd )  \dx
   \\
   & -\iD    
    \nabla 
    \left[  \dd \times ( \dd \times \Delta \dd )
    \right]: \nabla \bs d\dx \\
    \geq{}&  \int_{\D}|\div \bs q - \dd \times  \Delta \dd |^2  - (\bs d \times \nabla \bs d ): (\nabla (\dd\times\Delta \dd) ) \dx\\
    & +\int_{\D}  (\div \bs q - \dd \times \Delta \dd ) \cdot ((\dd -  \bs d) \times  \Delta \dd ) \dx 
        \\
        &+ \int_{\D}( \dd \times \Delta \dd) \cdot (( \dd - \bs d ) \times \Delta \dd )\dx -\iD    
    \nabla 
    \left[  \dd \times ( \dd \times \Delta \dd )
    \right]: \nabla \bs d\dx
    \\
    &-\| \nabla (\dd\times\Delta \dd)\|_{L^\infty(\D)} \left[ E - \tfrac{1}{2}\| \nabla \bs d \|_{L^2(\D)}^2+ \int_{\D} 1 - |\bs d | ^2 \dx  \right]\,.
    \end{split}
\end{align}
With an integration-by-parts and product rule, we find 
\begin{multline*}
    \int_{\D}( \dd \times \Delta \dd) \cdot (( \dd - \bs d ) \times \Delta \dd ) - (\bs d \times \nabla \bs d ): (\nabla (\dd\times\Delta \dd) )-    
    \nabla 
    \left[ ( \dd \times ( \dd \times \Delta \dd )
    \right]: \nabla \bs d\dx
    \\
    ={}- \iD \sum_{j=1}^3\left[\partial_j ( \dd \times \Delta \dd) : (( \dd -\bs d ) \times \partial_j \dd) + (\dd\times\Delta\dd) \cdot( \partial_j (\dd -\bs d) \times \partial_j\dd)\right]\dx 
    \\
    {}+ \iD \nabla (\dd\times\Delta \dd) :( (\dd-\bs d)\times\nabla \bs d ) + \sum_{j=1}^3 (\dd\times\Delta\dd) \cdot (\partial_j \dd\times\partial_j \bs d) \dx 
    \\
    ={}- \iD \sum_{j=1}^3\left[\partial_j ( \dd \times \Delta \dd) \cdot (( \dd -\bs d ) \times \partial_j \dd- (\dd-\bs d) \times \partial_j \bs d ) \right]\dx \,,
\end{multline*}
where we used the cancellation effects of the cross product, \textit{i.e.}, $ \partial_j\dd \times \partial_j\dd= 0 $ as well as $ \partial_j \bs d  \times \partial_j \dd+ \partial_j \dd \times \partial_j \bs d = 0 $. 
Combining this with the previous calculations and neglecting a nonnegative term, we end up with 
\begin{align}
\begin{split}
     \int_{\D} & |\div \bs q |^2 +| \dd \times \Delta \dd | ^2- \div \bs q  \cdot ( \bs d \times  \Delta \dd )      
    -\nabla 
    \left[ ( \dd \times ( \dd \times \Delta \dd )
    \right]: \nabla \bs d\dx \\
    \geq{}&  \int_{\D} \tfrac{1}{2}|\div \bs q - \dd \times  \Delta \dd |^2  - \frac{1}{2} |(\dd -  \bs d) \times  \Delta \dd |^2 \dx 
        \\
        &- \iD \sum_{j=1}^3\left[\partial_j ( \dd \times \Delta \dd) : (( \dd -\bs d ) \times (\partial_j \dd- \partial_j \bs d ) \right]\dx 
    \\
    &-\| \nabla (\dd\times\Delta \dd)\|_{L^\infty(\D)} \left[ E - \tfrac{1}{2}\| \nabla \bs d \|_{L^2(\D)}^2+ \int_{\D} 1 - |\bs d | ^2 \dx  \right]\,.
    \end{split}
\end{align}
For the terms, coming from the test with the lower-order terms, we find
\begin{align*}
    -& \int_{\D} \div \bs q  \cdot ( \bs d \times   \dd )- \left[ ( \dd \times ( \dd \times \Delta \dd ) - \tfrac{1}{2}\sum_{i=1}^3\sum_{l\in\mathds Z} (( \dd \times \lambda_{li}\g{li})\times \lambda_{li}\g{li})\right] \cdot \bs d \dx   \\
& - 
\int_{\D}\sum_{i=1}^3\sum_{l\in\mathds{Z}}(\dd \times \lambda_{li}\g{li}) \cdot ( \bs d \times \lambda_{li}\g{li})\dx 
 \\
& -\frac{1}{2}\sum_{i=1}^3\sum_{l\in \mathds{Z}} \iD |\lambda_{li}|^2 \trkla{(\bs d \times\g{li} )\times \g{li}} \cdot \dd \dx 
\\
={}&  - \iD ( \div \bs q - \dd \times \Delta \dd) \cdot ( \bs d \times \dd ) \dx\\
&+ \sum_{i=1}^3\sum_{l\in\mathds{Z}}\tabs{\lambda_{li}}^2\ekla{\rkla{ \tfrac{1}{2} ( \dd \times \g{li}) , \bs d \times \g{li}} - \rkla{\dd \times \g{li} ,  \bs d \times \g{li}}+ \rkla{\tfrac{1}{2} (\bs d \times \g{li}) ,  \dd\times \g{li} } }
\\
\leq {}& \iD\frac{1}{4}| \div \bs q - \dd \times \Delta \dd|^2 + | ( \bs d -\dd)\times \dd|^2 \dx \,.
\end{align*}
We used that the second to the last line vanishes and again the properties of the cross product for instance to infer that $  \bs d \times \dd = ( \bs d - \dd ) \times \dd$. 

In order to consider the remaining terms stemming from the \Ito-correction terms, we first use a calculation rule similar to the one derived in Lemma~\ref{lem} for the mixed terms. Therefore, we use the cross-product identities $ (a\times b) \times b = (a\cdot b ) b - |b|^2 a $ as well as~\eqref{crossproduct} such that
\begin{align}
    & \left[  \nabla \left[   \tfrac{1}{2} (( \dd \times \g{li})\times \g{li})\right] :\nabla \bs d    +
 \nabla(\dd \times \g{li}) :\nabla ( \bs d \times \g{li})+\tfrac{1}{2}  \nabla [ (\bs d \times\g{li} )\times \g{li} ] :\nabla\dd\right] \notag
 \\
 &={} \tfrac{1}{2}\nabla \left[ (\dd \cdot \g{li}) \g{li} - |\g{li}|^2 \dd \right ]: \nabla \bs d + \tfrac{1}{2}\nabla \left[ (\bs d \cdot \g{li}) \g{li} - |\g{li}|^2 \bs d  \right ]: \nabla \dd\notag\\
 &\quad + \sum_{j=1}^3\left[\partial_j \dd \times \g{li} + \dd \times \partial_j\g{li}  \right]\cdot \left[\partial_j \bs d \times \g{li} + \bs d \times \partial_j\g{li}  \right]\notag
 \\
 &={} \tfrac{1}{2}\sum_{j=1}^3 \left[  (\partial_j \dd \cdot \g{li} )(\partial_j\bs d\cdot \g{li}) + ( \dd \cdot \partial_j \g{li}) (\g{li}\cdot \partial_j \bs d) + (\dd \cdot \g{li})(\partial_j \g{li}\cdot \partial_j\bs d) 
\right]\notag \\
 &\quad -\tfrac{1}{2}\sum_{j=1}^3 \left[   2 (\partial_j\g{li}\cdot \g{li})(\dd\cdot \partial_j\bs d ) + |\g{li}|^2 \partial_j\dd \cdot \partial_j \bs d)\right] \notag\\&
 \quad+\tfrac{1}{2}\sum_{j=1}^3 \left[  (\partial_j \bs d \cdot \g{li} )(\partial_j\dd \cdot\g{li}) + ( \bs d \cdot \partial_j \g{li}) (\g{li}\cdot \partial_j \dd) + (\bs d \cdot \g{li})(\partial_j \g{li}\cdot \partial_j\dd) \right]\notag
 \\&
 \quad-\tfrac{1}{2}\sum_{j=1}^3 \left[   2 (\partial_j\g{li}\cdot \g{li})(\bs d\cdot \partial_j\dd ) +|\g{li}|^2 \partial_j\bs d \cdot \partial_j \dd)\right]\notag\\
 &\quad +\sum_{j=1}^3 \left[ (\partial_j \dd \cdot \partial_j\bs d)|\g{li}|^2 - ( \g{li} \cdot\partial_j\dd)(\g{li} \cdot \partial_j \bs d) +(\partial_j \dd \cdot \bs d )(\g{li}\cdot\partial_j \g{li}) - (\partial_j \dd\cdot\partial_j \g{li}) (\bs d \cdot \g{li}) \right] \notag
 \\
 &\quad + \sum_{j=1}^3
 \left[ (\dd \cdot \partial_j \bs d)(\g{li}\cdot \partial_j\g{li}) - (\dd\cdot\g{li})(\partial_j\bs d \cdot \partial_j \g{li}) + (\dd \cdot \bs d ) |\partial_j \g{li}|^2 -(\dd\cdot\partial_j\g{li})(\bs d \cdot \partial_j \g{li})   \right] \notag
 \\
 ={}& \tfrac{1}{2}\sum_{j=1}^3
\left[ (\dd \cdot \partial_j \g{li}) (\g{li}\cdot\partial_j \bs d) + (\bs d \cdot \partial_j \g{li})(\g{li}\cdot\partial_j \dd) - (\dd \cdot \g{li})(\partial_j \g{li}\cdot\partial_j\bs d) \right]  \notag
\\
&\quad +\sum_{j=1}^3
\left[- \tfrac12(\bs d \cdot \g{li}) (\partial_j \g{li}\cdot\partial_j\dd)+ (\dd \cdot \bs d) |\partial_j \g{li}|^2 - (\dd\cdot \partial_j \g{li})(\bs d \cdot \partial_j \g{li})\right] \notag
\\
= {}& \sum_{j=1}^3
\left[ (\dd \cdot \bs d) |\partial_j \g{li}|^2 - (\dd\cdot \partial_j \g{li})(\bs d \cdot \partial_j \g{li})+ \tfrac{1}{2}(\dd \times \partial_j\bs d+ \bs d \times \partial_j \dd)(\partial_j \g{li}\times \g{li}) \right] \,.
 \end{align}
Inserting this calculation back into the relevant terms of~\eqref{eq:envarineq2insertedtwo} and using the formulation~\eqref{eq:formulationq} in order to estimate
\begin{multline}
    -\int_{\D}  \bs q : ( \g{li}\times \nabla   \g{li}  )\dx \leq{} - \iD ( \bs d \times \nabla \bs d ) : ( \g{li}\times \nabla   \g{li}  )\dx 
\\+  \|  \g{li}\times\nabla  \g{li}\|_{L^\infty(\D)} \left[ E - \tfrac{1}{2}\| \nabla \bs d \|_{L^2(\D)}^2+ \int_{\D} 1 - |\bs d | ^2 \dx  \right]\,
\end{multline}
again, we find
 \begin{align}
 \begin{split}
     & \tfrac12
     \ekla{\norm{\nabla\g{li}}_{L^2\trkla{\D}}^2-\norm{\bs{d}\cdot\nabla\g{li}}_{L^2\trkla{\D}}^2-\rkla{\bs{q},\g{li}\times\nabla\g{li}}}
\\
&+    \tfrac{1}{2}
\left[  \|\nabla \g{li}\|_{L^2(\D)}^2 - \| \dd \cdot \nabla \g{li}\|^2_{L^2(\D)} - ( \dd \times \nabla \dd , \g{li}\times \nabla \g{li})  \right]\\&- \int_{\D} \nabla \left[   \tfrac{1}{2} (( \dd \times \g{li})\times \g{li})\right] :\nabla \bs d \dx   -
\rkla{\nabla(\dd \times \g{li}) ,\nabla ( \bs d \times \g{li})}\\&
  -\tfrac{1}{2}
 \iD  \nabla [ (\bs d \times\g{li} )\times \g{li} ] :\nabla\dd \dx 
 \\
 \leq{}&\tfrac12
\ekla{ \norm{\nabla \g{li}}_{L^2\trkla{\D}}^2 - \norm{\bs d \cdot \nabla \g{li} }_{L^2\trkla{\D}}^2 -  \rkla{\bs d \times \nabla \bs d , \g{li}\times \nabla  \g{li}  }}
\\
&+    \tfrac{1}{2}
\left[  \|\nabla \g{li}\|_{L^2(\D)}^2 - \| \dd \cdot \nabla \g{li}\|^2_{L^2(\D)} - ( \dd \times \nabla \dd , \g{li}\times \nabla \g{li})  \right]\\
&- \int_{\D} \sum_{j=1}^3
\left[ (\dd \cdot \bs d) |\partial_j \g{li}|^2 - (\dd\cdot \partial_j \g{li})(\bs d \cdot \partial_j \g{li})- \tfrac{1}{2}(\dd \times \partial_j\bs d+ \bs d \times \partial_j \dd)( \g{li}\times \partial_j\g{li}) \right]  \dx 
 \\
 &+ 
 \| \g{li} \times \nabla \g{li}\|_{L^\infty(\D)} \left[ E - \tfrac{1}{2}\|\nabla \bs d\|_{L^2(\D)}^2 + \frac{1}{2}\iD 1-|\bs d|^2\dx\right]\\
 ={}&
 \sum_{j=1}^3\iD | \partial_j \g{li}|^2 ( 1 - \dd\cdot \bs d) - \frac{1}{2}| \dd\cdot \partial_j \g{li}-\bs d \cdot \partial_j \g{li}|^2 \dx 
 \\
&- \tfrac{1}{2}
\sum_{j=1}^3\iD   ((\bs d -\dd)\times \partial_j(\bs d-\dd))\cdot ( \g{li}\times \partial_j\g{li})\dx
 \\
 &+\| \g{li} \times \nabla \g{li}\|_{L^\infty(\D)} \left[ E - \tfrac{1}{2}\|\nabla \bs d\|_{L^2(\D)}^2 + \frac{1}{2}\iD 1-|\bs d|^2\dx\right]\,.
 \end{split}
 \end{align}
Now, we are able to conclude by inserting everything back into~\eqref{eq:envarineq2insertedtwo}, which leads to 
 \begin{equation}
        \label{eq:envarineq2insertedfinal}
    \begin{aligned}
    &\expected{\left[ E - \tfrac{1}{2}\| \nabla \bs d \|_{L^2(\D)}^2 + \frac12\iD 1 - |\bs d |^2\dx + \frac{1}{2}\| \bs d - \dd \|_{H^1(\D)}^2  \right]\Big|_{(s\wedge \tau_L)-}^{t\wedge\tau_L}} \\
    &+ \expected{\int_{s\wedge \tau_L}^{t\wedge\tau_L}\int_{\D} \tfrac{1}{4}|\div \bs q - \dd \times  \Delta \dd |^2   + \frac{1}{2}\sum_{j=1}^3\sum_{i=1}^3\sum_{l\in\mathds Z}|\lambda_{li}|^2| \dd\cdot \partial_j \g{li}-\bs d \cdot \partial_j \g{li}|^2 \dx\dr 
    }
   \\
 \leq {} & 
\expected{\tfrac14\sum_{i=1}^3\sum_{l\in\mathds{Z}}\tabs{\lambda_{li}}^2\int_{s\wedge \tau_L}^{t\wedge\tau_L}\int_{\D} | \nabla  \g{li}|^2 [ 1 - |\bs d|^2 + |\bs d - \dd|^2] 
\dx\dr}
\\
&+\expected{\tfrac{1}{2}\sum_{i=1}^3\sum_{l\in\mathds{Z}}\tabs{\lambda_{li}}^2\int_{s\wedge \tau_L}^{t\wedge\tau_L}     ((\bs d -\dd)\times \nabla(\bs d-\dd)): ( \g{li}\times \nabla\g{li})\dr}
\\
   & +\expected{ \int_{s\wedge \tau_L}^{t\wedge\tau_L} \sum_{j=1}^3\left[\partial_j ( \dd \times \Delta \dd) : (( \dd -\bs d ) \times (\partial_j \dd- \partial_j \bs d ) \right]\dx \dr}\\
&+\expected{ \int_{s\wedge \tau_L}^{t\wedge\tau_L} \sum_{i=1}^3\sum_{l\in\mathds{Z}}\tabs{\lambda_{li}}^2\| \g{li} \times \nabla \g{li}\|_{L^\infty(\D)} \left[ E - \tfrac{1}{2}\|\nabla \bs d\|_{L^2(\D)}^2 + \frac{1}{2}\iD 1-|\bs d|^2\dx\right] \dr } 
\\ &+\expected{ \int_{s\wedge \tau_L}^{t\wedge\tau_L} \| \nabla (\dd\times\Delta \dd)\|_{L^\infty(\D)} \left[ E - \tfrac{1}{2}\| \nabla \bs d \|_{L^2(\D)}^2+ \int_{\D} 1 - |\bs d | ^2 \dx  \right]
    \dr}
\\&+\expected{ \int_{s\wedge \tau_L}^{t\wedge\tau_L} \int_{\D}  | ( \bs d -\dd)\times \dd|^2 + \frac{1}{2} |(\dd -  \bs d) \times  \Delta \dd |^2\dx \dr }  \,.
\end{aligned}
    \end{equation} 
Under the given regularity requirements, we may estimate the right-hand side, by the relatives errors in the first line of~\eqref{eq:envarineq2insertedfinal} such that we
find
    \begin{align}
    &\expected{\left[ E - \tfrac{1}{2}\| \nabla \bs d \|_{L^2(\D)}^2 + \frac12\iD 1 - |\bs d |^2\dx + \frac{1}{2}\| \bs d - \dd \|_{H^1(\D)}^2  \right]\Big|_{(s\wedge \tau_L)-}^{t\wedge\tau_L}} \notag \\
    &+ \expected{\int_{s\wedge \tau_L}^{t\wedge\tau_L}\int_{\D} \tfrac{1}{4}|\div \bs q - \dd \times  \Delta \dd |^2   + \frac{1}{2}\sum_{j=1}^3\sum_{i=1}^3\sum_{l\in\mathds Z}|\lambda_{li}|^2| \dd\cdot \partial_j \g{li}-\bs d \cdot \partial_j \g{li}|^2 \dx\dr 
    }
  \notag  \\
 \leq {} & C 
\mathds E \Bigg[ \int_{s\wedge \tau_L}^{t\wedge\tau_L}
 \Big(\sum_{i=1}^3\sum_{l\in\mathds{Z}}\tabs{\lambda_{li}}^2
\left[ \|  \nabla  \g{li}\|^2_{L^\infty(\D)}+\| \g{li} \times \nabla \g{li}\|_{L^\infty(\D)} \right]  
\notag \\& \quad \quad + \| \nabla (\dd\times\Delta \dd)\|_{L^\infty(\D)} + 1 + \| \Delta \dd \|_{L^3(\D)}^2 \Big) 
\notag \\ & \quad \qquad \cdot \left( E - \tfrac{1}{2}\| \nabla \bs d \|_{L^2(\D)}^2+\frac12\int_{\D}  [ 1 - |\bs d|^2 ] 
\dx+ \frac{1}{2}\| \bs d - \dd \|_{H^1(\D)}^2 \right) \dr \Bigg] 
\,.
\end{align}
and
may conclude by using Gronwall's lemma that 
$$\expected{\left[ E - \tfrac{1}{2}\| \nabla \bs d \|_{L^2(\D)}^2 + 
\frac12\iD 1 - |\bs d |^2\dx + \frac{1}{2}\| \bs d - \dd \|_{H^1(\D)}^2  \right]\Big|_{0}^{t\wedge\tau_L}}\leq 0\,. $$
As $\tfrac12\norm{\nabla\bs{d}}_{L^2\trkla{\D}}^2\leq E$ and $\tabs{\bs{d}}^2\leq 1$, we especially obtain~$\bs d \equiv \dd$. 
\end{proof}

\section{Numerical simulation of the harmonic heat flow into the sphere}\label{sec:numerics}

In order to illustrate the effect of stochastic fluctuations on the molecular scale, we investigate a numerical implementation of  the algorithm proposed in~\eqref{eq:discscheme} for the harmonic map heat flow perturbed by Stratonovich noise. Stochastic perturbations are a natural way to model thermal or microscopic fluctuations that influence the evolution of molecular alignment. In the physics community it is well established~\cite[Sec.~3.4]{gennes} that the locally averaged molecular orientation, which determines anisotropy, is subject to statistical fluctuations. This is especially important when coupled to the fluid flow of an anisotropic medium, where small changes in the anisotropy may cause vast changes in the fluid flow, \textit{e.g.,} possibly leading to tipping points. \\
We implemented the finite element scheme proposed in~\eqref{eq:discscheme} using the python API of FEniCSx~\cite{fenicsx}. The source code is publicly available~\cite{code}.
In order to be able to simulate a large number of paths in the probability space, the implementation must be computationally efficient.
Therefore, the nonlinear equations are linearized to a semi-implicit scheme in time using a so-called projection method, where the computed numerical solution of the linearized scheme is nodally projected onto the sphere after each time-step. We refer to the recent article~\cite{Max} for details on this method.\\

To highlight the influence of the stochastic fluctuations, we perform numerical experiments starting from an unstable stationary configuration.

Fixing a two dimensional unit square $(-0.5,0.5)^2\subset\mathds{R}^2$ as domain, we apply homogeneous Neumann boundary conditions on the left and right sides of the domain and constant-in-time Dirichlet boundary conditions on the top and bottom of the square pointing in opposite directions given by the initial data $\boldsymbol{d}_0 (x,y) = (\mathrm{sgn} (y), 0)^T$, (cf.~Figure~\ref{fig:states}, top).
The domain is decomposed into an equispaced triangulation such that all its vertices lie on a cartesian grid with mesh width $h=1/11$.
The time-step size is set to $\tau = 0.005$ in all simulations.\\
On this domain, the cylindrical Wiener process $\bs{W}$ is realized using the Fourier base $\{f_l\}_{l\in\mathds{Z}}$ of $L^2\trkla{\D}$, i.e.
$$
f_l
=
\begin{cases}
1 & \text{ if } l = 0 \,,\\ 
\sqrt{2}\sin (2 \pi l x) &\text{ if } l > 0 \,,\\ 
\sqrt{2}\cos (2 \pi l x) &\text{ if } l < 0 \,.
\end{cases}
$$
This induces the functions $\trkla{\g{li}}_{l\in\mathds{Z}}^{i=1,2,3}$ in Equation \eqref{eq:representationW}.
The cutoff for finite-dimensional modes in the numerical realization is chosen as $\mathds{Z}_h = \mathds{Z} \cap [-Z_h, Z_h]$ with the heuristic $Z_h = 2/h-1$. For simplicity, the coefficients $\lambda_{li}$ inducing the trace class operator $\mathcal{Q}$ are set to $\lambda_{li}=1$ for all indices considered in our simulations. As our simulations do not adress the limit $Z_h\rightarrow\infty$, our choice for a finite $Z_h$ does not violate Assumption \ref{item:W3}. If one considers the limit $Z_h\rightarrow\infty$, other choices respecting the decay property of $\trkla{\lambda_i}_{i\in\mathds{Z}}$ are necessary. 

\begin{figure}[ht]
    \centering        
    \includegraphics[width=.4\textwidth,trim= 6cm 6cm 6cm 6cm , clip]{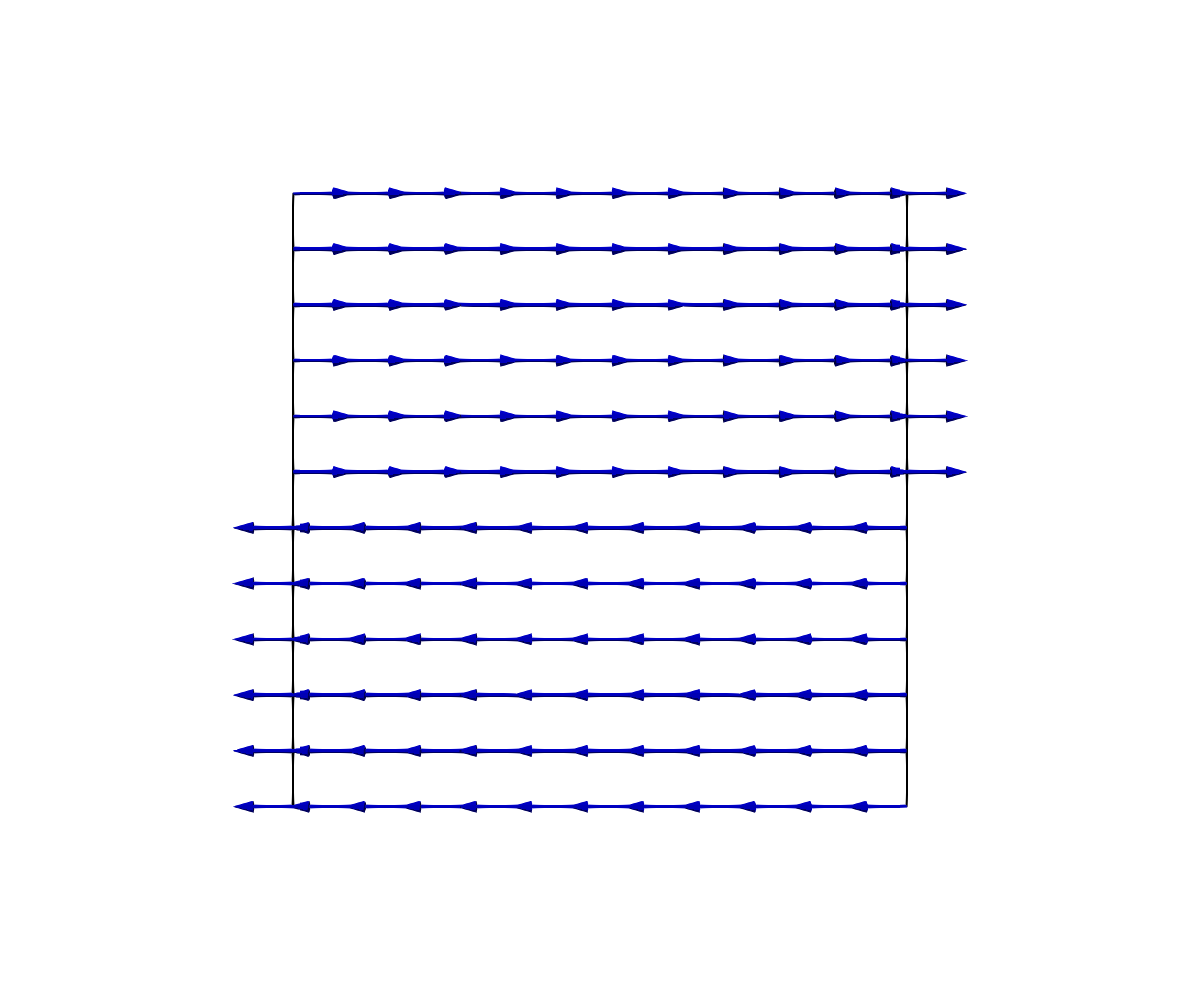}
    \\    
    \includegraphics[width=.4\textwidth,trim= 6cm 6cm 6cm 6cm , clip]{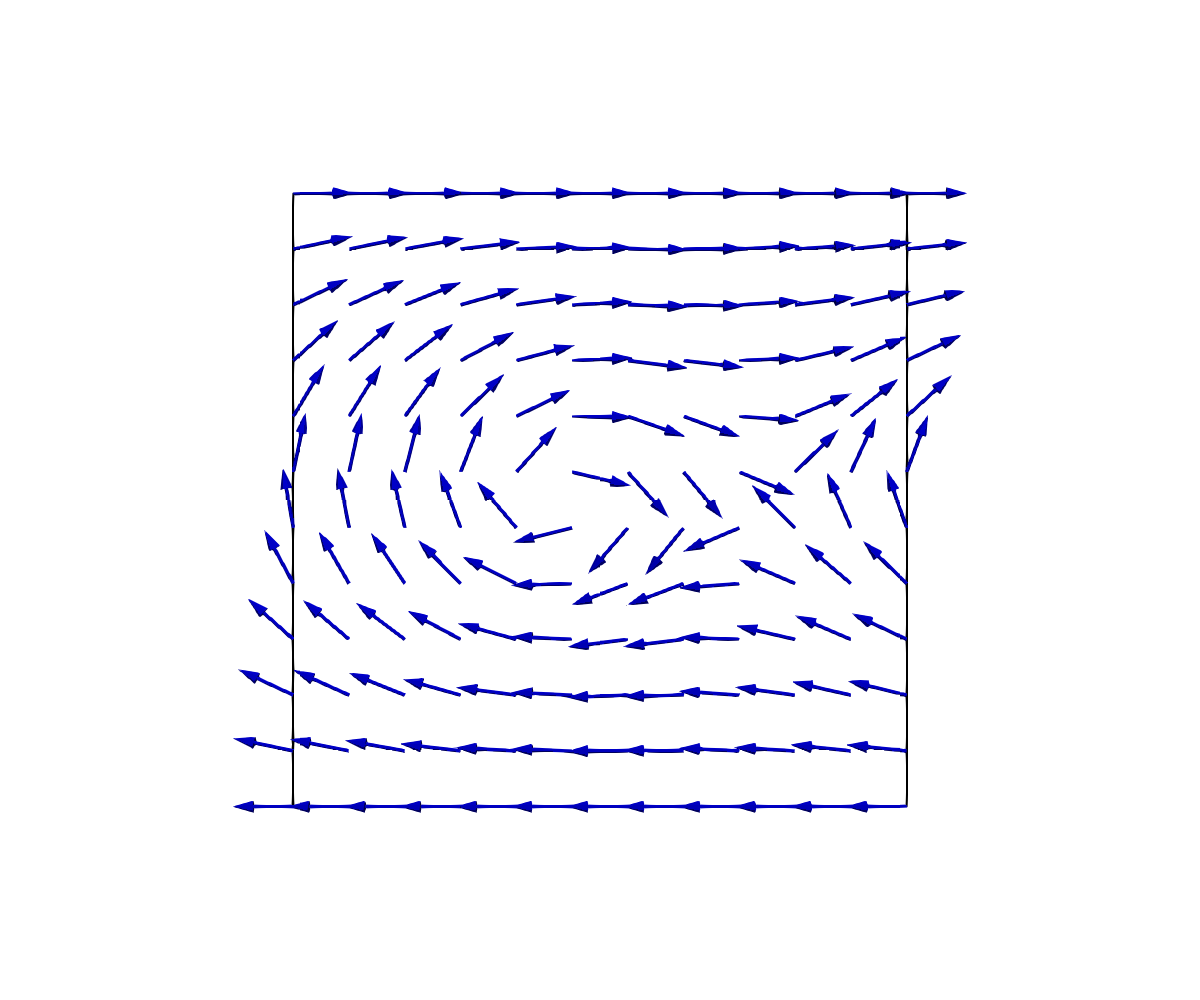}
    \hspace{1em}    
    \includegraphics[width=.4\textwidth,trim= 6cm 6cm 6cm 6cm , clip]{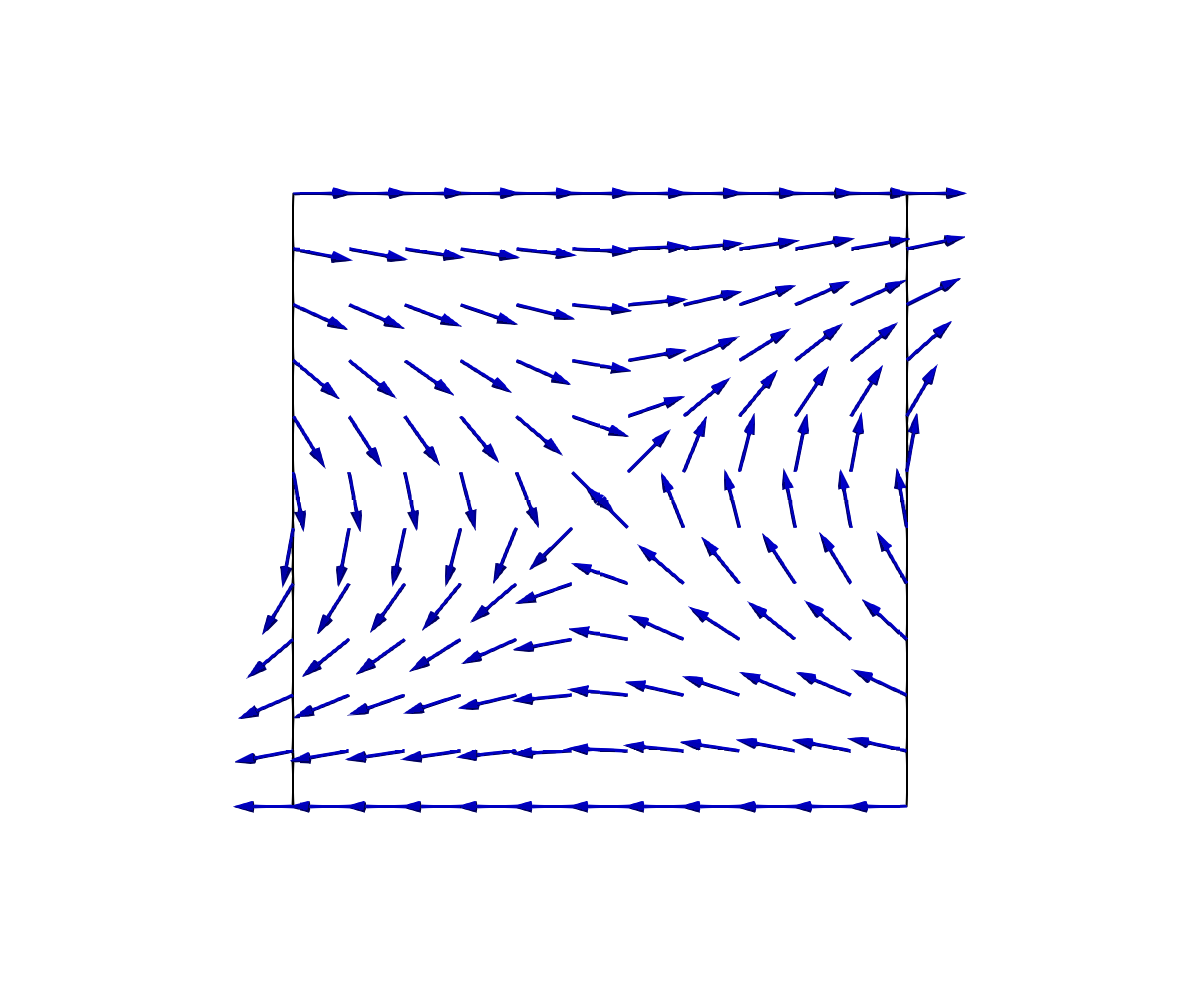}
    \\
    \includegraphics[width=.4\textwidth,trim= 6cm 6cm 6cm 6cm , clip]{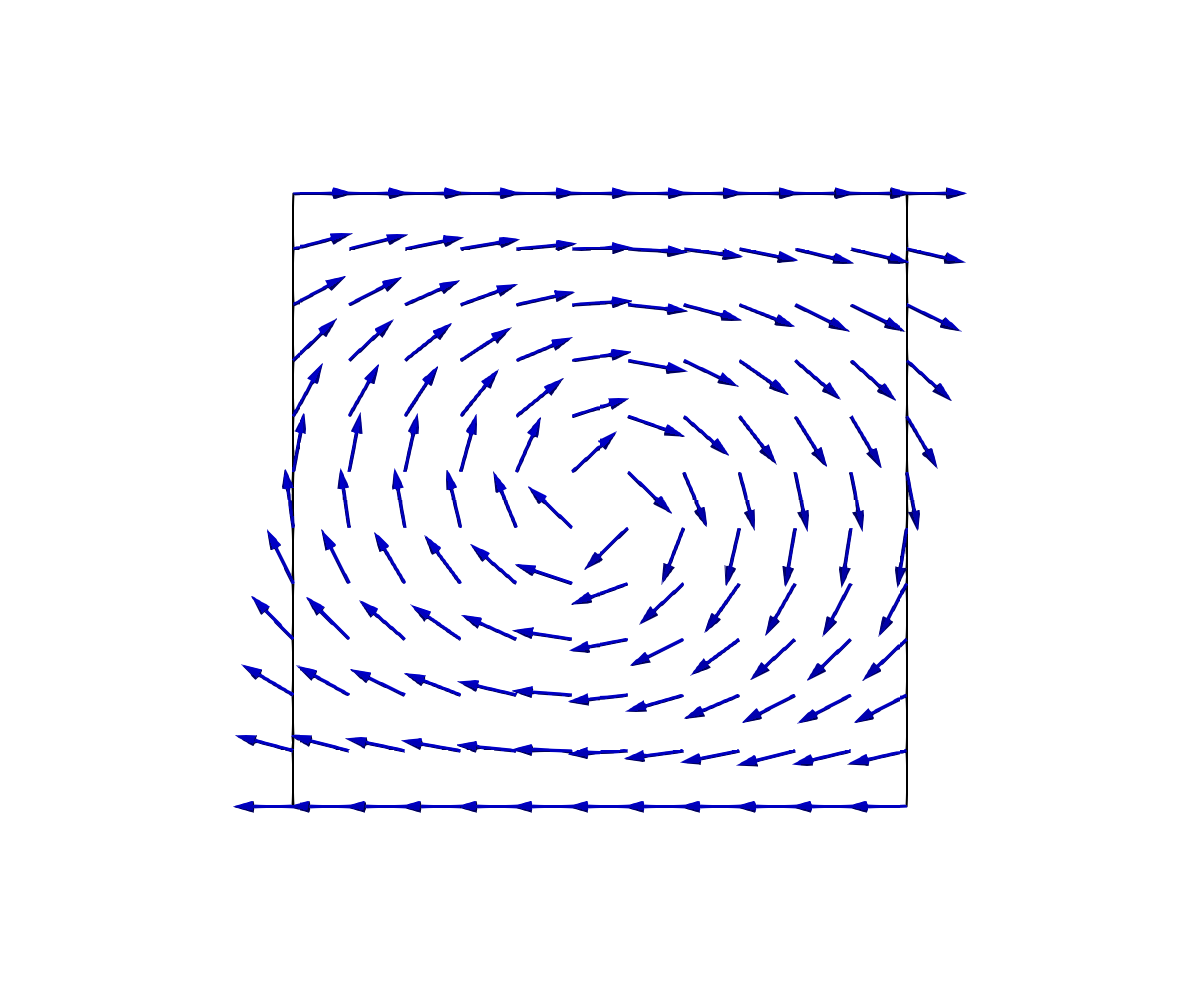}
    \hspace{1em}
    \includegraphics[width=.4\textwidth,trim= 6cm 6cm 6cm 6cm , clip]{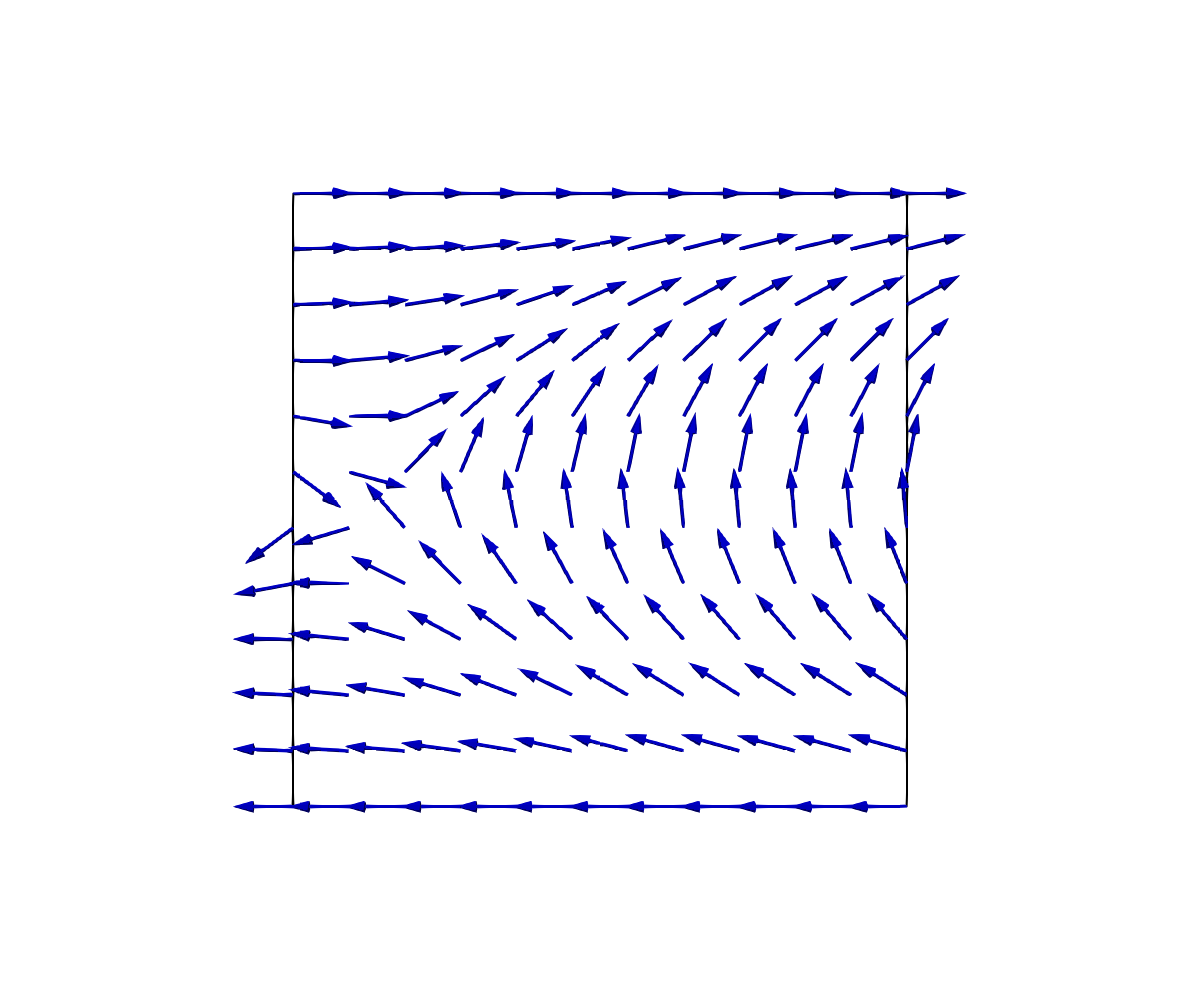}
    \\
    \includegraphics[width=.4\textwidth,trim= 6cm 6cm 6cm 6cm , clip]{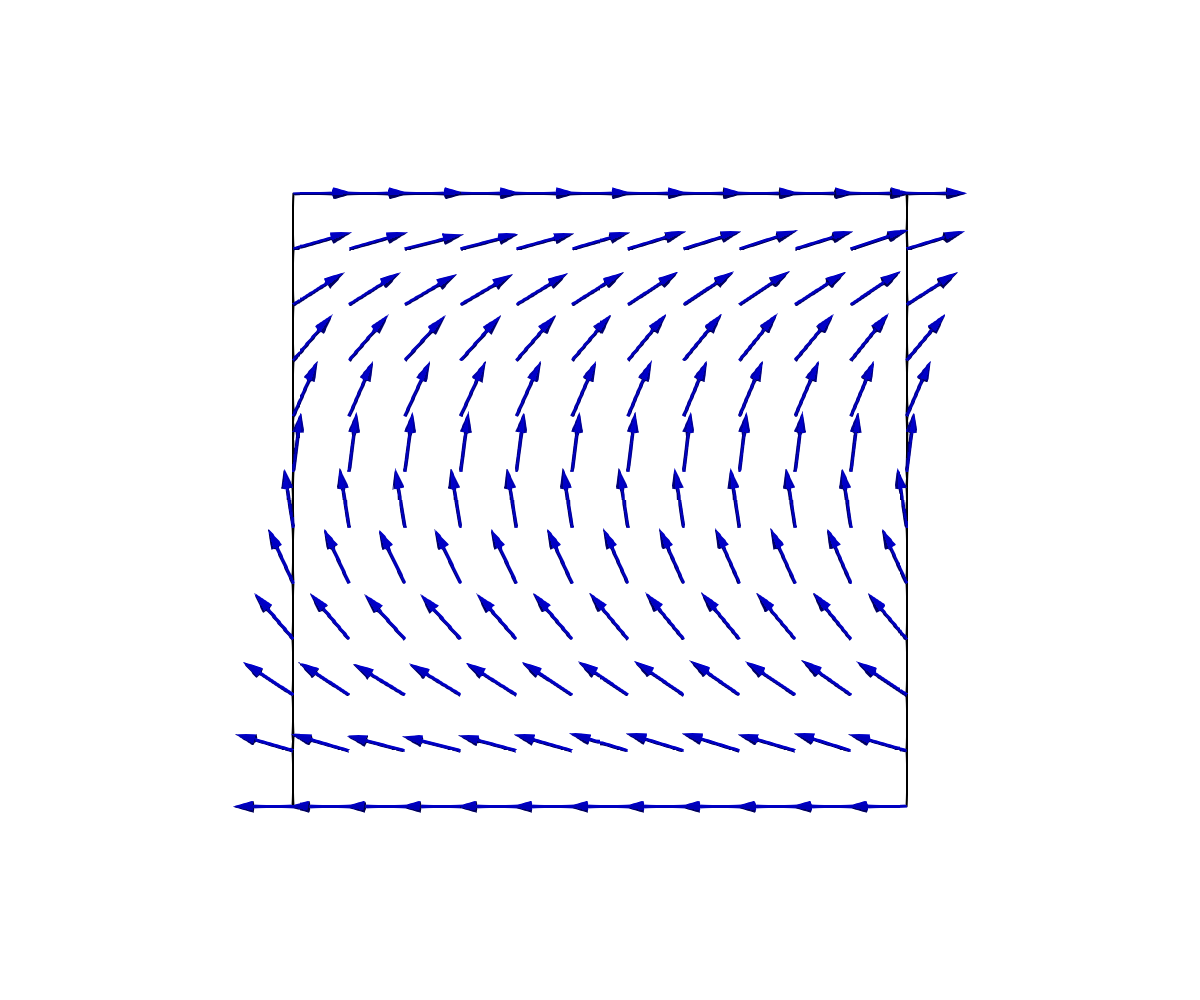}%
    \hspace{1em}
     \includegraphics[width=.4\textwidth,trim= 6cm 6cm 6cm 6cm , clip]{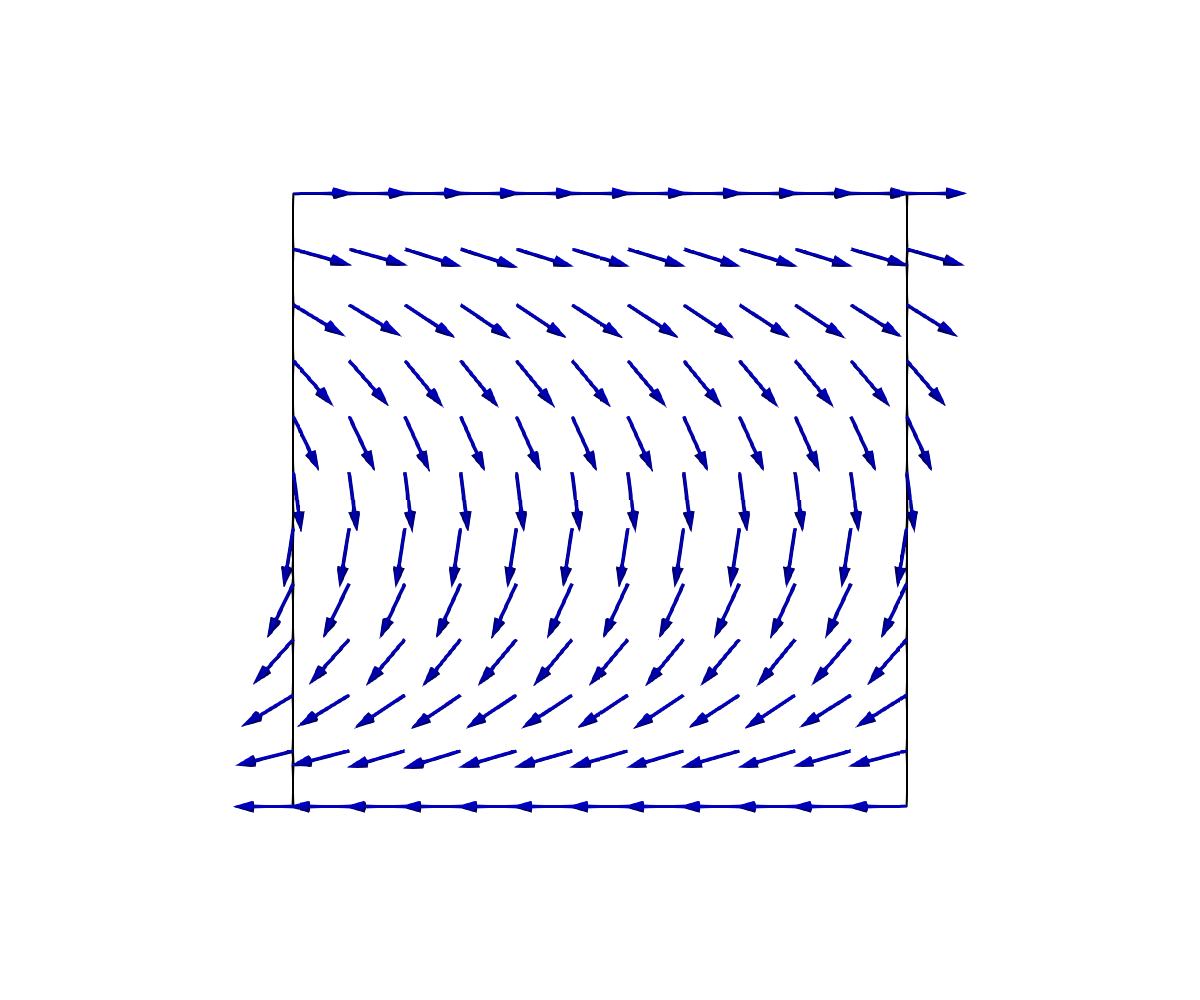}
   
    \caption{Different states of the evolution of $\bs{d}$: Initial configuration (top), local minimizer (second and third row), global minimizers (bottom row). Associated energy levels are $22.0, \, 15.5, \, 10.8 , \, 10.7 , \, 9.8, \,  4.9,\, 4.9$ (from left to right, from top to bottom).}
    \label{fig:states}
\end{figure}

Depending on the realization of the noise, the flow evolves towards one of the two global minimizers of the energy or becomes (temporarily) trapped in a nearby local minimum. Several representative outcomes of these evolutions are shown in the accompanying Figure~\ref{fig:states}. 
In particular Figure~\ref{fig:states} shows the initial data (top row), possible local minima ($2^{nd}$ and $3^{rd}$ row), and the two global minimizers (last row). 
The evolution of one path is visualized in Figure \ref{fig:evolution-local-to-global}. Here, we can observe that the director field first evolves towards a local energy minimum (energy$\approx16$) before finally attaining a configuration that is a global minimizer of the energy (energy$\approx5$). 
\begin{figure}[ht]
    \centering        
    \includegraphics[width=.45\textwidth,trim= 6cm 6cm 6cm 6cm , clip]{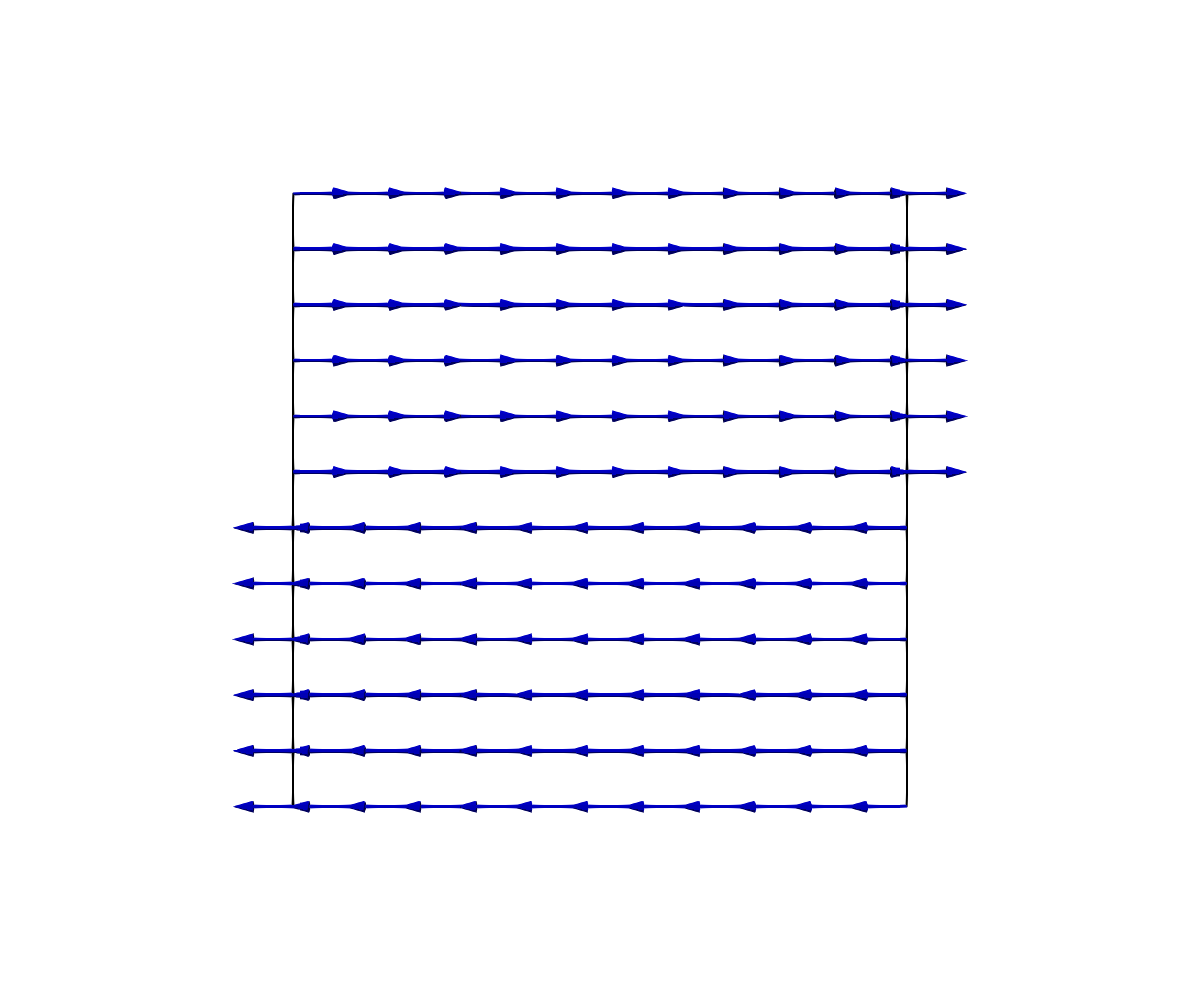}
    \hspace{1em}
    \includegraphics[width=.45\textwidth,trim= 6cm 6cm 6cm 6cm , clip]{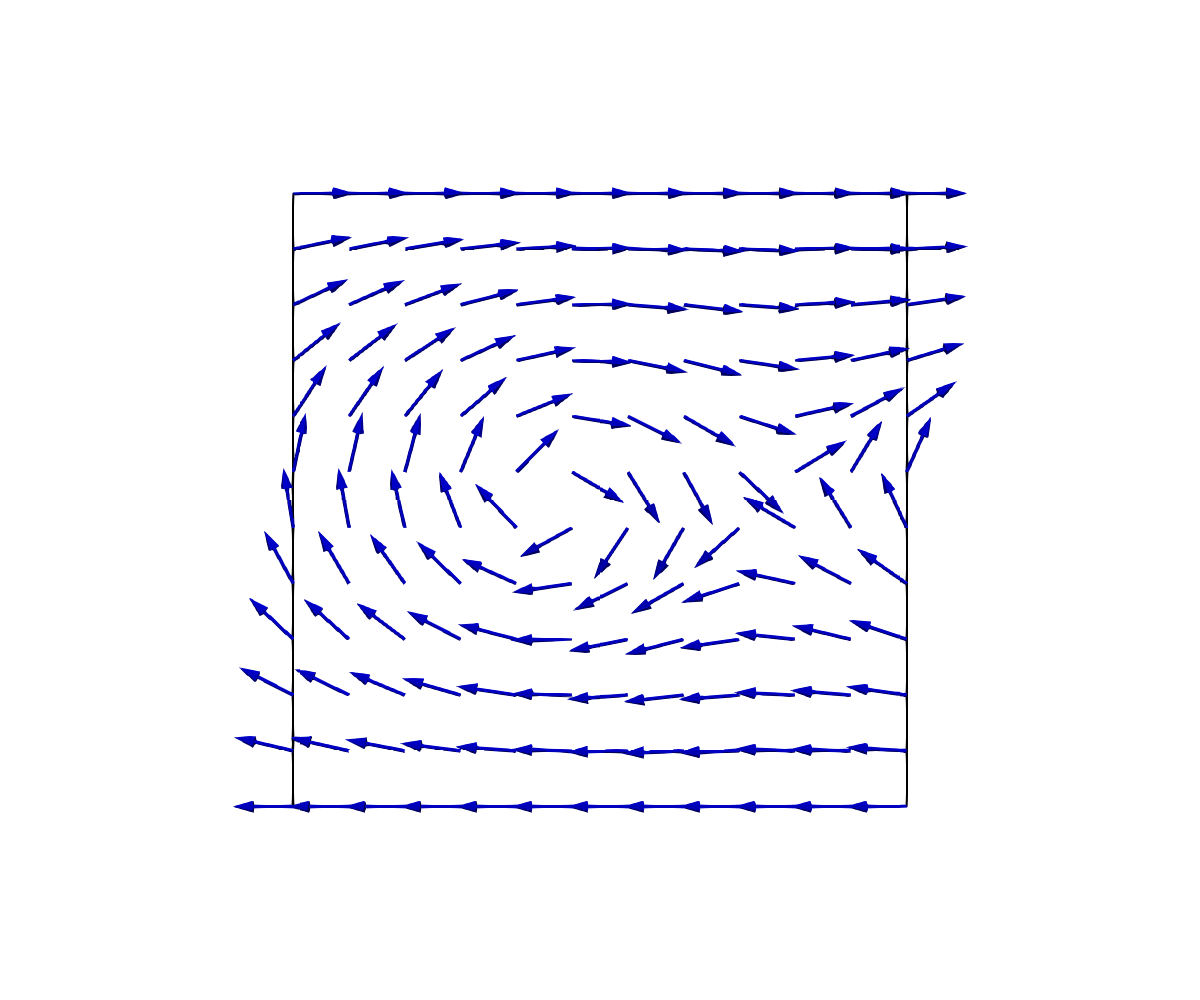}
    \\          
    \includegraphics[width=.45\textwidth,trim= 6cm 6cm 6cm 6cm , clip]{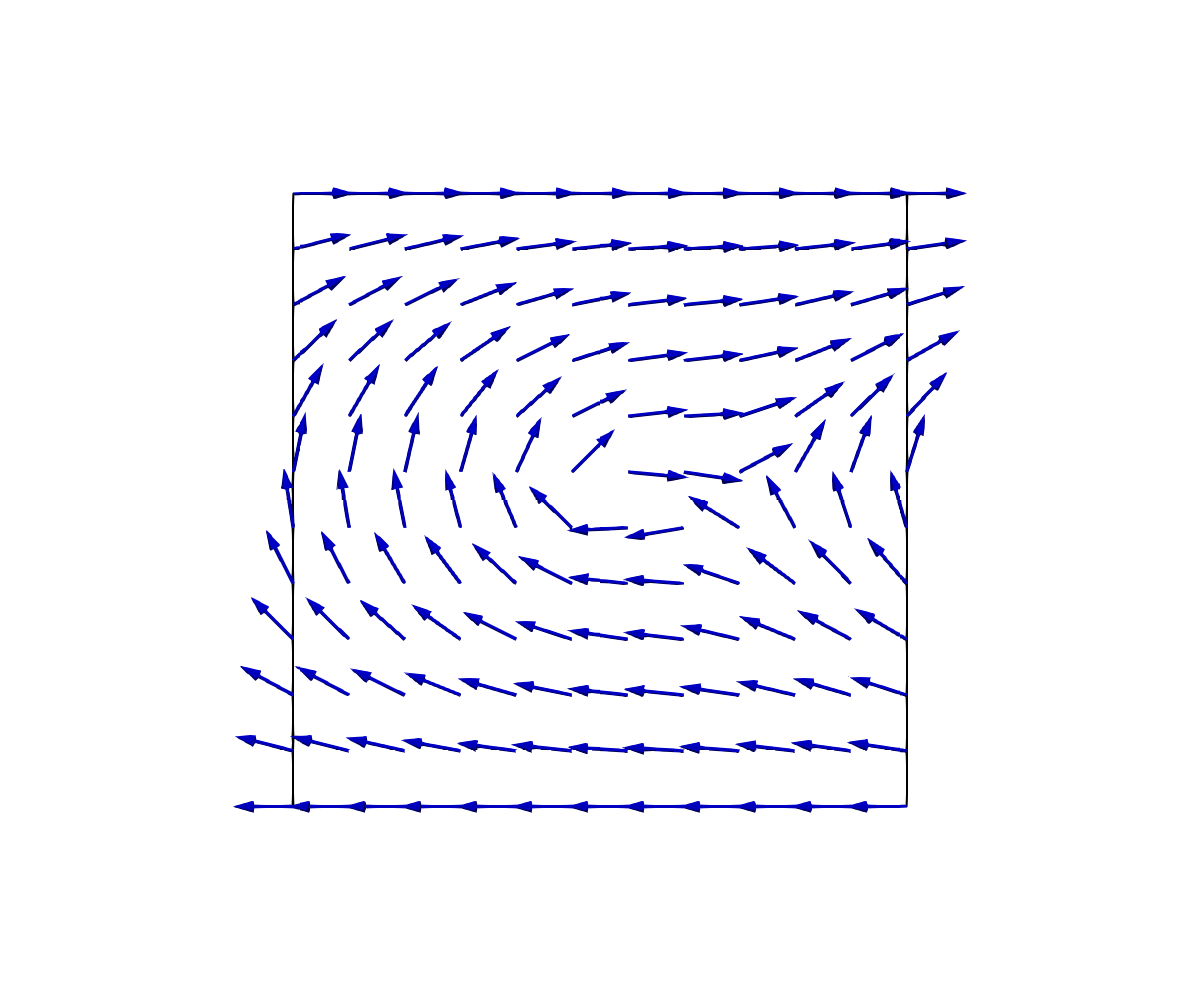}
    \hspace{1em}
    \includegraphics[width=.45\textwidth,trim= 6cm 6cm 6cm 6cm , clip]{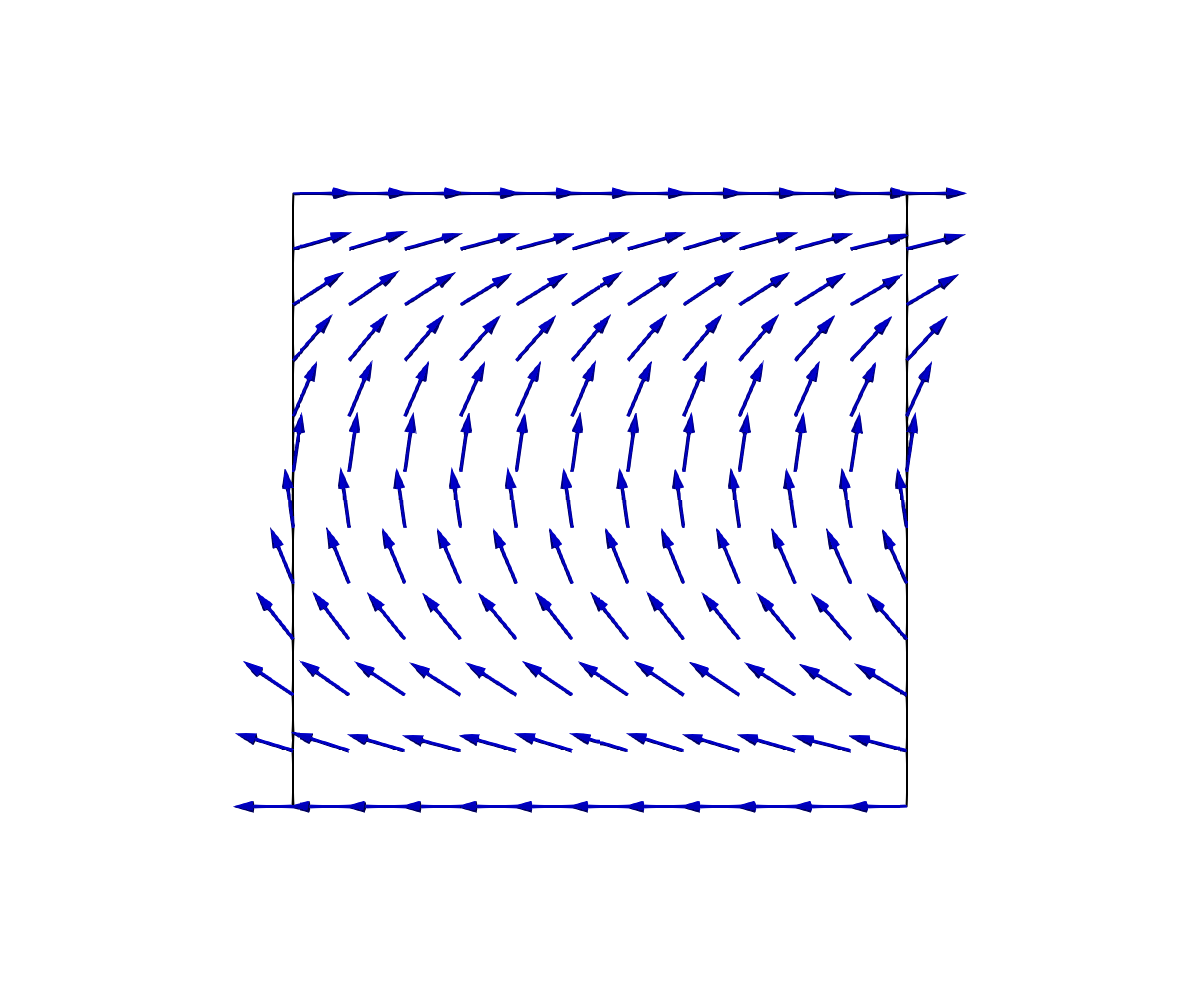}
    \\
    \includegraphics[width=.45\textwidth,trim= 6cm 6cm 6cm 6cm , clip]{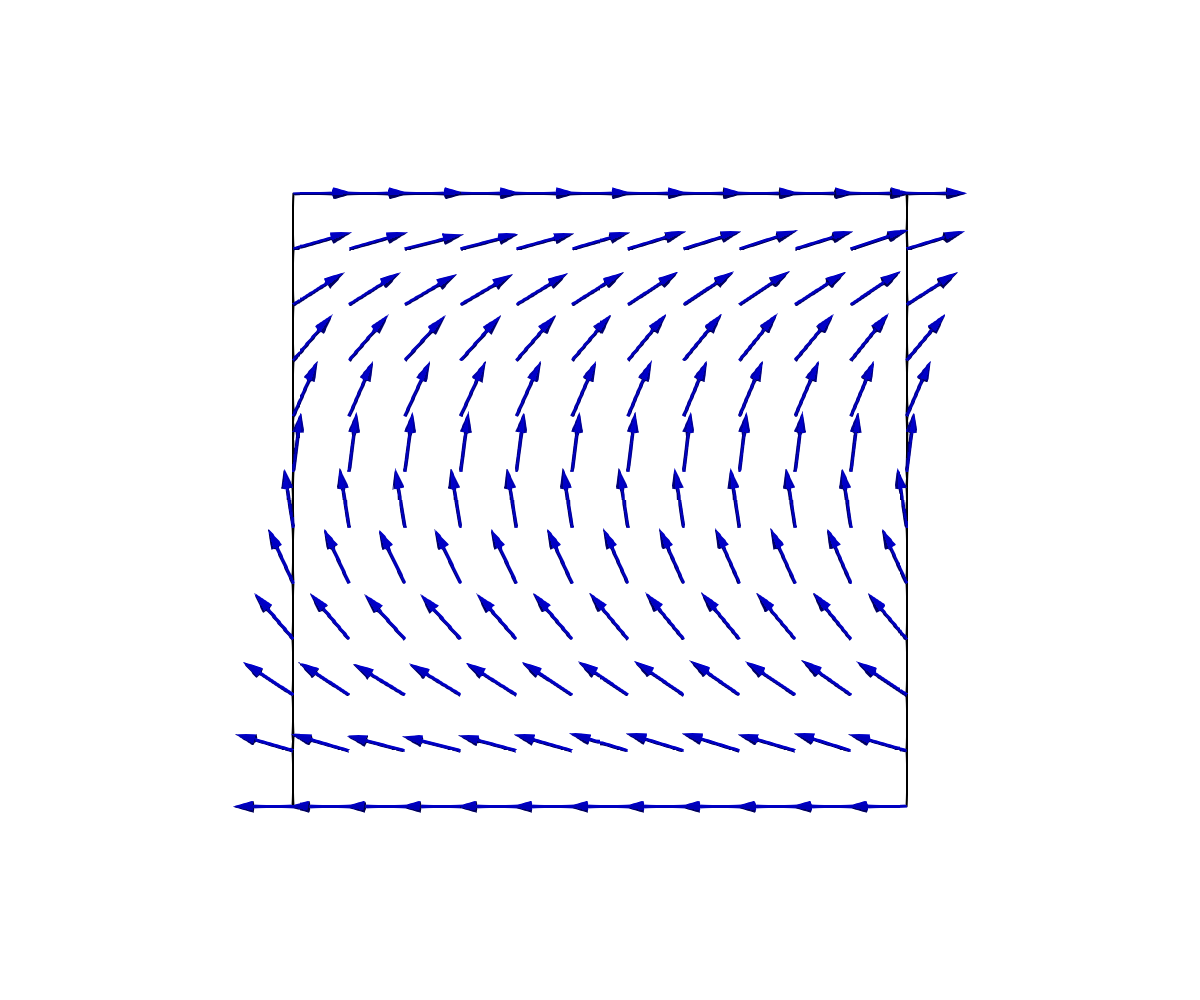}%
    \hspace{1em}
    \includegraphics[width=.45\textwidth]{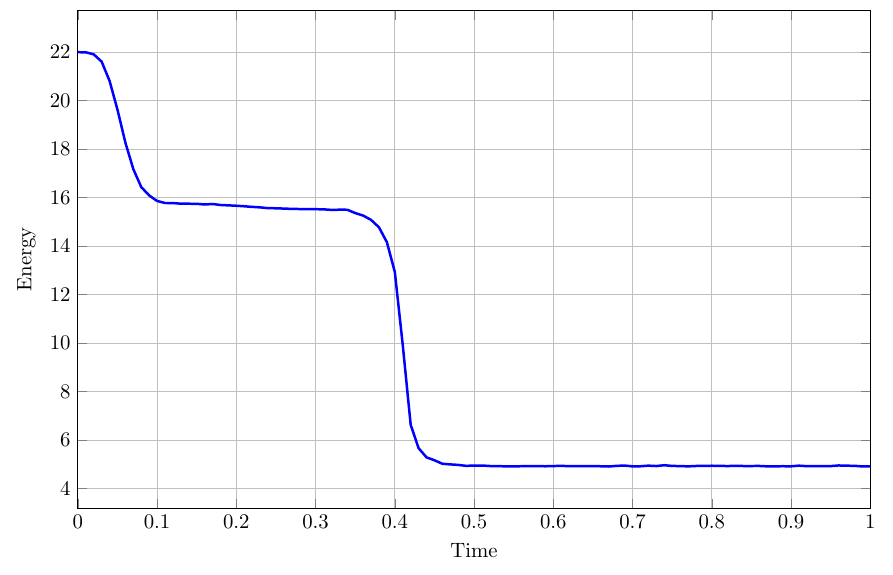}
        \caption{Evolution of the director field along one path at time $t=0.0, \, 0.2, \, 0.4, \, 0.6, \, 1.0$ (from top to bottom, from left to right) and the energy evolution over time (bottom right). The vertical lines indicate the shown time stamps.}
    \label{fig:evolution-local-to-global}
\end{figure}

We next examine the corresponding energy evolution over 500 realizations. The ensemble of energy curves depicted in Figure~\ref{fig:paths-by-noise} illustrates how randomness affects both the rate of relaxation and the likelihood of escaping unstable or metastable states. In Figure~\ref{fig:paths-by-noise}, the energy evolution of the 500 realizations as well as the inferred expectation and quantiles over time are visualized.
Thereby, the different rows depict the results for different noise intensities, which are realized by multiplying the coefficients $\lambda_{li}$ by a parameter $c_S$.
For all considered noise intensities, the configuration of the director field evolves from the initial condition (with energy$=22$) towards energetically more favorable configurations. In particular, we observe that due to local minima in the energy landscape, the different evolutions accumulate at three distinct energy levels at approximately $16$, $11$, and $5$. It is obvious that the levels $16$ and $11$ can only correspond to local minima. Thus, if the stochastic forcing is sufficiently strong, the configuration of the director field can break out of this locally stable configuration and evolve towards a configuration with a lower energy level. Possible configurations of the director field at energy level $16$ include, e.g., the configuration shown on the left-hand side of the second row of Figure \ref{fig:states}. 
An example for possible configuration at energy level $11$ can be seen on the left-hand side of the third row of Figure \ref{fig:states}, while the configurations on energy level $5$ correspond to the global optima depicted in the bottom row of Figure on the left-hand side of the second row of Figure \ref{fig:states}.\\
To understand the impact of the intensity of the noise, we compare the different rows of Figure \ref{fig:paths-by-noise} which correspond to $c_S=0.1,1,2,5$.
As one would expect, a stronger noise term leads to more oscillation of the energy, whereas a smaller noise term increases the realization of long-term stable local minima.
Increasing the noise level enhances the probability of overcoming energy barriers, leading more frequently to transitions between different (meta-)stable states, whereas lower noise intensities tend to trap the dynamics in local minima for a larger fraction of runs.

\begin{figure}
    \centering
    \includegraphics[width=0.45\linewidth]{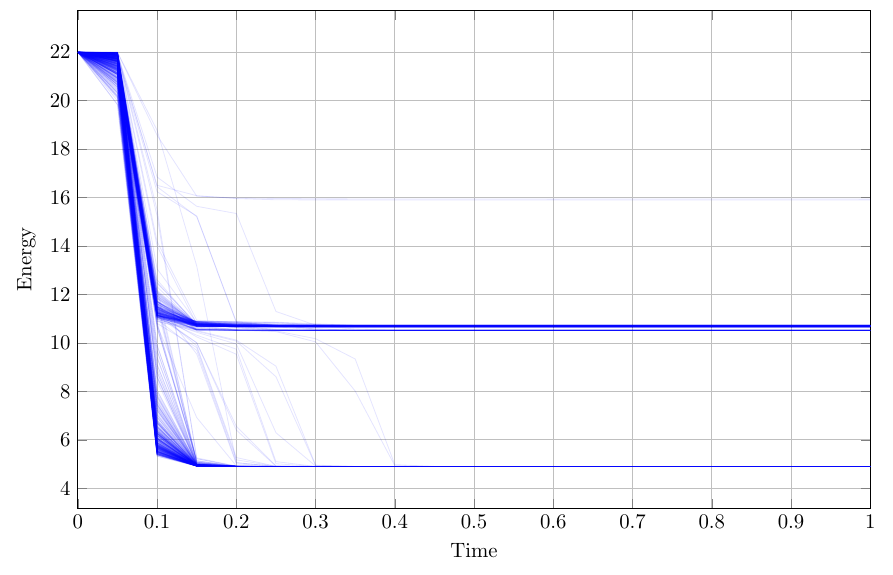}
    \includegraphics[width=0.45\textwidth]{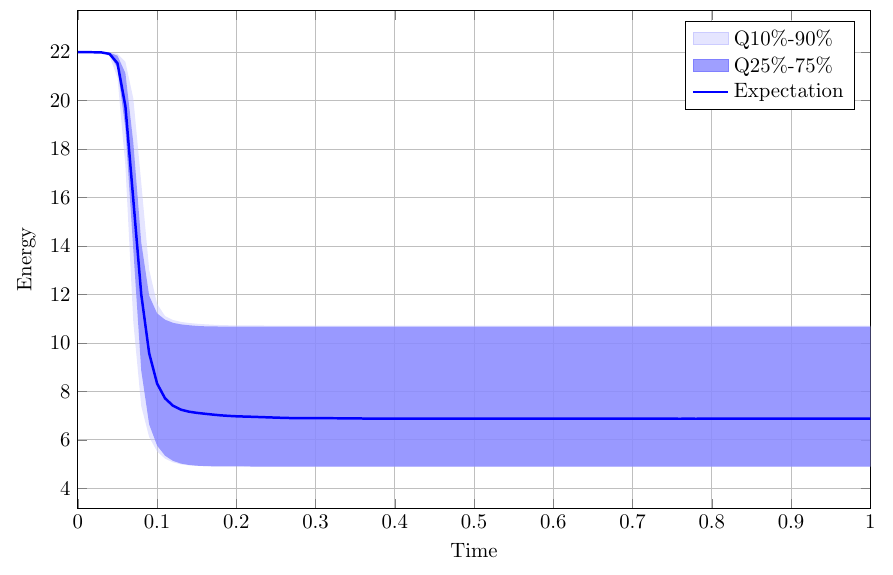}
    \\
    \includegraphics[width=0.45\linewidth]{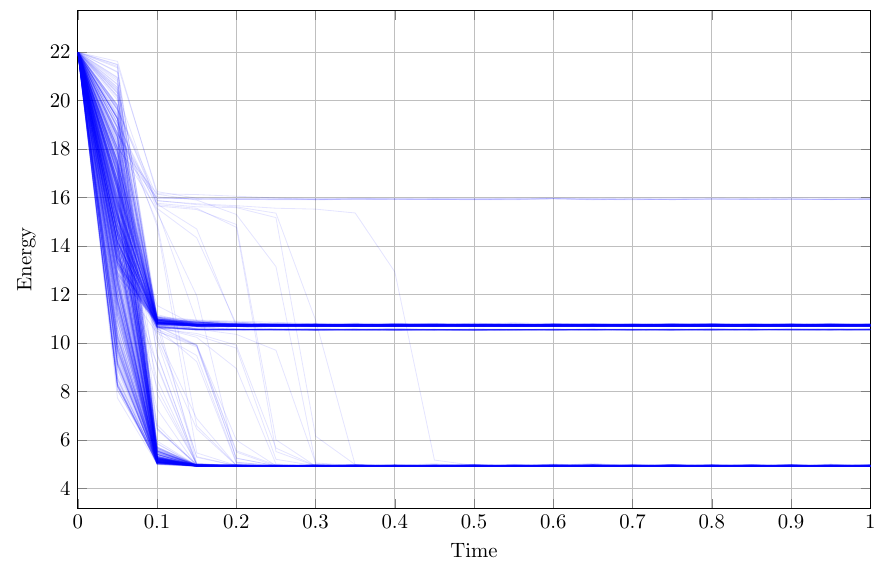}
    \includegraphics[width=0.45\textwidth]{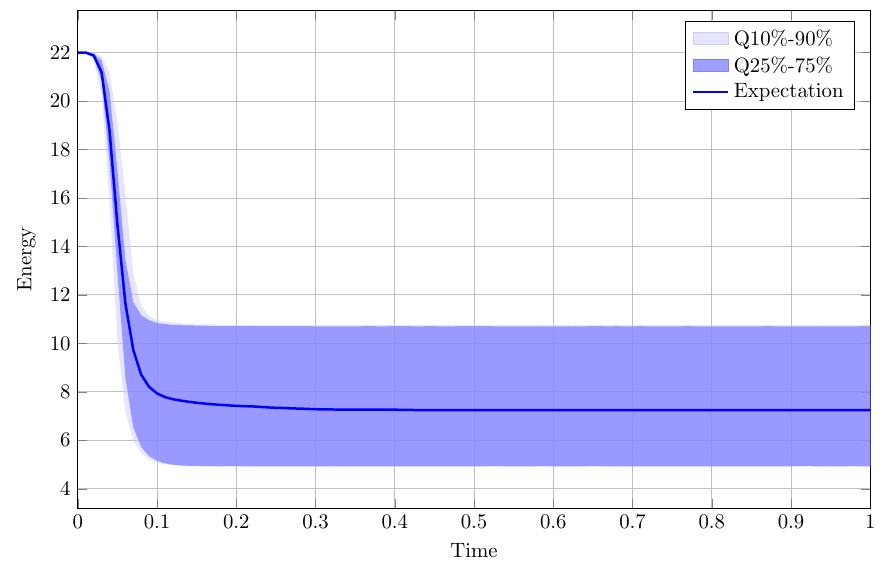}
    \\
    \includegraphics[width=0.45\linewidth]{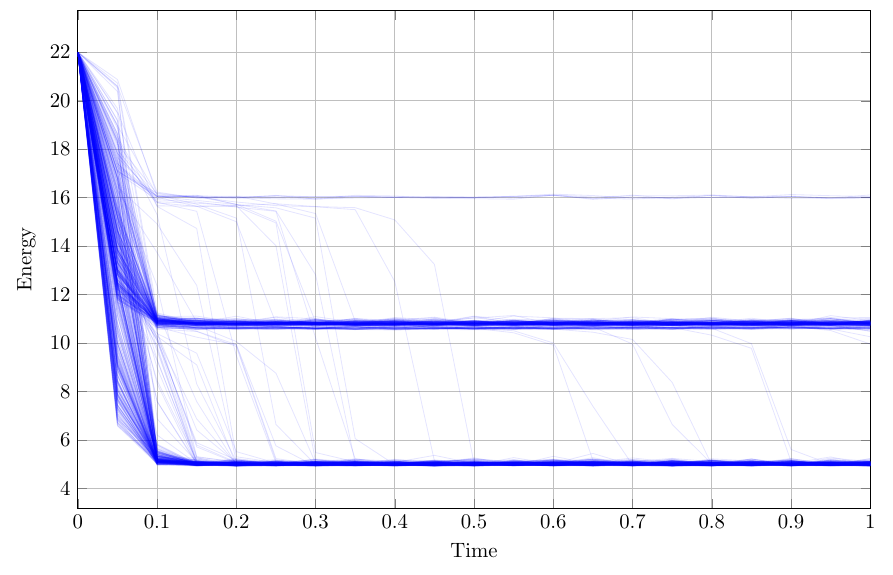}
    \includegraphics[width=0.45\textwidth]{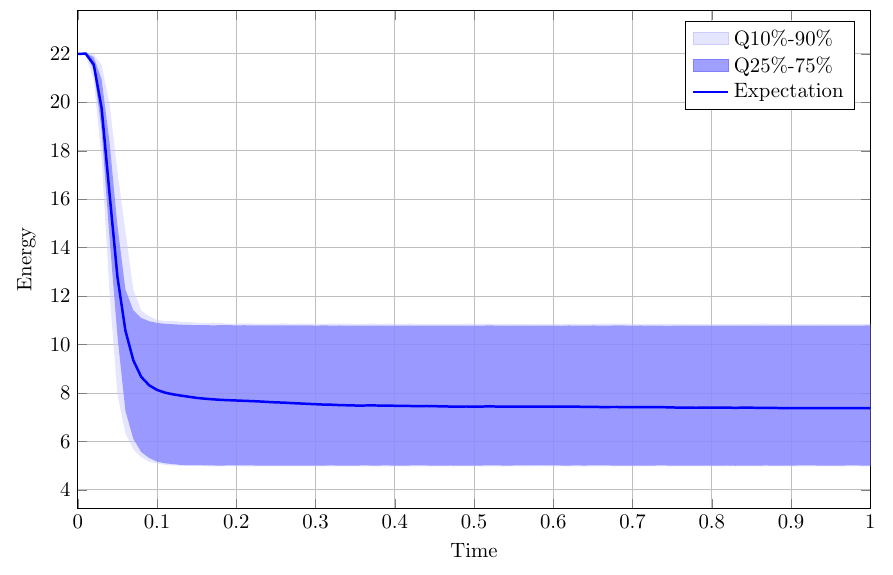}
    \\
    \includegraphics[width=0.45\linewidth]{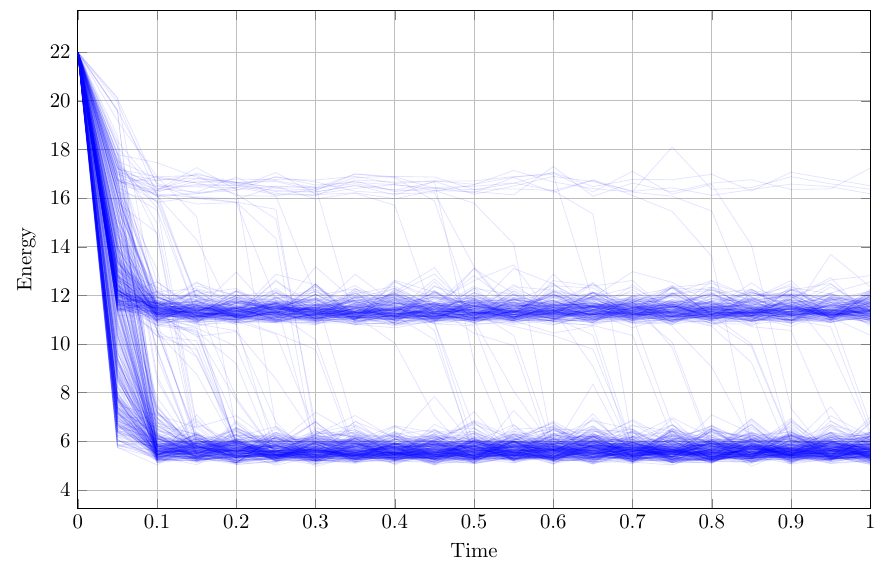}
    \includegraphics[width=0.45\textwidth]{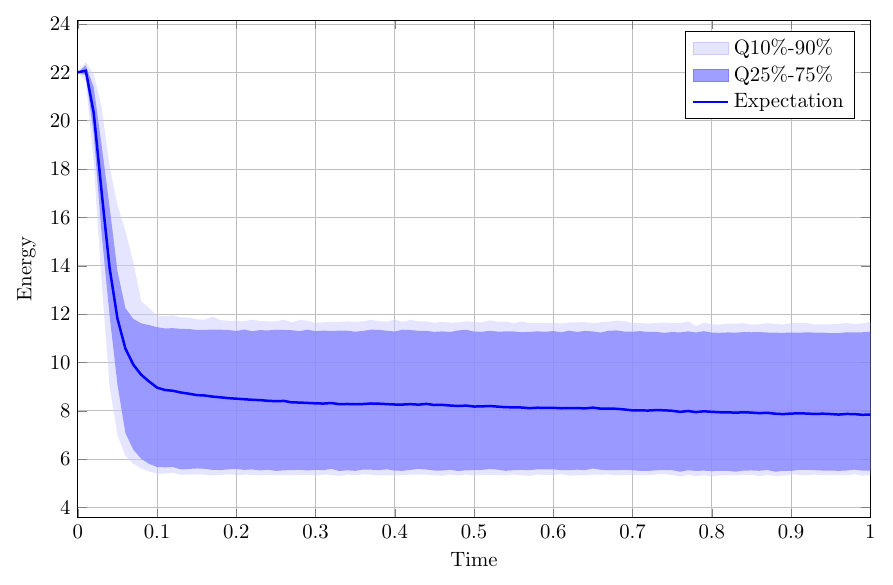}
    \caption{The energy evolution for different stochastic realizations and the associated distribution for different noise intensities $c_S = 0.1, \, 1, \, 2 , \, 5$ (from top to bottom).}
    \label{fig:paths-by-noise}
\end{figure}

These experiments show the numerical applicability of our proposed scheme. 
Furthermore, the presented simulations demonstrate the effect of the stochastic noise term which leads to physically more relevant solutions than the deterministic models, as it allows the director field to escape local optima of the energy landscape.

\section*{Acknowledgment}
Robert Lasarzik acknowledges funding by the Deutsche Forschungsgemeinschaft (DFG, German Research
Foundation) under Germany´s Excellence Strategy – The Berlin Mathematics
Research Center MATH+ (EXC-2046/1, EXC-2046/2, project ID: 390685689) within the project AA Tech 1.

%% file: stochasticEL.bbl
\providecommand{\bysame}{\leavevmode\hbox to3em{\hrulefill}\thinspace}
\providecommand{\MR}{\relax\ifhmode\unskip\space\fi MR }
\providecommand{\MRhref}[2]{%
  \href{http://www.ams.org/mathscinet-getitem?mr=#1}{#2}
}
\providecommand{\href}[2]{#2}